\documentclass[aop,preprint]{imsart}

\usepackage{amsthm,amsmath}
\usepackage[square,sort,comma,numbers]{natbib} 
\RequirePackage[colorlinks,citecolor=blue,urlcolor=blue]{hyperref}
\RequirePackage{natbib}

\usepackage{amssymb,bbm,enumerate,verbatim}
\usepackage{mathtools}
\usepackage{mathdots}

\arxiv{arXiv:1608.07347}

\startlocaldefs

\theoremstyle{plain}
 \newtheorem{theorem}{Theorem}[section]
 \newtheorem{lemma}[theorem]{Lemma}
 \newtheorem{proposition}[theorem]{Proposition}
 \newtheorem{corollary}[theorem]{Corollary}
 \newtheorem{conjecture}[theorem]{Conjecture}
 
 \newtheorem{observation}[theorem]{Observation}

\theoremstyle{definition}
 \newtheorem{definition}[theorem]{Definition}

\theoremstyle{remark}
 \newtheorem{remark}[theorem]{Remark}

\newcounter{question}[section]
\newenvironment{question}[1][]{\refstepcounter{question}\par\medskip
   \noindent\textbf{Question~\thequestion. #1} \rmfamily}{\medskip}
\renewcommand*{\thequestion}{\Alph{question}}

\numberwithin{equation}{section}
\numberwithin{theorem}{section}

\newcommand\nc\newcommand
\newcommand\dmo\DeclareMathOperator

\nc{\Red}[1]{\textcolor[rgb]{0.8,0,0}{#1}}

\def \re {\mathrm{Re}}
\def \im {\mathrm{Im}}
\dmo{\sign}{sign}
\dmo{\spt}{spt}
\dmo{\supp}{supp}
\dmo{\sym}{Sym}
\nc{\R}{\mathbb{R}}
\nc{\C}{\mathbb{C}}
\nc{\N}{\mathbb{N}}
\nc{\Z}{\mathbb{Z}}
\nc{\erdos}{Erd\H os }
\nc{\er}{Erd\H os--R\'enyi } 
\dmo{\ls}{\lesssim}
\dmo{\gs}{\gtrsim}
\def \<{\langle}
\def \>{\rangle}

\def \lf {\lfloor}
\def \rf {\rfloor}
\nc{\expo}[1]{\exp \left( #1 \rule{0mm}{-10mm}\right)}
\DeclarePairedDelimiter\parentheses{\lparen}{\rparen}

\nc{\dd}{\mathrm{d}}

\dmo{\e}{\mathbb{E}}
\dmo{\var}{Var}
\dmo{\pr}{\mathbb{P}}
\dmo{\un}{\mathbbm{1}}
\dmo{\1}{\mathbf{1}}
\nc{\eqd}{\,{\buildrel d \over =}\,}

\nc{\bad}{\mathcal{B}}
\nc{\event}{\mathcal{E}}
\nc{\good}{\mathcal{G}}
\nc{\pro}[1]{\mathbb{P}\parentheses*{#1 \rule{0mm}{0mm}}}
\nc{\set}[1]{\left\{ #1 \right\}}

\dmo{\tr}{tr}
\dmo{\rank}{rank}
\dmo{\rk}{Rank}
\dmo{\corank}{corank}
\def \tran {\mathsf{T}}
\def \HS {\mathrm{HS}}
\nc{\Span}{\operatorname{span}}
\dmo{\per}{\text{per}}
\dmo{\I}{\mathrm{I}}
\dmo{\diag}{diag}
\dmo{\id}{I}

\nc{\eps}{\varepsilon}
\nc{\ep}{\epsilon}

\nc{\mA}{\mathcal{A}}
\nc{\mB}{\mathcal{B}}
\nc{\mC}{\mathcal{C}}
\nc{\mD}{\mathcal{D}}
\nc{\mE}{\mathcal{E}}
\nc{\mF}{\mathcal{F}}
\nc{\mG}{\mathcal{G}}
\nc{\mH}{\mathcal{H}}
\nc{\mI}{\mathcal{I}}
\nc{\mJ}{\mathcal{J}}
\nc{\mK}{\mathcal{K}}
\nc{\mL}{\mathcal{L}}
\nc{\mM}{\mathcal{M}}
\nc{\mN}{\mathcal{N}}
\nc{\mO}{\mathcal{O}}
\nc{\mP}{\mathcal{P}}
\nc{\mQ}{\mathcal{Q}}
\nc{\mR}{\mathcal{R}}
\nc{\mS}{\mathcal{S}}
\nc{\mT}{\mathcal{T}}
\nc{\mU}{\mathcal{U}}
\nc{\mV}{\mathcal{V}}
\nc{\mW}{\mathcal{W}}
\nc{\mX}{\mathcal{X}}
\nc{\mY}{\mathcal{Y}}
\nc{\mZ}{\mathcal{Z}}

\dmo{\Sparse}{Sparse}
\dmo{\Comp}{Comp}
\dmo{\Incomp}{Incomp}
\dmo{\dist}{dist}
\nc{\ha}{\sigma_0}
\nc{\tM}{\widetilde{M}}
\nc{\tA}{\widetilde{A}}
\nc{\tY}{\widetilde{Y}}
\nc{\wY}{\widetilde{Y}}
\nc{\sig}{a}
\nc{\be}{b}
\nc{\M}{M}
\dmo{\MBP}{MBP}
\dmo{\pp}{\mathbf{p}}
\nc{\asp}{\theta}
\dmo{\badd}{bad}
\dmo{\nil}{nil}
\dmo{\cyc}{cyc}
\dmo{\free}{free}
\nc{\me}{m}

\endlocaldefs

\begin{document}

\begin{frontmatter}

\title{Lower bounds for the smallest singular value of structured random matrices}
\runtitle{Invertibility of structured random matrices}

\author{\fnms{Nicholas} \snm{Cook}\ead[label=e1]{nickcook@math.ucla.edu}\thanksref{t1}}
 \thankstext{t1}{Partially supported by NSF postdoctoral fellowship DMS-1266164.} 
\address{Department of Mathematics\\ 
University of California\\
Los Angeles, CA 90095-1555\\ 
\printead{e1}}
\affiliation{University of California, Los Angeles}

\runauthor{N. Cook}

\begin{abstract}
We obtain lower tail estimates for the smallest singular value of random matrices with independent but non-identically distributed entries.
Specifically, we consider $n\times n$ matrices with complex entries of the form
\[
M = A\circ X + B = (a_{ij}\xi_{ij} + b_{ij})
\]
where $X=(\xi_{ij})$ has iid centered entries of unit variance and $A$ and $B$ are fixed matrices.
In our main result we obtain polynomial bounds on the smallest singular value of $M$ for the case that $A$ has bounded (possibly zero) entries, and $B= Z\sqrt{n}$ where $Z$ is a diagonal matrix with entries bounded away from zero. 
As a byproduct of our methods we can also handle general perturbations $B$ under additional hypotheses on $A$, which translate to connectivity hypotheses on an associated graph. 
In particular, we extend a result of Rudelson and Zeitouni for Gaussian matrices to allow for general entry distributions satisfying some moment hypotheses.
Our proofs make use of tools which (to our knowledge) were previously unexploited in random matrix theory, in particular Szemer\'edi's Regularity Lemma, and a version of the Restricted Invertibility Theorem due to Spielman and Srivastava.
\end{abstract}



\end{frontmatter}



\section{Introduction}

Throughout the article we make use of the following standard asymptotic notation:
$f=O(g)$, $f\ll g$, $g\gg f$ all mean that $|f|\le Cg$ for some absolute constant $C<\infty$. 
We indicate dependence of the implied constant on parameters with subscripts, e.g.\ $f\ll_\alpha g$.
$C,c, c', c_0$, etc.\ denote unspecified constants whose value may be different at each occurence, and are understood to be absolute if no dependence on parameters is indicated.

\subsection{Background}

Recall that the singular values of an $n\times n$ matrix $M$ with complex entries are the eigenvalues of $\sqrt{M^*M}$, which we arrange in non-increasing order:
\[
\|M\|=s_1(M)\ge \cdots\ge s_n(M)\ge0.
\]
(throughout we write $\|\cdot \|$ for the $\ell_2^n\to \ell_2^n$ operator norm). 
$M$ is invertible if and only if $s_n(M)>0$, in which case $s_n(M)= \|M^{-1}\|^{-1}$.
We (informally) say that $M$ is ``well-invertible" if $s_n(M)$ is well-separated from zero. 

The largest and smallest singular values of random matrices with independent entries have been intensely studied, in part due to applications in theoretical computer science.
Motivated by their work on the first electronic computers, von Neumann and Goldstine sought upper bounds on the condition number $\kappa(M)= s_1(M)/s_n(M) $ of a large matrix $M$ with iid entries \citep{vNG:numerical}. 
More recently, bounds on the condition number of non-centered random matrices have been important in the theory of \emph{smoothed analysis of algorithms} developed by Spielman and Teng \citep{SST:smoothed}.
The smallest singular value 
has also received attention due to its connection with proving convergence of the empirical spectral distribution -- see 
\citep{TaVu:circ, BoCh:survey}.

Much is known about the largest singular value for random matrices with independent entries.
First we review the iid case: we denote by $X=X_n$ an $n\times n$ matrix whose entries $\xi_{ij}$ are iid copies of a centered complex random variable with unit variance, and refer to such $X$ as an ``iid matrix".
From the works \citep{BSY:largestsv, BKY:largestsv} it is known that $\frac{1}{\sqrt{n}}s_1(X_n)\in (2-\eps,2+\eps)$ with probability tending to one as $n\to \infty$ for any fixed $\eps>0$.
In connection with problems in computer science and the theory of Banach spaces there has been considerable interest in obtaining non-asymptotic bounds for matrices with independent but non-identically distributed entries; see the recent works \citep{BaHa} and \citep{vanHandel:norm} and references therein for an overview. 

The picture is far less complete for the smallest singular value of random matrices; however, recent years have seen much progress for the case of the iid matrix $X$.
The limiting distribution of $\sqrt{n}s_n(X)$ was obtained by Edelman for the case of Gaussian entries \citep{Edelman:condition}, and this law was shown by Tao and Vu to hold for all iid matrices with entries $\xi_{ij}$ having a sufficiently large finite moment \citep{TaVu:ssv}.

Quantitative lower tail estimates for $s_n(X)$ proved to be considerably more challenging than bounding the operator norm.
The first breakthrough was made by Rudelson \citep{Rudelson:inv}, who showed that if $X$ has iid real-valued sub-Gaussian entries, that is
\begin{equation}	\label{subgaussian}
\e \expo{ |\xi|^2/K_0}\le 2
\end{equation}
for some $K_0<\infty$,
then
\begin{equation}	\label{bd:rudelson}
\pro{ s_n(X) \le tn^{-3/2}} \ll_{K_0} t + n^{-1/2}\qquad \text{for all $t\ge 0$}.
\end{equation}

Around the same time, in \cite{TaVu:cond} Tao and Vu used methods from additive combinatorics to obtain bounds of the form
\begin{equation}	\label{tavu:arb}
\pro{ s_n(X) \le n^{-\beta}} \ll n^{-\alpha}
\end{equation}
for any fixed $\alpha>0$ and $\beta$ sufficiently large depending on $\alpha$, for the case that the entries of $X$ take values in $\{-1,0,1\}$. 
Roughly speaking, their approach was to classify potential almost-null vectors $v$ according to the amount of additive structure present in the multi-set of coordinate values $\{v_j\}_{j=1}^n$. 
They extended \eqref{tavu:arb} to uncentered matrices with general entry distributions having finite second moment in \cite{TaVu:circ} (see Theorem \ref{thm:tavu} below), which was instrumental for their proof of 
the celebrated circular law for the limiting spectral distribution of $\frac{1}{\sqrt{n}}X$. 

Motivated by these developments, in \cite{RuVe:ilo} Rudelson and Vershynin found a different way to quantify the additive structure of a vector $v$ called the \emph{essential least common denominator}, and obtained the following improvement of \eqref{bd:rudelson}, \eqref{tavu:arb} for matrices with sub-Gaussian entries:
\begin{equation}	\label{bd:ruve_exp}
\pro{ s_n(X) \le tn^{-1/2}} \ll_{K_0} t + e^{-cn}.
\end{equation}
This estimate is optimal up to the implied constant and $c=c(K_0)>0$ (with $K_0$ as in \eqref{subgaussian}). 

Finally, we mention that there has also been work on \emph{upper tail bounds} for the smallest singular value -- see in particular \cite{RuVe:uppertail, NgVu:normal} -- but we do not consider this problem further in the present work.

\subsection{A general class of non-iid matrices}

In this paper we are concerned with bounds for the smallest singular value of random matrices with independent but non-identically distributed entries.
The following definition allows us to quantify the dependence of our bounds on the distribution of the matrix entries.

\begin{definition}[Spread random variable]		\label{def:spread}
Let $\xi$ be a complex random variable and let $\kappa\ge 1$. 
We say that $\xi$ is \emph{$\kappa$-spread} if 
\begin{equation}	\label{def:kappa0}
\var \big[\,\xi\un(|\xi-\e\xi|\le \kappa)\,\big] \ge \frac1\kappa.
\end{equation}
\end{definition}

\begin{remark}	\label{rmk:kappap}
It follows from the monotone convergence theorem that any random variable $\xi$ with non-zero second moment is $\kappa$-spread for some $\kappa<\infty$.
Furthermore, if $\xi$ is centered with unit variance and finite $p$th moment $\mu_p$ for some $p>2$, then it is routine to verify that $\xi$ is $\kappa$-spread with $\kappa =3 (3\mu_p^p)^{1/(p-2)}$, say.
\end{remark}

Our results concern the following general class of matrices:

\begin{definition}[Structured random matrix]	\label{def:profile}
Let $A=(\sig_{ij})$ and $B=(\be_{ij})$ be deterministic $n\times m$ matrices with $\sig_{ij}\in [0,1]$ and $\be_{ij}\in \C$ for all $i,j$.
Let $X=(\xi_{ij})$ be an $n\times m$ matrix with independent entries, all identically distributed to a complex random variable $\xi$ with mean zero and variance one.
Put 
\begin{equation}	\label{Mdef}
\M=A\circ X+B = (\sig_{ij}\xi_{ij}+ \be_{ij})_{i,j=1}^n
\end{equation}
where $\circ$ denotes the matrix Hadamard product.
We refer to $A$, $B$ and $\xi$ as the \emph{standard deviation profile}, \emph{mean profile} and \emph{atom variable}, respectively.
We denote the $L^p$ norm of the atom variable by
\begin{equation}	\label{Lp}
\mu_p:= (\e|\xi|^p)^{1/p}.
\end{equation}
Without loss of generality, we assume throughout that $\xi$ is $\kappa_0$-spread for some fixed $\kappa_0\ge1$.
\end{definition}

(While all of our results are for square matrices, we give the definition for the general rectangular case as we will often need to consider rectangular submatrices in the proofs.)

\begin{remark}	\label{rmk:shiftscale}
The assumption that the entries of $\M$ are shifted scalings of random variables $\xi_{ij}$ having a common distribution is made for convenience, as it allows us to access some standard anti-concentration estimates (see Section \ref{sec:fourier}).
We expect the proofs can be modified to cover general matrices with independent entries having specified means and variances (possibly with additional moment hypotheses), but we do not pursue this here.
\end{remark}

As a concrete example one can consider a centered non-Hermitian band matrix, where one sets $a_{ij}\equiv 0$ for $|i-j|$ exceeding some bandwidth parameter $w\in [n-1]$ -- see Corollary \ref{cor:band}.

The singular value distributions for structured random matrices have been studied in connection with wireless MIMO networks \citep{TuVe:rmt_wireless,HLN:detequiv}.
The limiting spectral distributions and spectral radius for certain structured random matrices have been used to model the dynamical properties of neural networks \citep{RaAb:neural, ARS:block}.
In the recent work \citep{CHNR} with Hachem, Najim and Renfrew, the limiting spectral distribution was determined for a general class of centered structured random matrices.
That work required bounds on the smallest singular value for shifts of centered matrices by scalar multiples of the identity, which was the original motivation for the results in this paper (in particular, Corollary \ref{cor:scalarshift} below is a key input for the proofs in \citep{CHNR}). 

The picture for the smallest singular value of structured random matrices is far less complete than for the largest singular value. 
Here we content ourselves with identifying sufficient conditions on the matrices $A,B$ and the distribution of $\xi$ for a structured random matrix $M$ to be well-invertible with high probability. Specifically, we seek to address the following:

\begin{question}		\label{prob:main}
Let $\M$ be an $n\times n$ random matrix as in Definition \ref{def:profile}.
Under what assumptions on the standard deviation and mean profiles $A,B$ 
and the distribution of the atom variable $\xi$ 
do we have
\begin{equation}	\label{lowertail}
\pr\big( s_n(\M) \le n^{-\beta}\big) = O(n^{-\alpha})
\end{equation}
for some constants $\alpha,\beta>0$?
\end{question}

The case that $B=-z\sqrt{n} I $ for some fixed $z\in \C$ (where $I$ denotes the $n\times n$ identity matrix) is of particular interest for applications to the limiting spectral distribution of centered random matrices. 
As we shall see in the next subsection, existing results in the literature give lower tail bounds for $s_n(M)$ that are uniform in the shift $B$ under the size constraint $\|B\|=n^{O(1)}$, i.e. 
\begin{equation}	\label{lowertail-uniform}
\sup_{B\in \mM_n(\C): \;\|B\| \le  n^{C}} \pr\big( s_n(A\circ X + B) \le n^{-\beta}\,\big) = O(n^{-\alpha}).
\end{equation}
for some constant $C>0$ (results stated for centered matrices generally extend in a routine manner to allow a perturbation of size $\|B\|=O(\sqrt{n})$).
Such bounds can be viewed as matrix analogues of classical anti-concentration (or ``small ball") bounds of the form
\begin{equation}	\label{anticonc:general}
\sup_{z\in \C} \pro{ |S_n - z|\le r} \le f(r) + o(1)
\end{equation}
for a sequence of scalar random variables $S_n$ (such as the normalized partial sums of an infinite sequence of iid variables), where $f:\R_+\to \R_+$ is some continuous function such that $f(r)\to 0$ as $r\to 0$.
In fact, bounds of the form \eqref{anticonc:general} are a central ingredient in the proofs of estimates \eqref{lowertail-uniform}. 
Roughly speaking, the translation invariance of \eqref{anticonc:general} causes the uniformity in the shift $B$ in \eqref{lowertail-uniform} to come for free once one can handle the centered case $B=0$ (the assumption $\|B\|=n^{O(1)}$ is needed to have some continuity of the map $u\mapsto \|Mu\|$ on the unit sphere in order to apply a discretization argument). 
In light of this we may pose the following:

\begin{question}	\label{prob:sub}
Let $\M$ be an $n\times n$ random matrix as in Definition \ref{def:profile}, and let $\gamma>0$.
Under what assumptions on the standard deviation profile $A$ 
and the distribution of the atom variable $\xi$ do we have
\begin{equation}	\label{lowertail:sub}
\sup_{B\in \mM_n(\C): \; \|B\|\le n^\gamma}\pr\big( s_n(\M) \le n^{-\beta}\,\big) = O(n^{-\alpha})
\end{equation}
for some constants $\alpha,\beta>0$?
\end{question}

The following simple observation puts a clear limitation on the standard deviation profiles $A$ for which we can expect to have \eqref{lowertail:sub}.

\begin{observation}	\label{obs:submatrix}
Suppose that $A=(\sig_{ij})$ has a $k\times m$ submatrix of zeros for some $k,m$ with $k+m>n$. 
Then $A\circ X$ is singular with probability 1.
Thus, \eqref{lowertail:sub} fails (by taking $B=0$) for any fixed $\alpha,\beta>0$.
\end{observation}

Theorem \ref{thm:broad} below (see also Theorem \ref{thm:ruze} for the Gaussian case) shows that the above is in some sense the only obstruction to obtaining \eqref{lowertail:sub}.

\subsection{Previous results}	\label{sec:previous}

Before stating our main results on Questions \ref{prob:main} and \ref{prob:sub} we give an overview of what is currently in the literature.

For the case of a constant standard deviation profile $A$ and essentially arbitrary mean profile $B$ we have the following result of Tao and Vu: 

\begin{theorem}[Shifted iid matrix \citep{TaVu:circ}]		\label{thm:tavu}
Let $X$ be an $n\times n$ matrix with iid entries $\xi_{ij}\in \C$ having mean zero and variance one. 
For any $\alpha,\gamma>0$ there exists $\beta>0$ such that for any fixed (deterministic) $n\times n$ matrix $B$ with $\|B\|\le n^\gamma$,
\begin{equation}	\label{tavu:bound}
\pr\big( s_n(X+B) \le n^{-\beta} \,\big) = O_{\alpha,\gamma}(n^{-\alpha}).
\end{equation}
\end{theorem}

A stronger version of the above bound was established earlier by Sankar, Spielman and Teng for the case that $X$ has iid standard Gaussian entries \citep{SST:smoothed}.
For the case that $B=0$, the bound \eqref{bd:ruve_exp} of Rudelson and Vershynin gives the optimal dependence $\beta = \alpha +1/2$ for the exponents, but requires the stronger assumption that the entries are real-valued and sub-Gaussian (we remark that their proof extends in a routine manner to allow an arbitrary shift $B$ with $\|B\|=O(\sqrt{n})$). 
Recently, the sub-Gaussian assumption for \eqref{bd:ruve_exp} was relaxed by Rebrova and Tikhomirov to only assume a finite second moment \citep{ReTi}.

When the entries of $M$ have bounded density the problem is much simpler.
The following is easily obtained by the argument in \cite[Section 4.4]{BoCh:survey}.

\begin{proposition}[Matrix with entries having bounded density \citep{BoCh:survey}]	\label{prop:bdd}
Let $\M$ be an $n\times n$ random matrix with independent entries having density on $\C$ or $\R$ uniformly bounded by $\varphi> 1$.
For every $\alpha>0$ there is a $\beta=\beta(\alpha,\varphi)>0$ such that
\begin{equation}
\pr\big( s_n(\M) \le n^{-\beta}\,\big) = O(n^{-\alpha}).
\end{equation}
\end{proposition}

Note that above we make no assumptions on the moments of the entries of $\M$ -- in particular, they may have heavy tails. 
The following result of Bordenave and Chafa\"i (Lemma A.1 in \citep{BoCh:survey}) relaxes the hypothesis of continuous distributions from Proposition \ref{prop:bdd} while still allowing for heavy tails, but comes at the cost of a worse probability bound.

\begin{proposition}[Heavy-tailed matrix with non-degenerate entries {\citep{BoCh:survey}}]		\label{prop:BoCh}
Let $Y$ be an $n\times n$ random matrix with independent entries $\eta_{ij}\in \C$.
Suppose that for some $p,r,\ha>0$ we have that for all $i,j\in [n]$,
\begin{equation}	\label{BoCh:nondeg}
\pro{ |\eta_{ij}|\le r} \ge p, \quad\quad \var( \eta_{ij}\un(|\eta_{ij}|\le r)) \ge \ha^2.
\end{equation}
For any $s\ge1$, $t\ge 0$, and any fixed $n\times n$ matrix $B$ we have
\begin{equation}
\pro{ s_n(Y+B) \le \frac{t}{\sqrt{n}}, \; \|Y+B\| \le s} \ll_{p,r,\ha} \sqrt{\log s}\bigg(ts^2 + \frac{1}{\sqrt{n}}\bigg).
\end{equation}
\end{proposition}

The non-degeneracy conditions \eqref{BoCh:nondeg} do not allow for some entries to be deterministic.
Litvak and Rivasplata \citep{LiRi} obtained a lower tail estimate of the form \eqref{lowertail} for centered random matrices having a sufficiently small constant proportion of entries equal to zero deterministically.
Below we give new results (Theorems \ref{thm:broad} and \ref{thm:super}) allowing all but an arbitrarily small (fixed) proportion of entries to be deterministic. 

Finally, we recall a theorem of Rudelson and Zeitouni \citep{RuZe} for Gaussian matrices, showing that Observation \ref{obs:submatrix} is essentially the only obstruction to obtaining \eqref{lowertail:sub}.
To state their result we need to set up some graph theoretic notation, which will be used repeatedly throughout the paper.

To a non-negative $n\times m$ matrix $A=(\sig_{ij})$ we associate a bipartite graph $\Gamma_A = ([n],[m], E_A)$, with $(i,j)\in E_A$ if and only if $\sig_{ij}>0$. 
For a row index $i\in [n]$ we denote by
\begin{equation}	\label{def:nbhd}
\mN_A(i) = \set{j\in [m]: \sig_{ij}>0}
\end{equation}
its neighborhood in $\Gamma_A$. 
Thus, the neighborhood of a column index $j\in [m]$ is denoted $\mN_{A^\tran}(j)$.
Given sets of row and column indices $I\subset [n], J\subset[m]$, we define the associated \emph{edge count}
\begin{equation}	\label{def:edges}
e_A(I,J) := |\{(i,j)\in [n]\times [m]: \sig_{ij}>0\}|.
\end{equation}
We will generally work with the graph that only puts an edge $(i,j)$ when $\sig_{ij}$ exceeds some fixed cutoff parameter $\ha>0$.
Thus, we denote by
\begin{equation}	\label{Aha}
A(\ha) = (\sig_{ij}1_{\sig_{ij}\ge \ha})
\end{equation}
the matrix which thresholds out entries smaller than $\ha$.

Rudelson and Zeitouni work with Gaussian matrices whose matrix of standard deviations $A=(\sig_{ij})$ satisfies the following expansion-type condition.

\begin{definition}[Broad connectivity]		\label{def:broad}
Let $A=(\sig_{ij})$ be an $n\times m$ matrix with non-negative entries. 
For $I\subset [n]$ and $\delta\in (0,1)$, define the set of \emph{$\delta$-broadly connected} neighbors of $I$ as
\begin{equation}	\label{def:broadnbr}
\mN_{A}^{(\delta)}(I)= \{ j\in [m]: |\mN_{A^\tran}(j)\cap I|\ge \delta |I|\}.
\end{equation}
For $\delta,\nu\in (0,1)$, we say that $A$ is \emph{$(\delta,\nu)$-broadly connected} if 
\begin{enumerate}[(1)]
\item $|\mN_A(i)| \ge \delta m$ for all $i\in [n]$;\vspace{.2cm}
\item $|\mN_{A^\tran}(j)|\ge \delta n$ for all $j\in [m]$;\vspace{.2cm}
\item $|\mN_{A^\tran}^{(\delta)}(J)| \ge \min(n,(1+\nu)|J|)$ for all $J\subset [m]$.
\end{enumerate}
\end{definition}

\begin{theorem}[Gaussian matrix with broadly connected profile \citep{RuZe}]	\label{thm:ruze}
Let $G$ be an $n\times n$ matrix with iid standard real Gaussian entries, and let $A$ be an $n\times n$ matrix with entries $\sig_{ij}\in [0,1]$ for all $i,j$. 
With notation as in \eqref{Aha}, assume that $A(\ha)$ is $(\delta,\nu)$-broadly connected for some $\ha,\delta,\nu\in (0,1)$.
Let $K\ge1$, and let $B$ be a fixed $n\times n$ matrix with $\|B\|\le K\sqrt{n}$. 
Then for any $t\ge0$,
\begin{equation}	\label{bound:ruze}
\pr\big( s_n(A\circ G+ B) \le tn^{-1/2}\,\big) \ll_{\delta,\nu,\ha} K^{O(1)}t+ e^{-cn}
\end{equation}
for some $c=c(\delta,\nu,\ha)>0$.
\end{theorem}

Note that the assumption of broad connectivity gives us an ``epsilon of separation" from the bad example of Observation \ref{obs:submatrix}.
Thus, Theorem \ref{thm:ruze} provides a near-optimal answer to Question \ref{prob:sub} for Gaussian matrices.

\begin{remark}		\label{rmk:ruze_dense}
Since the dependence of the bound \eqref{bound:ruze} on the parameters $\delta$ and $\nu$ is not quantified, Theorem \ref{thm:ruze} only addresses Question \ref{prob:sub} for \emph{dense} standard deviation profiles, i.e. when $A$ has a non-vanishing proportion of large entries. 
While it would not be difficult to quantify the steps in \citep{RuZe}, the resulting dependence on parameters is not likely to be optimal.
\end{remark}

\subsection{New results}	\label{sec:results}

Our first result removes the Gaussian assumption from Theorem \ref{thm:ruze}, though at the cost of a worse probability bound.
Recall the parameter $\kappa_0$ from Definition \ref{def:profile}.

\begin{theorem}[General matrix with broadly connected profile]		\label{thm:broad}
Let $\M=A\circ X+B$ be an $n\times n$ matrix as in Definition \ref{def:profile}, and assume that $A(\ha)$ is $(\delta,\nu)$-broadly connected for some $\ha,\delta,\nu\in (0,1)$.
Let $K\ge 1$. 
For any $t\ge 0$,
\begin{equation}	\label{broad:poly}
\pro{ s_n(\M) \le \frac{t}{\sqrt{n}}, \, \|\M\| \le K\sqrt{n} } \ll_{K,\delta,\nu,\ha,\kappa_0} t + \frac{1}{\sqrt{n}}.
\end{equation}
\end{theorem}

\begin{remark}	\label{rmk:distribution_broad}
While we have stated no moment assumptions on the atom variable $\xi$ over the standing assumption of unit variance, the restriction to the event $\{\|M\|\le K\sqrt{n}\}$ requires us to assume at least four finite moments to deduce $\pr(s_n(M) \le t/\sqrt{n}) \ll t + o(1)$. 
Here we give a lower tail estimate at the optimal scale $s_n(M)\sim n^{-1/2}$; however, the arguments in this paper can be used to establish a polynomial lower bound on $s_n(M)$ of non-optimal order under larger perturbations $B$ (similar to \eqref{bd:super} below). 
\end{remark}

\begin{remark}[Improving the probability bound]
We expect that the probability bound in \eqref{broad:poly} can be improved by making use of more advanced tools of Littlewood--Offord theory introduced in \citep{TaVu:circ, RuVe:ilo}, 
though it appears these tools cannot be applied in a straightforward manner.
In the interest of keeping the paper of reasonable length we do not pursue this here.
\end{remark}

\begin{remark}[Bounds on moderately small singular values]
The methods used to prove Theorem \ref{thm:broad} together with an idea of Tao and Vu from \citep{TaVu:esd} can be used to give lower bounds of optimal order on $s_{n-k}(M)$ with $n^{\eps}\le k \le cn$ for any $\eps>0$ and a sufficiently small constant $c=c(\kappa_0,\ha,\delta,\nu,K)>0$; see \citep[Theorem 4.5.1]{Cook:thesis}. 
Such bounds are of interest for proving convergence of the empirical spectral distribution; see \citep{TaVu:esd, BoCh:survey}.
\end{remark}

In light of Observation \ref{obs:submatrix}, Theorem \ref{thm:broad} gives an essentially optimal answer to Question \ref{prob:sub} for \emph{dense} random matrices (see Remark \ref{rmk:ruze_dense}). 
It would be interesting to establish a version of this result that allows for only a proportion $o(1)$ of the entries to be random.
Indeed, we expect a version of the above theorem to hold when $A$ has density as small $(\log^{O(1)}n)/n$. (Quantifying the dependence on $\delta,\nu$ in \eqref{broad:poly} would only allow a slight polynomial decay in the density.)

We note that they broad connectivity hypothesis includes many standard deviation profiles of interest, such as band matrices:

\begin{corollary}[Shifted non-Hermitian band matrices]	\label{cor:band}
Let $M=A\circ X+ B$ be an $n\times n$ matrix as in Definition \ref{def:profile}, and assume that for some fixed $\ha,\eps\in (0,1)$, $\sig_{ij} \ge \ha$ for all $i,j$ with $\min(|i-j|,n-|i-j|)\le \eps n$. 
Let $K\ge 1$.
Then \eqref{broad:poly} holds for any $t\ge 0$ (with implied constant depending on $K,\ha,\eps$ and $\kappa_0$).
\end{corollary}

We defer the proof to Appendix \ref{app:band}.

\begin{remark}
It is possible to modify our argument for the above corollary to treat a band profile that does not ``wrap around", i.e.\ only enforcing $a_{ij}\ge \ha$ for $i,j$ with $|i-j|\le \eps n$. 
\end{remark}


%

Having addressed Question \ref{prob:sub}, we now ask whether we can further relax the assumptions on the standard deviation profile $A$ by assuming more about the mean profile $B$. 
In particular, can we make assumptions on $B$ that give \eqref{lowertail} while allowing $A\circ X$ to be singular deterministically?

Of course, a trivial example is to take $A=0$ and $B$ any invertible matrix.
Another easy example is to take take $B$ to be \emph{very well-invertible}, with $s_n(B) \ge K\sqrt{n}$ for a large constant $K>0$ (for instance, take $B=K\sqrt{n}I$, where $I$ is the identity matrix).
Indeed, standard estimates for the operator norm of random matrices with centered entries (cf.\ Section \ref{sec:op}) give $\|A\circ X\| = O(\sqrt{n})$ with high probability provided the atom variable $\xi$ satisfies some additional moment hypotheses. 
From the triangle inequality
\begin{align*}
s_n(M)  &= \inf_{u\in S^{n-1}} \|(A\circ X+B)u\| \ge  s_n(B) - \|A\circ X\|,
\end{align*}
so $s_n(M)\gg \sqrt{n}$ with high probability if $K$ is sufficiently large.

The problem becomes non-trivial when we allow $B$ to have singular values of size $\eps \sqrt{n}$ for small $\eps>0$ and $A$ as in Observation \ref{obs:submatrix}. 
In this case any proof of a lower tail estimate of the form \eqref{lowertail} must depart significantly from the proofs of the results in the previous section by making use of arguments which are not translation invariant.

Our main result shows that when the mean profile $B$ is a \emph{diagonal} matrix with smallest entry at least an arbitrarily small (fixed) multiple of $\sqrt{n}$, then we do not need to assume anything further about the standard deviation profile $A$.

\begin{theorem}[Main result]
\label{thm:main}
Fix arbitrary $r_0\in (0,\frac12]$, $K_0\ge1$, and let $Z$ be a (deterministic) diagonal matrix with diagonal entries $z_1,\dots,z_n\in \C$ satisfying 
\begin{equation}	\label{constraint:zi}
 |z_i|\in [r_0,K_0]\qquad \forall i\in[n].
\end{equation}
Let $M$ be an $n\times n$ random matrix as in Definition \ref{def:profile} with $B=Z\sqrt{n}$, and assume $\mu_{4+\eta}<\infty$ for some fixed $\eta>0$.
There are $\alpha(\eta)>0$ and $\beta(r_0,\eta,\mu_{4+\eta})>0$ such that
\begin{equation}	\label{bd:gen}
\pr\big( s_n(M)\le n^{-\beta} \, \big) = O_{r_0,K_0,\eta,\mu_{4+\eta}}(n^{-\alpha}).
\end{equation}
\end{theorem}

\begin{remark}[Moment assumption]	
The assumption of $4+\eta$ moments is due to our use of a result of Vershynin, Theorem \ref{thm:vershynin} below, on the operator norm of products of random matrices. 
Apart from this, at many points in our argument we use that an $m\times m$ submatrix of $M$ has operator norm $O(\sqrt{m})$ with high probability (assuming $m$ grows with $n$), which requires at least four finite moments. 
Under certain additional assumptions on the standard deviation profile we only need to assume two moments -- see Remark \ref{rmk:relaxmom}.
\end{remark}

\begin{remark}[Dependence of $\alpha,\beta$ on parameters]	\label{rmk:rdep}
The proof gives $\alpha(\eta)=\frac19\min(1,\eta)$. If we were to assume $\xi$ has finite $p$th moment for a sufficiently large constant $p$ then we could take any fixed $\alpha<1/2$ in \eqref{bd:gen}.
The dependence of $\beta$ on 
$\mu_{4+\eta}$
and $r_0$ given by our proof is very bad, of the form
\begin{equation}
\beta = \text{twr}\big(O_\eta(1)\exp((\mu_{4+\eta}/r_0)^{O(1)})\big)
\end{equation}
where $\text{twr}(x)$ is a tower exponential $2^{2^{\iddots^2}}$ of height $x$.
(The factor $O_\eta(1)$ comes from Vershynin's bound mentioned in the previous remark -- we do not know the precise dependence on $\eta$, but we expect it is relatively mild.)
This is due to our use of Szemer\'edi's regularity lemma (specifically, a version for directed graphs due to Alon and Shapira -- see Lemma \ref{lem:regularity}). 
It would be interesting to obtain a version of Theorem \ref{thm:main} with a better dependence of $\beta$ on the parameters. 
\end{remark}

As we remarked above, the case of a diagonal mean profile is of special interest for the problem of proving convergence of the empirical spectral distribution of centered random matrices with a variance profile.

\begin{corollary}[Scalar shift of a centered random matrix]	\label{cor:scalarshift}
Let $X=(\xi_{ij})$ be an $n\times n$ matrix whose entries are iid copies of a centered complex random variable $\xi$ having unit variance and $(4+\eta)$-th moment $\mu_{4+\eta}<\infty$ for some fixed $\eta>0$.
Let $A=(\sig_{ij})$ be a fixed $n\times n$ non-negative matrix with entries uniformly bounded by $\sigma_{\max}<\infty$.
Put $Y= \frac1{\sqrt{n}}A\circ X$, and fix an arbitrary $z\in \C\setminus\{0\}$.
There are constants $\alpha=\alpha(\eta)>0$ and $\beta=\beta(|z|,\eta,\mu_{4+\eta},\sigma_{\max})>0$ such that
\begin{equation}	\label{bd:scalarshift}
\pr\big( s_n(Y-zI)\le n^{-\beta} \, \big) =O_{|z|,\sigma_{\max},\mu_{4+\eta}}(n^{-\alpha}).
\end{equation}
\end{corollary}

While our main motivation was to handle diagonal perturbations of centered random matrices, we conjecture that Theorem \ref{thm:main} extends to matrices as in Definition \ref{def:profile} with more general mean profiles $B$:

\begin{conjecture}	\label{conj:main}
Theorem \ref{thm:main} continues to hold for $B\in \mM_n(\C)$ not necessarily diagonal, where the constraint \eqref{constraint:zi} is replaced with $\frac1{\sqrt{n}}s_i(B)\in [r_0,K_0]$ for all $1\le i\le n$. 
\end{conjecture}

\subsection{Ideas of the proof}		\label{sec:ideas}

Here we give an informal discussion of the main ideas in the proof of Theorem \ref{thm:main}.

\subsubsection*{Regular partitions of graphs}

As with Theorem \ref{thm:broad}, the key is to associate the standard deviation profile $A$ with a graph. 
Since we want the diagonal of $M$ to be preserved under relabeling of vertices will will associate $A$ with a directed graph (digraph) which puts an edge $i\to j$ whenever $a_{ij}$ exceeds some small threshold $\ha>0$.
Since $A$ has no special connectivity structure \emph{a priori}, we will apply a version of Szemer\'edi's regularity lemma for digraphs (Lemma \ref{lem:regularity}) to partition the vertex set $[n]$ into a bounded number of parts of equal size $I_1,\dots, I_m$, together with a small set of ``bad" vertices $I_{\badd}$, such that for most $(k,l)\in [m]^2$ the subgraph on $I_k\cup I_l$ enjoys certain ``pseudorandomness" properties.
These properties will not be quite strong enough to control the smallest singular value of the corresponding submatrix $M_{I_k,I_l}$ of $M$, but we can apply a ``cleaning" procedure (as it is called in the extremal combinatorics literature) to remove a small number of bad vertices from each part in the partition (which we add to $I_{\badd}$), after which we will be able to control $s_{\min}(M_{I_k,I_l})$ for most $(k,l)\in [m]^2$.
We defer the precise formulation of the pseudorandomness properties and corresponding bound on the smallest singular value to Definition \ref{def:super} and Theorem \ref{thm:super} below.

\subsubsection*{Schur complement formula}

The task will then be to lift this control on the invertibility of submatrices to the whole matrix $M$. 
The key tool here is the \emph{Schur complement formula} (see Lemma \ref{lem:schur}) which allows us to control the smallest singular value of a block matrix
\begin{equation}	\label{sketch:partition}
\begin{pmatrix} M_{11} & M_{12} \\ M_{21} & M_{22} \end{pmatrix}
\end{equation}
assuming some control on the smallest singular values of (perturbations of) the diagonal block submatrices $M_{11},M_{22}$ and on the operator norm of the off-diagonal submatrices $M_{12}, M_{21}$.
The control on the smallest singular value of the whole matrix is somewhat degraded, but this is acceptable as we will only apply Lemma \ref{lem:schur} a bounded number of times.
If we can find a generalized diagonal of ``good" block submatrices that are well-invertible under additive perturbations, 
then after permuting the blocks to lie on the main diagonal we can apply the Schur complement bound along a nested sequence of submatrices partitioned as in \eqref{sketch:partition}, where $M_{11}$ is a ``good" matrix and $M_{22}$ is well-invertible by the induction hypothesis.
We remark that the strategy of leveraging properties of a small submatrix using the Schur complement formula was recently applied in a somewhat different manner in \citep{BEYY:band} to prove the universality of spectral statistics of random Hermitian band matrices.

\subsubsection*{Decomposition of the reduced digraph} 

At this point it is best to think of the regular partition $I_1,\dots, I_m$ as inducing a ``macroscopic scale" digraph $\mR = ([m],E)$ (often called the \emph{reduced digraph} in extremal combinatorics) that puts an edge $(k,l)\in E$ whenever the corresponding submatrix $A_{I_k,I_l}$ is pseudorandom and sufficiently dense. 
If we can cover the vertices of $\mR$ with vertex-disjoint directed cycles, then we will have found a generalized diagonal of submatrices of $M$ with the desired properties, and we can finish with a bounded number of applications of the Schur complement formula as described above. 

Of course, it may be the case that $\mR$ cannot be covered by disjoint cycles. 
For instance, if $A$ were to have all ones in the first $n/2$ columns and all zeros in the last $n/2$ columns then roughly half of the vertices of $\mR$ would have no incoming edges. 
This is where we make crucial use of the diagonal perturbation $Z\sqrt{n}$ (indeed, without this perturbation $M$ would be singular in this example).
The top left $n/2\times n/2$ submatrix of $M$ is dense, and we can apply Theorem \ref{thm:super} to control its smallest singular vale. 
The bottom right $n/2\times n/2$ submatrix is a diagonal matrix with diagonal entries of size at least $r_0\sqrt{n}$, and hence its smallest singular value is at least $r_0\sqrt{n}$. 
This argument even allows for the bottom right submatrix of $A$ to be nonzero but sufficiently sparse: we can use the triangle inequality and standard bounds on the operator norm of sparse random matrices to argue that the smallest singular value of the bottom right submatrix is still of order $\gg r_0\sqrt{n}$.

We handle the general case as follows. 
We greedily cover as many of the vertices of $\mR$ as we can with disjoint cycles -- call this set of vertices $U_{\cyc}\subset[m]$. 
At this point we have either covered the whole graph (and we are done) or the graph on the remaining vertices $U_{\free}$ is cycle-free. This means that the vertices of $\mR$ can be relabeled so that its adjacency matrix is upper-triangular on $U_{\free}\times U_{\free}$. 
Write $J_{\cyc} = \bigcup_{k\in U_{\cyc}} I_k$, $J_{\free}=\bigcup_{k\in U_{\free}}I_k$ and denote the corresponding submatrices of $A$ on the diagonal by $A_{\cyc}, A_{\free}$, and likewise for $M$. 
We thus have a relabeling of $[n]$ under which $A_{\free}$ is close to upper triangular (there may be some entries of $A_{\free}$ below the diagonal of size less than $\ha$, or which are contained in a small number of exceptional pairs from the regular partition).
Crucially, this relabeling has preserved the diagonal, so the submatrix $M_{\free}$ is a diagonal perturbation of an (almost) upper-triangular random matrix. 
We then show that such a matrix has smallest singular value of order $\gg_{r_0}\sqrt{n}$ with high probability. 
With another application of the Schur complement bound we can combine the control on the submatrices $M_{\cyc},M_{\free}$ (along with standard bounds on the operator norm for the off-diagonal blocks) to conclude the proof.
(Actually, the bad set $I_{\badd}$ of rows and columns requires some additional arguments, but we do not discuss these here.)

This concludes the high level description of the proof of Theorem \ref{thm:main}. 
We only remark that the above partitioning and cleaning procedures will generate various error terms and residual submatrices (such as the vertices in $I_{\badd}$, or the small proportion of pairs $(I_k,I_l)$ which are not sufficiently pseudorandom). 
As the smallest singular value is notoriously sensitive to perturbations, it will take some care to control these terms.
We will use some high-powered tools such as bounds on the operator norm of sparse random matrices and products of random matrices due to Lata\l a and Vershynin -- see Section \ref{sec:op}.

\subsubsection*{Invertibility from connectivity assumptions}

Now we state the specific pseudorandomness condition on a standard deviation profile under which we have good control on the smallest singular value.
While ``pseudorandom" generally means that the edge distribution in a graph is close to uniform on a range of scales, we will only need control from below on the edge densities (morally speaking, we want the matrix $A$ to be as far as possible from the zero matrix, the most poorly invertible matrix). 
The following one-sided condition is taken from the combinatorics literature (see \cite[Definition 1.6]{KoSi:survey}). The reader should recall the notation introduced in \eqref{def:nbhd}--\eqref{Aha}.

\begin{definition}[Super-regularity]	\label{def:super}
Let $A$ be an $n\times m$ matrix with non-negative entries.
For $\delta,\eps\in (0,1)$, we say that $A$ is \emph{$(\delta,\eps)$-super-regular} if the following hold:
\begin{enumerate}[(1)]
\item $|\mN_A(i)| \ge \delta m$ for all $i\in [n]$;
\item $|\mN_{A^\tran}(j)|\ge \delta n$ for all $j\in [m]$;
\item $e_A(I,J)\ge \delta |I||J|$ for all $I\subset[n],J\subset[m]$ with $|I|\ge \eps n$ and $|J|\ge \eps m$. 
\end{enumerate}
\end{definition}

The reader should compare this condition with Definition \ref{def:broad}. 
Conditions (1) and (2) are are the same in both definitions, while it is not hard to see that condition (3) above implies 
\begin{equation}
|\mN_{A^\tran}^{(\delta)}(J)| \ge (1-\eps)n
\end{equation}
whenever $|J|\ge \eps n$ (with notation as in \eqref{def:broadnbr}), which is stronger than condition (3) in Definition \ref{def:broad} for such $J$. 
On the other hand, conditions (1) and (2) imply that $|\mN_{A^\tran}^{(\sqrt{\delta}/2)}(J)|\ge \frac12\delta n$ for any $J\subset[m]$ (see Lemma \ref{lem:goodrows0}), so super-regularity is stronger than broad connectivity for $\eps,\eta$ sufficiently small depending on $\delta$.

\begin{theorem}[Matrix with super-regular profile]	\label{thm:super}
Let $M=A\circ X+B$ be an $n\times n$ matrix as in Definition \ref{def:profile}.
Assume that $A(\ha)$ (as defined in \eqref{Aha}) is $(\delta,\eps)$-super-regular for some $\delta,\ha\in (0,1)$ and $0<\eps<c_1\delta \ha^2$ with $c_1>0$ a sufficiently small constant.
For any $\gamma\ge1/2$ there exists $\beta = O(\gamma^2)$ such that
\begin{equation}	\label{bd:super}
\pr\big( s_n(M) \le n^{-\beta},\, \|M\|\le n^\gamma\,\big) \ll_{\gamma,\delta,\ha,\kappa_0} \sqrt{\frac{\log n}{n}}.
\end{equation}
\end{theorem}

Note that Theorem \ref{thm:super} allows for a mean profile $B$ of arbitrary polynomial size in operator norm, whereas in Theorem \ref{thm:broad} we only allowed $\|B\|=O(\sqrt{n})$. 
The ability to handle such large perturbations will be crucial in the proof of Theorem \ref{thm:main}, as the iterative application of the Schur complement bound discussed above will lead to perturbations of increasingly large polynomial order. 

We defer discussion of the key technical ideas for Theorem \ref{thm:broad} and Theorem \ref{thm:super} to Sections \ref{sec:comp} and \ref{sec:incomp}. 
We only mention here that our proof of Theorem \ref{thm:super} makes crucial use of a new ``entropy reduction" argument, which allows us to control the event that $\|Mu\|$ is small for some $u$ in certain portions of the sphere $S^{n-1}$ by the event that this holds for some $u$ in a random net of relatively low cardinality. 
The argument uses an improvement by Spielman and Srivastava \citep{SpSr:rit} of the classic Restricted Invertibility Theorem due to Bourgain and Tzafriri \citep{BoTz:rit} -- see Section \ref{sec:comp} for details.

\subsection{Organization of the paper}

The rest of the paper is organized as follows.
Sections \ref{sec:anti}, \ref{sec:comp} and \ref{sec:incomp} are devoted to the proofs of Theorems \ref{thm:broad} and \ref{thm:super}.
We prove these theorems in parallel as they involve many similar ideas.
In Section \ref{sec:anti} we collect some standard lemmas on anti-concentration for random walks and products of random matrices with fixed vectors, along with some facts about nets in Euclidean space.
In Section \ref{sec:comp} we show that random matrices as in Theorems \ref{thm:broad} and \ref{thm:super} are well-invertible over sets of ``compressible" vectors in the unit sphere, and in Section \ref{sec:incomp} we establish control over the complementary set of ``incompressible" vectors.
Theorem \ref{thm:main} is proved in Section \ref{sec:diag}.

\subsection{Notation}		\label{sec:notation}

In addition to the asymptotic notation defined at the beginning of the article,
we will occasionally use the notation $f=o(g)$ to mean that $f/g\rightarrow 0$ as $n\rightarrow \infty$, where the parameter $n$ will be the size of the matrix under consideration (this will only be for the sake of brevity, as all of our arguments are quantitative).

$\mM_{n,m}(\C)$ denotes the set of $n\times m$ matrices with complex entries. When $m=n$ we will write $\mM_n(\C)$.
For a matrix $A=(a_{ij})\in \mM_{n,m}(\C)$ we will sometimes use the notation $A(i,j)= a_{ij}$.
For $I\subset[n],J\subset[m]$, $A_{I,J}$ denotes the $|I|\times |J|$ submatrix with entries indexed by $I\times J$.
We abbreviate $A_{J}:= A_{J,J}$.

$\|\cdot\|$ denotes the Euclidean norm when applied to vectors, and the $\ell_2^m\to\ell_2^n$ operator norm when applied to an $n\times m$ matrix.
$\|A\|_{\HS}$ denotes the Hilbert--Schmidt (or Frobenius) norm of a matrix $A$. 
We will sometimes denote the smallest singular value of a square matrix $M$ by $s_{\min}(M)$ (in situations where $M$ is a submatrix of a larger matrix this will often be clearer than writing the dimension).

We denote the unit sphere in $\C^n$ by $S^{n-1}$. 
For $J\subset [n]$, we denote by $\C^J\subset\C^n$ (resp. $S^J\subset S^{n-1}$) the set of vectors (resp. unit vectors) in $\C^n$ supported on $J$.
Given a vector $v\in \C^n$, we denote by $v_J\in \C^n$ the projection of $v$ to the coordinate subspace $\C^J$.
For $m\in \N$, $x\in \R$, ${[m]\choose x}$ denotes the family of subsets of $[m]$ of size $\lf x\rf$.

When considering a random matrix $M$ as in Definition \ref{def:profile}, 
we use $R_i$ to denote the $i$th row of $M$, and write
\begin{equation}	\label{def:sigmaalgebra}
\mF_{I,J}:= \langle \{\xi_{ij}\}_{i\in I,j\in J}\rangle
\end{equation}
for the sigma algebra of events generated by the entries $\{\xi_{ij}\}_{i\in I,j\in J}$ of $X$.
For $I\subset[n]$ we write $\pr_I(\cdot)$ for probability conditional on $\mF_{[n]\setminus I,[n]}$.\\

\noindent{\bf Acknowledgements.}
The author thanks David Renfrew and Terence Tao for useful conversations, and also thanks David Renfrew for providing helpful comments on a preliminary version of the manuscript.

\section{Preliminaries}		\label{sec:anti}

\subsection{Partitioning and discretizing the sphere}		\label{sec:net}

For the proofs of Theorems \ref{thm:broad} and \ref{thm:super} we make heavy use of ideas and notation developed in \citep{LPRT,LPRTV2,Rudelson:inv, RuVe:ilo} and related ideas from geometric functional analysis. 
In particular, in order to lower bound
\[
s_n(M) = \inf_{u\in S^{n-1}} \|Mu\|
\]
we partition the sphere into sets of vectors of different levels of ``compressibility", which we presently define, and separately obtain control on the infimum of $\|Mu\|$ over each set.

Recall from Section \ref{sec:notation} our notation $\C^J\subset \C^m$ for the set of vectors supported on $J\subset[m]$.
For a set $T\subset \C^n$ and $\rho>0$ we write $T_\rho$ for the set of points within Euclidean distance $\rho$ of $T$.
We recall also the following definitions from \citep{RuZe}.
For $\theta,\rho\in (0,1)$, we define the set of \emph{compressible vectors}
\begin{equation}	\label{def:compr}
\Comp(\theta,\rho) := S^{m-1}\cap\bigcup_{J\in {[m]\choose  \theta m}} (\C^J)_\rho
\end{equation}
and the complementary set of \emph{incompressible vectors}
\begin{equation}	\label{def:incompr}
\Incomp(\theta,\rho) := S^{m-1}\setminus \Comp(\theta,\rho).
\end{equation}
That is, $\Comp(\theta,\rho)$ is the set of unit vectors within (Euclidean) distance $\rho$ of a vector supported on at most $\theta m$ coordinates.
On the other hand, incompressible vectors enjoy the following property which will lead to good anti-concentration properties for an associated random walk.

\begin{lemma}[Incompressible vectors are spread, cf.\ {\cite[Lemma 3.4]{RuVe:ilo}}]	\label{lem:spread}
Fix $\theta,\rho\in (0,1)$ and let $v\in \Incomp(\theta,\rho)$. 
There is a set $L^+\subset[m]$ with $|L^+|\ge \theta m$ such that $|v_j|\ge \rho/\sqrt{m}$ for all $j\in L^+$.
Moreover, for all $\lambda\ge1$ there is a set $L\subset [m]$ with $|L| \ge (1-\frac1{\lambda^2})\theta m$ such that for all $j\in L$, 
\[
\frac{\rho}{\sqrt{m}} \le |v_j| \le \frac{\lambda}{\sqrt{\theta m}}.
\]
\end{lemma}

\begin{proof}
Take $L^+=\{j:|v_j|\ge \rho/\sqrt{m}\}$ and denote $L^-= \{j:|v_j|\le \lambda/\sqrt{\theta m}\}$.
Since $v$ lies a distance at least $\rho$ from any vector supported on at most $\theta m$ coordinates we must have $|L^+|\ge \theta m$, which gives the first claim. 
On the other hand, since $v\in S^{m-1}$, by Markov's inequality we have 
$|(L^-)^c|\le \theta m/\lambda^2$, 
so taking $L=L^+\cap L^-$ we have $|L|\ge (1-\frac1{\lambda^2})\theta m$.
\end{proof}

For fixed choices of $\theta,\rho$ we informally refer to the coordinates of $v\in \Incomp(\theta,\rho)$ where $|v_j|\ge \rho/\sqrt{n}$ as the \emph{essential support of $v$}. 

Now we recall a standard fact about nets of the sphere of controlled cardinality.
For $\rho>0$, recall that a $\rho$-net of a set $T\subset \C^m$ is a finite subset $\Sigma\subset T$ such that for all $v\in T$ there exists $v'\in \Sigma$ with $\|v-v'\|\le \rho$.

\begin{lemma}[Metric entropy of the sphere]	\label{lem:net}
Let $V\subset \C^m$ be a subspace of (complex) dimension $k$, let $T\subset V\cap S^{m-1}$, and let $\rho\in(0,1)$.
Then $T$ has a $\rho$-net $\Sigma\subset T$ of cardinality $|\Sigma|\le (3/\rho)^{2k}$.
\end{lemma}

\begin{proof}
Let $\Sigma\subset T$ be a $\rho$-separated (in Euclidean distance) subset that is maximal under set inclusion.
It follows from maximality that $\Sigma$ is a $\rho$-net of $T$.
Let $\Sigma_{\rho/2}$ denote the $\rho/2$ neighborhood of $\Sigma$ in $V$.
Noting that $\Sigma_{\rho/2}$ is a disjoint union of $k$-dimensional Euclidean balls of radius $\rho/2$, we have
$$|\Sigma| c_{k} (\rho/2)^{2k}  \le \text{vol}_k(\Sigma_{\rho/2}) \le c_{k} (1+\rho/2)^{2k}$$
where $\text{vol}_k$ denotes the $k$-dimensional Lebesgue measure on $V$ and $c_{k}$ is the volume of the Euclidean unit ball in $\C^k$. 
The desired bound follows by rearranging.
\end{proof}

\subsection{Anti-concentration for scalar random walks}		\label{sec:fourier}

In this subsection we collect some standard anti-concentration estimates for scalar random walks, which are perhaps the most central tool for proving that random matrices are (well-)invertible with high probability. 

\begin{definition}[Concentration probability]
Let $\xi$ be a complex-valued random variable. 
For $v\in \C^n$ we let
\begin{equation}
S_\xi(v) = \sum_{j=1}^n \xi_jv_j
\end{equation}
where $\xi_1,\dots,\xi_n$ are iid copies of $\xi$.
For $r\ge0$ we define the \emph{concentration probability}
\begin{equation}
p_{\xi,v}(r)= \sup_{z\in \C}\pr\big( |S_\xi(v)-z|\le r\big).
\end{equation}
\end{definition}

Throughout this section we operate under the following distributional assumption on $\xi$.

\begin{definition}[Controlled second moment, cf.\ {\cite[Definition 2.2]{TaVu:circ}}]	\label{def:kappa}
Let $\kappa\ge 1$. A complex random variable $\xi$ is said to have \emph{$\kappa$-controlled second moment} if one has the upper bound
\begin{equation}	\label{cond.kappa1}
\e|\xi|^2\le \kappa
\end{equation}
(in particular, $|\e \xi|\le \kappa^{1/2}$), and the lower bound
\begin{equation}	\label{cond.kappa2}
\e [\re(z\xi-w)]^2 \un(|\xi|\le \kappa) \ge \frac1\kappa [\re(z)]^2
\end{equation}
for all $z\in \C, a\in \R$.
\end{definition}

Roughly speaking, a complex random variable $\xi$ has controlled second moment if its distribution has a one-(real-)dimensional marginal with fairly large variance on some compact set.
The following is a quantitative version of \cite[Lemma 2.4]{TaVu:circ}, and shows that by multiplying the matrices $X$ and $B$ in Definition \ref{def:profile} by a scalar phase (amounting to multiplying $M$ by a phase, which does not affect its singular values) we can assume the atom variable $\xi$ has $O(\kappa_0)$-controlled second moment in all of our proofs with no loss of generality.
The proof is 
deferred to Appendix \ref{app:anti}.

\begin{lemma}		\label{lem:wlog.kappa}
Let $\xi$ be a centered complex random variable with unit variance, and assume $\xi$ is $\kappa_0$-spread for some $\kappa_0 \ge 1$ (see Definition \ref{def:spread}).
Then there exists $\theta \in \R$ such that $e^{i\theta}\xi$ has $\kappa$-controlled second moment for some $\kappa=O(\kappa_0)$.
\end{lemma}


Below we give two standard bounds on the concentration function $p_{\xi,v}(r)$ when $\xi$ is a $\kappa$-controlled random variable and $v\in S^{n-1}$. 
The first gives a crude constant order bound that is uniform in $v\in S^{n-1}$:

\begin{lemma}[Crude anti-concentration, cf.\ {\cite[Corollary 6.3]{TaVu:smooth}}]	\label{lem:anti_crude}
Let $\xi$ be a complex random variable with $\kappa$-controlled second moment. 
There exists $r_0>0$ depending only on $\kappa$ such that $p_{\xi,v}(r_0) \le 1-r_0$ for all $v\in S^{n-1}$.
\end{lemma}

Note that Lemma \ref{lem:anti_crude} is sharp for the case that $v$ is a standard basis vector.
The following gives an improved bound when $v$ has small $\ell_\infty$ norm.

\begin{lemma}[Improved anti-concentration]		\label{lem:anti_improved}
Let $\xi$ be a complex random variable that is $\kappa$-controlled for some $\kappa>0$, and let $v\in S^{n-1}$.
For all $r\ge0$,
\begin{equation}	\label{be:1d}
p_{\xi,v}(r) \ll_\kappa r + \|v\|_\infty.
\end{equation}
\end{lemma}

Lemma \ref{lem:anti_improved} can be deduced from the Berry--Ess\'een theorem (which is the approach taken in \citep{LPRT}, for instance), but this would require $\xi$ to have finite third moment, which we do not assume. 
(Generally speaking, higher moment assumptions should only be necessary to prove concentration bounds as opposed to anti-concentration.) 
Since we could not locate a proof in the literature for the case that $\xi$ 
and
the coefficients of $v$ take values in $\C$, we provide a proof in 
Appendix \ref{app:anti}. 

\subsection{Anti-concentration for the image of a fixed vector}	\label{sec:image}

In this subsection we boost the anti-concentration bounds for scalar random variables from the previous sections to anti-concentration for the image of a fixed vector under a random matrix.
The following lemma of Rudelson and Vershynin is convenient for this task.

\begin{lemma}[Tensorization, cf.\ {\cite[Lemma 2.2]{RuVe:ilo}}]	\label{lem:tensorize}
Let $\zeta_1,\dots, \zeta_n$ be independent non-negative random variables.
\begin{enumerate}[(a)]
\item Suppose that for some $\eps_0,p_0>0$ and all $j\in [n]$, $\pro{\zeta_j\le \eps_0}\le p_0$. There are $c_1,p_1\in (0,1)$ depending only on $p_0$ such that
\begin{equation}
\pr\bigg( \sum_{j=1}^n \zeta_j^2 \le c_1\eps_0^2n\bigg) \le p_1^n.
\end{equation}
\item  Suppose that for some $K,\eps_0\ge 0$ and all $j\in [n]$, $\pro{\zeta_j\le\eps} \le K\eps$ for all $\eps\ge \eps_0$.
Then for all $\eps\ge \eps_0$,
\begin{equation}
\pr\bigg( \sum_{j=1}^n \zeta_j^2 \le\eps^2n \bigg) \le (CK\eps)^n.
\end{equation}
\end{enumerate}
\end{lemma}

Note that in part (a) we have given more specific dependencies on the parameters than in \citep{RuVe:ilo}. For completeness we provide the proof of this modified version in Appendix \ref{app:anti}.

Let $\M=A\circ X+B$ be as in Definition \ref{def:profile}.
Recall that we denote by $R_i$ the $i$th row of $\M$.
In the following lemmas we assume that the atom variable $\xi$ has $\kappa$-controlled second moment for some fixed $\kappa\ge1$.
For $v\in \C^m$ and $i\in [n]$ we write 
\begin{equation}	\label{def:vi}
v^i :=(v_j\sig_{ij})_{j=1}^m
\end{equation}
For $\alpha>0$ we denote 
\begin{equation}	\label{def:Ialpha}
I_\alpha(v):= \{i\in [n]: \|v^i\|\ge \alpha\}.
\end{equation}

\begin{lemma}[Crude anti-concentration for the image of a fixed vector]	\label{lem:fixed_crude}
Fix $v\in \C^m$ and let $\alpha>0$ such that $I_\alpha(v)\ne \varnothing$.
For all $I_0\subset I_\alpha(v)$, 
\begin{equation}
\sup_{w\in \C^n} \pr_{I_0}\Big( \|\M v -w\| \le c_0 \alpha|I_0|^{1/2} \Big) \le e^{-c_0|I_0|}
\end{equation}
where $c_0>0$ is a constant depending only on $\kappa$ (recall our notation $\pr_{I_0}(\,\cdot\,)$ from Section \ref{sec:notation}).
\end{lemma}

\begin{proof}
Fix $w\in \C^n$ arbitrarily. 
For any $i\in I_\alpha(v)$ and any $t\ge0$ we have
\begin{align*}
\pro{ |R_i\cdot v -w_i|\le t} &\;\le\; p_{\xi,v^i}(t) \,=\,p_{\xi,v^i/\|v^i\|}(t/\|v^i\|) \,\le\, p_{\xi,v^i/\|v^i\|}(t/\alpha).
\end{align*}
Taking $t=\alpha r_0$, by Lemma \ref{lem:anti_crude} we have
\begin{equation}
\pro{ |R_i\cdot v -w_i|\le \alpha r_0} \le 1-r_0
\end{equation}
where $r_0>0$ depends only on $\kappa$.

Fix $I_0\subset I_\alpha(v)$ arbitrarily. We may assume without loss of generality that $I_0$ is non-empty. 
By Lemma \ref{lem:tensorize}(a) there exists $c_1>0$ depending only on $\kappa$ such that
\begin{equation}	\label{eq:crude1}
\pr_{I_0}\bigg( \sum_{i\in I_0} |R_i\cdot v - w_i|^2 \le c_1 r_0^2\alpha^2 |I_0|\bigg) \le e^{-c_1|I_0|}.
\end{equation}
Now for any $\tau\ge 0$, 
\begin{align*}
\pr_{I_0}\Big( \|\M v -w\| \le \tau |I_0|^{1/2}\Big) &= \pr_{I_0}\bigg( \sum_{i=1}^n |R_i\cdot v -w_i|^2 \le \tau^2 |I_0|\bigg)\\
&\le \pr_{I_0}\bigg( \sum_{i\in I_0} |R_i\cdot v -w_i|^2 \le \tau^2 |I_0|\bigg)
\end{align*}
and the claim follows by taking $\tau=c_1^{1/2}r_0\alpha =: c_0\alpha$ and applying \eqref{eq:crude1}.
\end{proof}

By similar lines, using Lemmas \ref{lem:tensorize}(b) and \ref{lem:anti_improved} in place of Lemmas \ref{lem:tensorize}(a) and \ref{lem:anti_crude}, respectively, one obtains the following, which is superior to Lemma \ref{lem:fixed_crude} for vectors $v$ with small $\ell_\infty$ norm. 
The details are omitted.

\begin{lemma}[Improved anti-concentration for the image of a fixed vector]	\label{lem:fixed_improved}
Fix $v\in \C^m$. 
Let $\alpha>0$ such that $I_\alpha(v)\ne \varnothing$ and fix $I_0\subset I_\alpha(v)$ nonempty.
For all $t\ge 0$,
\begin{equation}
\sup_{w\in \C^n} \pr_{I_0}\Big( \|\M v -w\| \le t|I_0|^{1/2}\Big) \le O_\kappa\bigg(\frac{1}{\alpha}\big( t+\|v\|_\infty\big) \bigg)^{|I_0|}.
\end{equation}
\end{lemma}

\section{Invertibility from connectivity: Compressible vectors}	\label{sec:comp}

In this section we combine the anti-concentration estimates from Section \ref{sec:anti} with union bounds over $\eps$-nets (as obtained for instance from Lemma \ref{lem:net}) to prove that with high probability, a random matrix $M$ as in Theorem \ref{thm:broad} or Theorem \ref{thm:super} is well-invertible on the set of compressible vectors $\Comp(\theta,\rho)$ (as defined in \eqref{def:compr}) for appropriate choices of $\theta,\rho$.
Hence, there will be a competition between the quality of the anti-concentration estimates and the cardinality of the $\eps$-nets.
For small values of $\theta$ we can use $\eps$-nets of small cardinality, but only have poor anti-concentration bounds (namely, Lemma \ref{lem:fixed_crude}), while for large $\theta$ the nets are very large, but we have access to the improved anti-concentration of Lemma \ref{lem:fixed_improved}.

In both cases we start with a crude result, Lemma \ref{lem:high}, giving control for the vectors in $\Comp(\theta_0,\rho_0)$ for some small value of $\theta_0$ (possibly depending on $n$).
We then use an iterative argument argument to obtain control on $\Comp(\theta,\rho)$ for larger values of $\theta$ while lowering the parameter $\rho$.
For Theorem \ref{thm:broad} we want to take $\theta$ close to 1, while for Theorem \ref{thm:super} a constant order value of $\theta$ will suffice.

It turns out that that while the standard $\eps$-net from Lemma \ref{lem:net} suffices to prove Lemma \ref{lem:high}, it is insufficient to obtain control on $\Comp(\theta,\rho)$ for the desired values of $\theta$.
For the broadly connected case this is essentially due to working in $\C^n$ rather than $\R^n$, which causes a factor $2$ increase in metric entropies (this difficulty was not present in the proof of Theorem \ref{thm:ruze} in \citep{RuZe} as they worked in $\R^n$).
The situation is worse for the case of Theorem \ref{thm:super}, the main source of difficulty being that $\|B\|$ can be of arbitrary polynomial order. As a consequence, the starting point $\theta_0$ for our iterative argument will be of size $o(1)$.
This prevents us from using the third condition of the super-regularity hypothesis (see Definition \ref{def:super}), which only ``sees" vectors that are essentially supported on more than $\eps n$ coordinates.

We deal with this by \emph{reducing the entropy cost} of the nets over which we take union bounds.
In Section \ref{sec:entropy} we prove Lemma \ref{lem:entropy} which shows, roughly speaking, that if we have already established control on vectors in $\Comp(\theta,\rho)$ for some $\theta,\rho$, then we can control the vectors in $\Comp(\theta+\Delta,\rho')$ for some small $\Delta, \rho'$ using a \emph{random} net of significantly smaller cardinality than the net provided by Lemma \ref{lem:net}.
We can then increment $\theta$ from $\theta_0$ up to size $\gg 1$, taking steps of size $\Delta$.
For the broadly connected case we can continue and take $\theta$ as close to $1$ as desired.
The entropy reduction argument for Lemma \ref{lem:entropy} makes use of a strong version of the well-known Restricted Invertibility Theorem due to Spielman and Srivastava -- see Theorem \ref{thm:rit}.

We now state the main results of this section.
For $K\ge1$ we denote the \emph{boundedness event}
\begin{equation}	\label{event:bdd}
\mB(K):= \big\{\|\M\|\le K\sqrt{n}\,\big\}.
\end{equation}
With a fixed choice of $K$ we write
\begin{equation}	\label{def:eventcomp}
\event(\theta,\rho):= \mB(K) \wedge \big\{\, \exists u\in \Comp(\theta,\rho): \|\M u\|\le \rho K \sqrt{n}\,\big\}.
\end{equation}

\begin{proposition}[Compressible vectors: broadly connected profile]	\label{prop:slight}
Let $\M=A\circ X+B$ be as in Definition \ref{def:profile} with $n/2\le m\le 2n$, and assume that $\xi$ has $\kappa$-controlled second moment for some $\kappa\ge1$ (see Definition \ref{def:kappa}).
Let $K\ge1$ and $\ha,\delta,\nu\in (0,1)$.
There exist $\theta_0(\kappa,\ha,\delta,K)>0$ and $\rho(\kappa,\ha,\delta,\nu,K)>0$ such that the following holds. 
Assume
\begin{enumerate}[(1)]
\item $|\mN_{A(\ha)^\tran}(j)|\ge \delta n$ for all $j\in [m]$; \vspace{.2cm}
\item $|\mN_{A(\ha)^\tran}^{(\delta)}(J)| \ge \min((1+\nu)|J|,n)$ for all $J\subset[m]$ with $|J|\ge \theta_0m$.
\end{enumerate}
Then for any $0<\theta\le (1-\frac\delta4)\min(\frac{n}m,1)$, 
\begin{equation}
\pro{\event(\theta,\rho)}\ll_{\kappa,\ha,\delta,\nu,K} \expo{-c_\kappa\delta\ha^2n}
\end{equation}
where $c_\kappa>0$ depends only on $\kappa$.
\end{proposition}


The following gives control of compressible vectors for more general profiles than in Proposition \ref{prop:slight} (essentially removing the condition (2)). 
However, we have to take the parameter $\rho$ much smaller, and we only cover vectors that are essentially supported on a small (linear) proportion of the coordinates, rather than a proportion close to one. 

\begin{proposition}[Compressible vectors: general profile with large perturbation]	\label{prop:comp}
Let $\M=A\circ X+B$ be as in Definition \ref{def:profile} with $n/2\le m\le 2n$.
Assume $\xi$ has $\kappa$-controlled second moment for some $\kappa\ge1$, and that for some $\sig_0>0$ we have
\begin{equation}	\label{LB:columnvar}
\sum_{i=1}^n \sig_{ij}^2\ge \sig_0^2n\quad\text{ for all }j\in [m].
\end{equation}
Fix $\gamma\ge 1/2$ and let $1\le K=O(n^{\gamma-1/2})$.
Then for some $\rho= \rho(\gamma,\sig_0,\kappa,n) \gg_{\gamma,\sig_0,\kappa} n^{-O(\gamma^2)}$ and a sufficiently small constant $c_0>0$ we have
\begin{equation}
\pro{\event(c_0\sig_0^2,\rho)}\ll_{\gamma,\sig_0,\kappa}\expo{-c_\kappa\sig_0^2n}
\end{equation}
where $c_\kappa>0$ depends only on $\kappa$.
\end{proposition}

\subsection{Highly compressible vectors}		\label{sec:high}

In this subsection we establish the following crude version of Proposition \ref{prop:comp}, giving control on vectors in $\Comp(\theta_0,\rho_0)$ with $\theta_0$ sufficiently small depending on $\sig_0$ and $K$.

\begin{lemma}[Highly compressible vectors]	\label{lem:high}
Let $\M=A\circ X+B$ be as in Definition \ref{def:profile} with $m\le 2n$.
Assume that $\xi$ has $\kappa$-controlled second moment for some $\kappa\ge1$.
Suppose also that there is a constant $\sig_0>0$ such that for all $j\in [m]$, $\sum_{i=1}^n \sig_{ij}^2 \ge \sig_0^2n$.
Let $K\ge 1$.
Then with notation as in \eqref{def:eventcomp} we have
\begin{equation}
\pro{ \event(\theta_0,\rho_0)} \le e^{-c_\kappa\sig_0^2n}
\end{equation}
where $\theta_0= c_\kappa\sig_0^2 /\log(K/\sig_0^2)$ and $\rho_0=c_\kappa\sig_0^2/K$ for a sufficiently small $c_\kappa>0$ depending only on $\kappa$.
\end{lemma}

We will need the following lemma, which ensures that the set $I_\alpha(v)$ from \eqref{def:Ialpha} is reasonably large when the columns of $A$ have large $\ell_2$ norm.
A similar argument has been used in \citep{LiRi} and \citep{RuZe}.

\begin{lemma}[Many good rows]	\label{lem:goodrows0}
Let $A$ be an $n\times m$ matrix as in Definition \ref{def:profile}, and assume that for some $\sig_0>0$ we have $\sum_{i=1}^n \sig_{ij}^2\ge \sig_0^2 n$ for all $j\in [m]$.
Then for any $v\in S^{m-1}$ we have $|I_{\sig_0/2}(v)|\ge \frac12\sig_0^2n$.
\end{lemma}

\begin{proof}
Writing $\alpha=\sig_0/\sqrt{2}$, we have
\begin{align*}
\sig_0^2 n &\le \sum_{i=1}^n\sum_{j=1}^m |v_j|^2 \sig_{ij}^2\\
&=\sum_{i\in I_{\alpha}(v)}\sum_{j=1}^m |v_j|^2 \sig_{ij}^2 + \sum_{i\notin I_{\alpha}(v)}\sum_{j=1}^m |v_j|^2 \sig_{ij}^2\\
&\le \sum_{i\in I_{\alpha}(v)}\sum_{j=1}^m |v_j|^2 + \sum_{i\notin I_{\alpha}(v)} \frac{1}2\sig_0^2 \\
&\le |I_{\alpha}(v)| + \frac12 \sig_0^2n
\end{align*}
and rearranging gives the claim.
\end{proof}

\begin{proof}[Proof of Lemma \ref{lem:high}]
Fix $J\subset [m]$ of size $\lf \theta_0m\rf$ and let $v\in S^J$ be arbitrary. 
Writing $\alpha=\sig_0/\sqrt{2}$, by Lemma \ref{lem:fixed_crude} and our choice of $\rho_0$ (with $c_\kappa>0$ sufficiently small depending on $\kappa$),
\[
\pro{ \|\M v\|\le \rho_0K \sqrt{n}} \le \pro{ \|\M v\|\le c_\kappa\sig_0|I_{\alpha}(v)|^{1/2}}
\le e^{-c_\kappa|I_{\alpha}(v)|}.
\]
Applying Lemma \ref{lem:goodrows0}, we obtain
\begin{equation}	\label{high:forallv}
\pro{ \|\M v\|\le \rho_0K\sqrt{n}} \le e^{-c_\kappa\sig_0^2 n}\qquad \forall v\in S^J
\end{equation}
(adjusting $c_\kappa$). 
By Lemma \ref{lem:net} we may fix $\Sigma_J\subset S^J$ a $\rho_0/4$-net for $S^J$ such that $|\Sigma_J|\le (12/\rho_0)^{2k}$.
Suppose that $\|M\|\le K\sqrt{n}$ and that $\|Mu\|\le \rho_0 K\sqrt{n}$ for some $u\in S^{m-1}\cap (\C^J)_{\rho_0/4}$.
Let $u'\in \C^J$ with $\|u-u'\|\le \rho_0/4$, and let $u''\in \Sigma_J$ with $\|u''-\frac{u'}{\|u'\|}\|\le \rho_0/4$. 
By the triangle inequality,
\[
\| u-u''\| \le \|u-u'\| + \left\| u'-\frac{u'}{\|u'\|}\right\| + \left\| \frac{u'}{\|u'\|} - u''\right\| \le 3\rho_0/4
\]
where the bound on the middle term follows from $|\|u'\|-1|\le \rho_0/4$ (also by the triangle inequality).
We have
\[
\|Mu''\|\le \|Mu\| + \|M(u-u'')\| \le \rho_0K\sqrt{n} + K\sqrt{n}\cdot (3\rho_0/4) \le 2\rho_0 K\sqrt{n}.
\]
Applying the union bound and \eqref{high:forallv} (adjusting $c_\kappa$ to replace $\rho_0$ by $2\rho_0$),
\begin{align*}
&\pro{ \exists u\in S^{m-1}\cap (\C^J)_{\rho_0/8}: \|Mu\|\le \rho_0K\sqrt{n}} \\
&\qquad\qquad\qquad\qquad\qquad\qquad\le \pro{ \exists u''\in \Sigma_J: \|Mu''\|\le 2\rho_0K\sqrt{n}}\\
&\qquad\qquad\qquad\qquad\qquad\qquad\le O(1/\rho_0)^{2\theta_0 m} e^{-c_\kappa a_0^2n}
\end{align*}
From \eqref{def:compr} and applying the union bound over all choice of $J\in {[m]\choose  \theta_0 m}$, 
\begin{align*}
\pro{ \event(\theta_0,\rho_0/4)} &\le O(1/\theta_0)^{\theta_0m} O(1/\rho_0)^{2\theta_0m} e^{-c_\kappa\sig_0^2n}\le O\left(\frac{1}{\theta_0\rho_0^2}\right)^{2\theta_0 n} e^{-c_\kappa\sig_0^2n},
\end{align*}
where we used our assumption $m\le 2n$.
The desired bound now follows from substituting our choices of $\theta_0,\rho_0$, and again adjusting the constant $c_\kappa$ to replace $\rho_0/4$ by $\rho_0$ in the above.
\end{proof}

\subsection{An entropy reduction lemma}		\label{sec:entropy}

The aim of this subsection is to establish the following:

\begin{lemma}[Control by a random net of small cardinality]	\label{lem:entropy}
For every $I\subset[n], J\subset [m]$, $\eps>0$ there is a random finite set $\Sigma_{I,J}(\eps)\subset S^J$, measurable with respect to $\mF_{I.J} =\langle \{\xi_{ij}\}_{i\in I,j\in J}\rangle$, such that the following holds.
Let $\rho\in (0,1)$, $K>0$ and $0<\theta< \frac nm$.
On $\mB(K)\wedge \event(\theta,\rho)^c$, for all $J\subset [m]$ with $|J|>\theta m$ and all $\beta,\rho'\in (0,1)$, 
putting
\begin{equation}	\label{def:rhodub}
\rho''=\frac{6\rho'}{\beta \rho}\left(\frac{n}{\lf \theta m\rf}\right)^{1/2}
\end{equation}
there exists $I\subset[n]$ with $|I|=\lf (1-\beta)^2\lf \theta m\rf\rf$ such that 
\begin{enumerate}[(1)]
\item $|\Sigma_{I,J}(\rho'')| \le (C/\rho'')^{2(|J|-|I|)}$ for an absolute constant $C>0$, and
\item for any $u\in S^{m-1}\cap (\C^J)_{\rho'}$ such that $\|\M u\|\le \rho'K\sqrt{n}$, we have $\dist(u,\Sigma_{I,J}(\rho'')) \le 3\rho''$.
\end{enumerate}
Furthermore, writing
\begin{equation}
\good_{I,J}(\rho'') := \bigg\{ \big| \Sigma_{I,J}(\rho'')\big|\le \bigg(\frac{C}{\rho''}\bigg)^{2(|J|-|I|)} \bigg\}
\end{equation}
we have that for any $\theta'\in (\theta,1]$,
\begin{align}
&\event(\theta,\rho)^c \wedge \event(\theta',\rho') \subset \notag\\
&\quad\quad
\bigvee_{J\in {[m]\choose  \theta' m}} \bigvee_{I\in {[n]\choose  (1-\beta)^2\lf \theta m\rf}} \bigg(  \good_{I,J}(\rho'')
\wedge\Big\{ \exists u \in \Sigma_{I,J}(\rho''): \|\M u\|\le 4\rho''K\sqrt{n} \Big\}\bigg).\label{cont:entropy}
\end{align}
\end{lemma}

\begin{remark}
We obtain the random set $\Sigma_{I,J}(\eps)$ as the intersection of the sphere $S^J$ with an $\eps$-net of the kernel of the submatrix $\M_{I, J}$. 
However, for our purposes it only matters that it is fixed by conditioning on the rows $\{R_i\}_{i\in I}$, has small cardinality, and serves as a net for almost-null vectors of $\M$ that are supported on $J$.
\end{remark}

To prove Lemma \ref{lem:entropy} we use the following version of the Restricted Invertibility Theorem \citep{SpSr:rit} (the version below is taken from \cite[Theorem 3.1]{MSS:icm}).

\begin{theorem}[Restricted Invertibility Theorem]	
\label{thm:rit}
Suppose $v_1,\dots, v_n\in \C^m$ are such that $\sum_{i=1}^n v_iv_i^* = I_m$. 
For any $\beta\in (0,1)$, there is a subset $I\subset [n]$ of size $|I|= \lf (1-\beta)^2m\rf$ for which
\begin{equation}
\lambda_{|I|}\bigg(\sum_{i\in I} v_iv_i^*\bigg) \ge \beta^2m/n
\end{equation}
where $\lambda_k(A)$ denotes the $k$th largest eigenvalue of a Hermitian matrix $A$.
\end{theorem}

This has the following consequence, which can be seen as a robust quantitative version of the basic fact from linear algebra that the row rank of a matrix is equal to its column rank.

\begin{corollary}	\label{cor:rit}
Let $M$ be an $n\times m$ matrix with $n\ge m$, and assume $s_m(M)\ge \eps_0\sqrt{n}$ for some $\eps_0>0$. 
For any $\beta\in (0,1)$ there exists $I\subset [n]$ with $|I|=\lf (1-\beta)^2 m\rf$ such that
\[
s_{|I|}(M_{I, [m]}) \ge \beta \eps_0\sqrt{m}.
\]
\end{corollary}

\begin{remark}
The original Restricted Invertibility Theorem of Bourgain and Tzafriri \citep{BoTz:rit} only gives $|I|\ge cm$ and $s_{|I|}(M_{I, [m]}) \ge c \eps_0\sqrt{m}$ for some (small) absolute constant $c>0$, while it will be important for our purposes to be able to take $I$ of size close to $m$.
\end{remark}

\begin{proof}[Proof of Corollary \ref{cor:rit}]
By the singular value decomposition it suffices to consider $M$ of the form
$
M=U\Sigma
$
where $U$ is an $n\times m$ matrix with orthonormal columns and $\Sigma$ is an $m\times m$ diagonal matrix with entries bounded below by $\eps_0\sqrt{n}$.
Fix $\alpha \in (0,1)$.
Letting $v_1^*,\dots, v_n^*\in \C^m$ denote the rows of $U$, it follows from orthonormality that
\[
I_m = U^*U = \sum_{i=1}^n v_iv_i^*.
\]
Hence, we can apply Theorem \ref{thm:rit} to obtain a subset $I\subset[n]$ with $|I|= \lf (1-\beta)^2m\rf$
such that 
\[
s_{|I|}(U_{I, [m]})^2 = \lambda_{|I|}\bigg( \sum_{i\in I} v_iv_i^*\bigg) \ge \beta^2m/n.
\]
Now we have
\[
s_{|I|}(M_{I, m}) \ge s_{|I|}(U_{I, m}) s_m(\Sigma) \ge \beta \sqrt{\frac{m}{n}} \eps_0\sqrt{n} = \beta\eps_0\sqrt{m}.\qedhere
\]
\end{proof}

\begin{proof}[Proof of Lemma \ref{lem:entropy}]
Let $I\subset [n], J\subset[m]$, and write $V_{I,J} = \C^J\cap \ker(\M_{I,J})$.
Conditional on $\mF_{I,J}$, for $\eps>0$ we let $\Sigma_{I,J}(\eps)$ be an $\eps$-net of $S^{m-1}\cap V_{I,J}$. 
By Lemma \ref{lem:net} we may take
\begin{equation}	\label{sigmacard1}
|\Sigma_{I,J}(\eps)| = O(1/\eps)^{2\dim(V_{I,J})}.
\end{equation}

Let $\rho,\rho'\in (0,1)$, $K>0$ and $0<\theta< \frac nm$.
Fix $\beta\in (0,1)$ and $J\subset [m]$ with $|J|>\theta m$.
On $\event(\theta,\rho)^c$, for all $J_0\subset J$ with $|J_0|=\lf \theta m\rf$ we have
\[
s_{\lf \theta m\rf}(\M_{[n], J_0})\ge \rho K\sqrt{n}.
\]
By Corollary \ref{cor:rit} there exists $I\subset [n]$ with $|I|=\lf (1-\beta)^2 \lf \theta m\rf\rf$ such that 
\[
s_{|I|}(\M_{I, J_0})\ge \beta \rho K\sqrt{\lf \theta m\rf}.
\]
By the Cauchy interlacing law,
\begin{equation}	\label{sklb}
s_{|I|}(\M_{I, J}) \ge \beta \rho K\sqrt{\lf \theta m\rf}.
\end{equation}
In particular, the submatrix $(y_{ij})_{i\in I,j\in J}$ has full row-rank, which implies $\dim(V_{I,J}) = |J| - |I|$.
From \eqref{sigmacard1} we conclude
\begin{equation}	\label{sigmacard2}
|\Sigma_{I,J}(\eps)| = O(1/\eps)^{2(|J|-|I|)}
\end{equation}
for any $\eps>0$.

Now suppose there exists $u\in S^{m-1}\cap (\C^J)_{\rho'}$ such that
\begin{equation}	
\|\M u\|\le \rho'K\sqrt{n}.
\end{equation}
Letting $v'\in \C^J$ such that $\|u-v'\|\le \rho'$, and putting $v:=v'/\|v'\|\in S^J$, by the triangle inequality we have $\|u-v\|\le 2\rho'$ and
\begin{equation}	\label{Muub}
\|\M v\|\le \|\M u\|+\|\M \|\|u-v\| \le 3\rho'K\sqrt{n}.
\end{equation}
On the other hand,
\[
\|\M  v\| \ge \|\M_{I, [m]}v\| = \|\M_{I, [m]}(\id -P_{V_{I,J}})v\|
\]
where $P_{V_{I,J}}$ is the matrix for orthogonal projection to the subspace $V_{I,J}$. 
Applying \eqref{sklb}, 
\[
\|\M  v\| \ge \|(\id - P_{V_{I,J}})v\| \beta \rho K\sqrt{\lf \theta m\rf}.
\]
Together with \eqref{Muub} this implies that $v$ lies within distance
\begin{equation}
\frac{3\rho' \sqrt{n}}{\beta \rho \sqrt{\lf \theta m\rf}}  =\rho''/2
\end{equation}
of the subspace $V_{I,J}$. Since $v$ is a unit vector we have $\dist(v,S^{m-1}\cap V_{I,J})\le \rho''$ by the triangle inequality, and
\begin{align*}
\dist(u,\Sigma_{I,J}(\rho'')) &\le \|u-v\|+\rho''+ \dist(v, S^{m-1}\cap V_{I,J}) \le 2\rho' + 2\rho'' \le 3\rho''
\end{align*}
as desired (that $2\rho'\le \rho''$ follows from inspection of \eqref{def:rhodub}).

Now to prove \eqref{cont:entropy}, let $\theta'\in (\theta,1]$.
Intersecting with $\event(\theta,\rho)^c$ and applying the first part of the lemma,
\begin{align}
&\event(\theta,\rho)^c\wedge\event(\theta',\rho') \notag\\
&\quad = \mB(K)\wedge\event(\theta,\rho)^c\wedge\bigvee_{J\in {[m]\choose  \theta' m}} \Big\{\exists v\in (S^J)_{\rho'}: \|\M v\|\le \rho' K\sqrt{n}\Big\}\notag\\
&\quad\subset
\bigvee_{J\in {[m]\choose \theta' m}} \bigvee_{I\in {[n]\choose  (1-\beta)^2\lf \theta m\rf}} \bigg(  \good_{I,J}(\rho'')
\wedge\Big\{ \exists u \in \Sigma_{I,J}(\rho''): \|\M u\|\le 4\rho''K\sqrt{n} \Big\}\bigg)	
\end{align}
where in the last line we noted that for $v\in (S^J)_{\rho'}, u\in \Sigma_{I,J}(\rho'')$ such that $\|u-v\|\le 3\rho''$, we have 
\[
\|\M u\|\le \|\M v\|+3\rho'' K\sqrt{n} \le (\rho'+3\rho'')K\sqrt{n}\le 4\rho'' K\sqrt{n}.\qedhere
\]
\end{proof}

\subsection{Broadly connected profile: Proof of Proposition \ref{prop:slight}}

We will obtain Proposition \ref{prop:slight} from an iterative application of the following lemma:

\begin{lemma}[Incrementing compressibility: broadly connected profile]	\label{lem:increment_broad}
Let $\M=A\circ X+B$ be as in Definition \ref{def:profile} with $m\ge n/2$.
Assume $\xi$ has $\kappa$-controlled second moment for some $\kappa\ge1$, and that for some $\ha,\delta,\nu,\theta_1\in (0,1)$ we have
\begin{enumerate}[(1)]
\item $|\mN_{A(\ha)}(j)|\ge \delta n$ for all $j\in [m]$;
\item $|\mN_{A(\ha)}^{(\delta)}(J)|\ge \min((1+\nu) |J|,n)$ for all $J\subset[m]$ with $|J|\ge (\theta_1/2)m$.
\end{enumerate}
Let $K\ge1$, $\rho\in (0,1)$, and $\theta\in [\theta_1,1)$ such that $(1+\frac\nu2)\theta m<n$.
There exists $\rho'=\rho'(\kappa,\ha,\delta,\nu,\rho,\theta,K)>0$ such that
\begin{equation}
\pro{ \event(\theta,\rho)^c \wedge \event\Big( \Big(1+\frac{\nu}{10}\Big)\theta,\rho'\Big)} 
=O_{\kappa,\ha,\delta,\nu,\rho,\theta,K}(e^{-n}).
\end{equation}
\end{lemma}

\begin{proof}
We may assume $n$ is sufficiently large depending on $\kappa,\ha,\delta,\nu,\rho,\theta,K$.
Write $\theta'=\big(1+\frac{\nu}{10}\big)\theta$ and take $\beta=\frac{\nu}{10}$.
Let $\rho'>0$ to be taken sufficiently small depending on $\kappa,\ha,\delta,\nu,\rho,\theta,K$, and let $\rho''$ be as in \eqref{def:rhodub}.
Intersecting the right hand side of \eqref{cont:entropy} with $\event(\theta,\rho)^c$, we have
\begin{align}
&\event(\theta,\rho)^c \wedge \event(\theta',\rho') \subset \notag\\
&\quad
\bigvee_{J\in {[m]\choose  \theta' m }} \bigvee_{I\in {[n]\choose (1-\beta)^2\lf \theta m\rf}}   \good_{I,J}(\rho'')
\wedge\event(\theta,\rho)^c\wedge\Big\{ \exists u \in \Sigma_{I,J}(\rho''): \|\M u\|\le 4\rho''K\sqrt{n} \Big\}\notag\\
&
\subset\bigvee_{J\in {[m]\choose  \theta' m}} \bigvee_{I\in {[n]\choose (1-\beta)^2\lf \theta m\rf}}   \good_{I,J}(\rho'')
\wedge\Big\{ \exists u \in \Sigma_{I,J}(\rho'')\setminus \Comp(\theta,\rho): \|\M u\|\le 4\rho''K\sqrt{n} \Big\}\label{unionIJ}
\end{align}
where the second line follows by taking $\rho'$ small enough that $4\rho''<\rho$.

Fix $J\subset [m]$ and $I\subset[n]$ of sizes $\lf \theta'm\rf,\lf(1-\beta)^2\lf \theta m\rf\rf$, respectively, and condition on $\mF_{I,[n]}$ (recall the notation \eqref{def:sigmaalgebra}) to fix $\Sigma_{I,J}(\rho'')$.
Consider an arbitrary element $u\in \Sigma_{I,J}(\rho'')\setminus \Comp(\theta,\rho)$.
By Lemma \ref{lem:spread}, there is a set $L\subset[m]$ with $|L|\ge (1-\frac{\nu}{C_0^2})\theta m$
and 
\begin{equation}	\label{flat:broad}
\frac{\rho}{\sqrt{m}}\le |u_j|\le \frac{C_0}{\sqrt{\nu\theta m}}
\end{equation}
for all $j\in L$, where $C_0>0$ is an absolute constant to be taken sufficiently large.
For any $i\in \mN^{(\delta)}(L)$, we have
\begin{equation}
\|(u_L)^i\|^2 \ge \sum_{i\in L:\sig_{ij}\ge \ha} |u_j|^2\sig_{ij}^2\ge \frac{\rho^2}{m}\ha^2\delta |L| \ge \frac12\rho^2\ha^2\delta\theta=:\alpha^2
\end{equation}
where in the last inequality we took $C_0$ sufficiently large.
Hence,
\begin{equation}
|I_\alpha(u_L)|\ge |\mN^{(\delta)}(L)| \ge \min\big(n,(1+\nu)(1-\nu/C_0^2)\theta m\big) \ge \Big(1+\frac{\nu}2\Big)\theta m
\end{equation}
taking $C_0$ larger if necessary, where in the second inequality we used our assumption $\theta\ge \theta_1$, and in the third inequality we used our assumption $(1+\frac\nu2)\theta m<n$.

Fix $I_0\subset I_\alpha(u_L)\setminus I$ of size $n_0:=\lf (1+\frac\nu2)\theta m\rf -|I|$.
In particular,
\begin{align}	
\frac{\nu}{2}\theta m\le n_0 &\le \Big(1+\frac\nu2\Big)\theta m -(1-2\beta)\theta m \le \nu\theta m		\label{broad:i0}
\end{align}
and
\begin{align}
n_0+2|I|-2|J| &\ge \Big(1+\frac\nu2\Big)\theta m + (1-2\beta)\theta m- 2\Big(1+\frac{\nu}{10}\Big)\theta m -O(1) \notag\\
&=\frac1{10}\nu\theta m-O(1).	\label{i0i2j2}
\end{align}
by our choice of $\beta$.
By Lemma \ref{lem:fixed_improved},
\begin{equation}	\label{broad:fixedbd}
\pr_{I_0}\big( \|\M u\| \le 4\rho''K\sqrt{n}\big) \le O_\kappa\bigg(\frac1\alpha\bigg( \frac{\rho''K\sqrt{n}}{\sqrt{|I_0|}} + \frac{1}{\sqrt{\nu\theta m}}\bigg)\bigg)^{n_0}\le O_\kappa\bigg(\frac{\rho''K}{\alpha\theta^{1/2}}\bigg)^{n_0}
\end{equation}
where in the second inequality we applied the assumption $m\ge n/2$ and assumed that $n$ is sufficiently large that $\rho''\gg 1/K\sqrt{n}$ (it follows from \eqref{def:rhodub} and our assumption that $\rho'$ is independent of $n$ that $\rho''$ is bounded below independent of $n$).

Suppose that $\good_{I,J}(\rho'')$ holds.
Since the bound \eqref{broad:fixedbd} is uniform in the choice of $I_0$, we can undo the conditioning and apply the union bound over elements of $\Sigma_{I,J}(\rho'')\setminus \Comp(\theta,\rho)$ to find
\begin{align*}
&\pr\Big( \exists u\in \Sigma_{I,J}(\rho'')\setminus \Comp(\theta,\rho): \|\M u\|\le 4\rho''K\sqrt{n}\Big) \\
&\qquad\qquad\qquad\qquad\qquad\le O\bigg(\frac{1}{\rho''}\bigg)^{2(|J|-|I|)} O_\kappa\bigg(\frac{\rho''K}{\alpha\theta^{1/2}}\bigg)^{n_0}\\
&\qquad\qquad\qquad\qquad\qquad=O_\kappa\bigg(\frac{K}{\alpha\theta^{1/2}}\bigg)^{n_0}O(\rho'')^{n_0+2|I|-2|J|}\\
&\qquad\qquad\qquad\qquad\qquad=O_\kappa\bigg(\frac{K}{\alpha\theta^{1/2}}\bigg)^{\nu \theta m}O(\rho'')^{\frac1{10}\nu\theta m - O(1)}
\end{align*}
where in the last line we applied the bounds \eqref{broad:i0} and \eqref{i0i2j2}.
Since this is uniform in $I,J$, we can undo the conditioning on $\mF_{I,[n]}$ and apply \eqref{unionIJ} with another union bound over the choices of $I,J$ to obtain
\begin{equation}	\label{incr:pen}
\pro{\event(\theta,\rho)^c\wedge\event(\theta',\rho')} 
 \le 2^{m+n} O_\kappa\bigg(\frac{K}{\alpha\theta^{1/2}}\bigg)^{\nu \theta m}O\bigg(\frac{\rho'}{\nu\rho\theta^{1/2}}\bigg)^{\frac1{10}\nu\theta m - O(1)}
\end{equation}
where we have substituted the definition of $\rho''$.
The result now follows by taking $\rho'$ sufficiently small.
\end{proof}

Now we conclude the proof of Proposition \ref{prop:slight}.
From our assumptions it follows that for all $j\in[m]$ we have $\sum_{i=1}^n \sig_{ij}^2 \ge \delta\ha^2n$.
Together with our assumption $m\le 2n$, this means we can apply Lemma \ref{lem:high} to find that 
\begin{equation}	\label{fromhigh}
\pr(\event(\theta_0,\rho_0)) \le e^{-c_\kappa\delta\ha^2n}
\end{equation}
where $\theta_0= c_\kappa\delta\ha^2/\log(K/\delta\ha^2)$ and $\rho_0=c_\kappa\delta\ha^2/K$.

We may assume without loss of generality that $\nu\le \delta/2$.
For $l\ge1$ set $\theta_l=(1+\frac{\nu}{10})^l\theta_0$, and let $k$ be the smallest $l$ such that $\theta_l\ge \theta$.
We have
\[
\Big(1+\frac{\nu}{2}\Big)\theta_{k-1}m \le \Big(1+\frac{\nu}{2}\Big)\theta m \le \Big(1-\frac{\delta^2}{16}\Big)\min(m,n)<n.
\]
In particular, $(1+\nu/10)^k\theta_0\le (1+\nu/10)\theta\le 1$, so
\begin{equation}
k\le \frac{\log\frac1{\theta_0}}{\log\big(1+\frac{\nu}{10}\big)} \ll_{\kappa,\ha,\delta,\nu,K}1.
\end{equation}
Applying Lemma \ref{lem:increment_broad} inductively, we have that for every $1\le l\le k$ there is $\rho_l>0$ depending only on $\kappa,\ha,\delta,\nu$ and $K$ such that
\begin{equation}
\pro{ \event(\theta_l,\rho_l)\setminus \event(\theta_{l-1},\rho_{l-1})} = O_{\kappa,\ha,\delta,\nu,K}(e^{-n}).
\end{equation}
Together with \eqref{fromhigh} and the union bound,
\begin{align*}
\pro{ \event(\theta,\rho)} &\le \pro{\event(\theta_0,\rho_0)} + \sum_{l=1}^k \pro{\event(\theta_l,\rho_l)\setminus \event(\theta_{l-1},\rho_{l-1})}\\
&\le e^{-c_\kappa\delta\ha^2n} + O_{\kappa,\ha,\delta,\nu,K}(e^{-n}) = O_{\kappa,\ha,\delta,\nu,K}(e^{-c_\kappa\delta\ha^2n}).
\end{align*}

\subsection{General profile: Proof of Proposition \ref{prop:comp}}

For technical reasons (essentially due to the fact that we want to allow the operator norm to have arbitrary polynomial size) the anti-concentration argument from the previous section will not suffice here, and we will need the following substitute.
Roughly speaking, while previously we argued by isolating a large set of coordinates on which the vector $u$ is ``flat" (see \eqref{flat:broad}), here we will need to locate a set on which $u$ is \emph{very flat}, only fluctuating by a constant factor.
This is done by a simple dyadic decomposition of the range of $u$, which is responsible for the loss of a logarithmic factor in the probability bound.
A similar argument will be used in Section \ref{sec:super}.

\begin{lemma}[Anti-concentration for the image of an incompressible vector]	\label{lem:anti_incomp}
Let $M$ be as in Proposition \ref{prop:comp}.
Let $v\in \Incomp(\theta,\rho)$ for some $\theta,\rho\in (0,1)$, and fix $I_0\subset [n]$ with $|I_0|\le \frac14\sig_0^2n$.
Then for all $t\ge a_0\rho/\sqrt{m}$,
\begin{equation}
\sup_{w\in \C^n} \pr_{[n]\setminus I_0} \Big( \|Mv-w\| \le t \sqrt{n} \Big) = O_\kappa\left(\frac{t\log^{1/2}(\frac{\sqrt{m}}{\rho})}{\sig_0^2\rho \theta^{1/2}}\right)^{\frac14\sig_0^2n}.
\end{equation}
\end{lemma}

\begin{remark}
Proceeding as in the proof of Lemma \ref{lem:increment_broad} would yield 
\begin{equation}
\sup_{w\in \C^n} \pr_{[n]\setminus I_0} \Big( \|Mv-w\| \le t \sqrt{n} \Big) = O_\kappa\left(\frac{t}{\sig_0^2\rho \theta^{1/2}}\right)^{\frac14\sig_0^2n} \quad\quad \text{for all }\; t\ge \frac{a_0}{\sqrt{\theta m}}.
\end{equation}
The ability to take $t$ down to the scale $\sim \rho/\sqrt{m}$ will be crucial in the proof of Lemma \ref{lem:increment_gen} below.
\end{remark}

\begin{proof}
We begin by finding a set of indices on which $v$ varies by at most a factor of 2. 
For $k\ge 0$ let $L_k = \{j\in [m]: 2^{-(k+1)}<|v_j|\le 2^{-k}\}$. 
Since $v\in \Incomp(\theta,\rho)$, we have
\[
|L^+| := |\{j\in [m]: |v_j|\ge \rho/\sqrt{m}\}| \ge \theta m.
\]
Indeed, were this not the case then $v$ would be within distance $\rho$ of the vector $v_{L^+}$ whose support is smaller than $\theta m$, implying $v\in \Comp(\theta,\rho)$. 
Thus,
$
L^+\subset \bigcup_{k=0}^\ell L_k
$
for some $\ell \ll \log(\frac{\sqrt{m}}{\rho})$.
By the pigeonhole principle there exists $k^*\le \ell$ such that $L^*:= L_{k^*}$ satisfies
\begin{equation}	\label{LB:Lstar}
|L^*|\ge \frac{\theta n}{\ell} \gg \frac{\theta m}{ \log(\frac{\sqrt{m}}{\rho})}.
\end{equation}
Denote $I^*:= I_{\frac{a_0}{2}\|v_{L^*}\|}(v_{L^*})$.
By Lemma \ref{lem:goodrows0}, 
\begin{equation}	\label{LB:istar}
|I^*|\ge \frac12\sig_0^2 n.
\end{equation}

Fix $i\in I^*$.
By definition of $I^*$, 
\begin{equation}	\label{vi:2}
\| (v^i)_{L^*}\| \ge \frac{1}{2}a_0 \|v_{L^*}\| 
\end{equation}
and since $|v_j|\gg \rho/\sqrt{m}$ on $L^*$, 
\begin{equation}	\label{vstar:2}
\|v_{L^*}\| \gg \frac{\rho}{\sqrt{m}}|L^*|^{1/2}.
\end{equation}
Furthermore, since $a_{ij}\le 1$ for all $j\in [m]$ and $v$ varies by a factor at most 2 on $L^*$, 
\begin{equation}	\label{vi:infty}
\|(v^i)_{L^*}\|_\infty \le \|v_{L^*}\|_{\infty} \le 2\frac{\|v_{L^*}\|}{|L^*|^{1/2}}.
\end{equation}

Fix $w\in \C^n$ arbitrarily, and recall that $R_i$ denotes the $i$th row of $M$.
By Lemma \ref{lem:anti_improved} and the above estimates,
for all $t\ge 0$ we have 
\begin{align*}
\pr(|R_i\cdot v-w_i|\le t)
&\ll_\kappa \frac{t+ \|(v^i)_{L^*}\|_\infty }{\|(v^i)_{L^*}\|}\\
&\ll \frac{1}{a_0} \left( \frac{t}{\|v_{L^*}\|} + \frac{\|(v^i)_{L^*}\|_\infty}{\|v_{L^*}\|}\right)\\
&\ll \frac1{a_0} \left( \frac{t}{\rho} \left(\frac{m}{|L^*|}\right)^{1/2} + \frac1{|L^*|^{1/2}}\right)\\
&= \frac1{a_0} \left( \frac{m}{|L^*|}\right)^{1/2} \left( \frac{t}{\rho} + \frac1{\sqrt{m}}\right).
\end{align*}
By Lemma \ref{lem:tensorize}, 
\begin{align*}
\pr_{I^*\setminus I_0} \Big( \|Mv-w\| \le t |I^*\setminus I_0|^{1/2}\Big) 
&\le \pr_{I^*\setminus I_0} \Big( \sum_{i\in I^*\setminus I_0}|R_i\cdot v - w_i|^2  \le t^2 |I^*\setminus I_0|\Big) \\
&= O_\kappa\left(\frac{t\sqrt{m}}{a_0\rho|L^*|^{1/2}}\right)^{|I^*\setminus I_0|}
\end{align*}
for all $t\ge \rho/\sqrt{m}$.
Substituting the lower bounds \eqref{LB:Lstar}, \eqref{LB:istar} on $|L^*|$ and $|I^*|$ and our assumption $|I_0| \le \frac14\sig_0^2n$, 
\[
\pr_{I^*\setminus I_0} \bigg( \|Mv-w\| \le  \frac12ta_0\sqrt{n} \bigg) = O_\kappa\left(\frac{t\log^{1/2}(\frac{\sqrt{m}}{\rho})}{a_0\rho \theta^{1/2}}\right)^{\frac14\sig_0^2n}
\]
for all $t\ge \rho/\sqrt{m}$.
The result now follows by replacing $t$ with $2t/a_0$ as undoing the conditioning on the remaining rows in $[n]\setminus I_0$.
\end{proof}

Now we are ready to prove the analogue of Lemma \ref{lem:increment_broad} for general profiles.
Whereas in the broadly connected case we obtained control on vectors in $\Comp((1+\beta)\theta,\rho')$ after restricting to the event that we have control on $\Comp(\theta,\rho)$, for small $\beta>0$, here we will also need to assume control on $\Comp(\theta_0,\rho_0)$ for a fixed small $\theta_0$ at each step. 
The control on $\Comp(\theta,\rho)$ will be used to obtain a net of low cardinality using Lemma \ref{lem:entropy}, while the control on $\Comp(\theta_0,\rho_0)$ will be used to obtain good anti-concentration estimates using Lemma \ref{lem:anti_incomp}. 
(In the broadly connected case the control on $\Comp(\theta,\rho)$ was sufficient for both purposes.)

\begin{lemma}[Incrementing compressibility: general profile]	\label{lem:increment_gen}
Let $\M$ be as in Proposition \ref{prop:comp}, fix $\gamma>1/2$ and put $K=n^{\gamma-1/2}$. Let $\theta_0,\rho_0$ be as in Lemma \ref{lem:high}, and fix $\theta\in [\theta_0,c_0\sig_0^2]$, where $c_0$ is a sufficiently small constant (we may assume the constant $c$ in Lemma \ref{lem:high} is sufficiently small so that this interval is non-empty).
We have
\begin{equation}	\label{bound:increment_gen}
\pro{ \event(\theta_0,\rho_0)^c\wedge\event(\theta,\rho)^c \wedge \event( \theta+\beta a_0^2,\rho'\Big)} 
=O_{\gamma,\sig_0,\kappa}(e^{-n})
\end{equation}
for some $\rho' \gg_{\gamma,\sig_0,\kappa}  n^{-O(\gamma)}\rho$,
where we set
\begin{equation}	\label{set:beta}
\beta= c_1\min\left(1,\frac1{\gamma-1/2}\right)
\end{equation}
for a sufficiently small constant $c_1>0$.
\end{lemma}

\begin{proof}
Let $\rho'>0$ to be taken sufficiently small, and let $\rho''$ be as in \eqref{def:rhodub}.
We denote $\theta'=\theta +\beta a_0^2$.
Intersecting both sides of \eqref{cont:entropy} with $\event(\theta_0,\rho_0)^c$, we have
\begin{align}
&\event(\theta_0,\rho_0)^c\wedge\event(\theta,\rho)^c \wedge \event(\theta',\rho') \subset \notag\\
&
\bigvee_{J\in {[m]\choose  \theta' m }} \bigvee_{I\in {[n]\choose (1-\beta)^2\lf \theta m\rf}}  \good_{I,J}(\rho'')
\wedge\Big\{ \exists u \in \Sigma_{I,J}(\rho'')\setminus \Comp(\theta_0,\rho_0): \|\M u\|\le 4\rho''K\sqrt{n} \Big\}	\label{unionIJ:gen}
\end{align}
where we have assumed $\rho'$ is small enough that $4\rho''<\rho_0$.

Fix $J\subset[m]$ and $I\subset[n]$ of size $\lf \theta'm\rf$, $\lf (1-\beta)^2\lf \theta m\rf\rf$, respectively, and condition on $\mF_{I,[n]}$ to fix $\Sigma_{I,J}(\rho'')$.
Fix an arbitrary $u\in \Sigma_{I,J}(\rho'')\setminus \Comp(\theta_0,\rho_0)$. 
From Lemma \ref{lem:anti_incomp} we have
\begin{equation}
\pr_{[n]\setminus I} \Big( \|Mu\| \le 4\rho'' K\sqrt{n} \Big) = O_\kappa\left(\frac{\rho'' K\log^{1/2}(\frac{\sqrt{n}}{\rho_0})}{\sig_0^2\rho_0 \theta_0^{1/2}}\right)^{\frac14\sig_0^2n}
\end{equation}
provided
\begin{equation}	\label{assume:rhodub}
\rho'' \ge \frac{c a_0\rho_0}{K\sqrt{n}}
\end{equation}
for some small constant $c>0$ (note that we used our assumption $n/2\le m\le 2n$). 

Applying the union bound over the choices of $u\in \Sigma_{I,J}(\rho'')\setminus \Comp(\theta_0,\rho_0)$, on the event $\good_{I,J}(\rho'')$ we have
\begin{align*}
&\pr\Big( \exists u\in \Sigma_{I,J}(\rho'')\setminus \Comp(\theta_0,\rho_0): \, \|Mu\|\le 4\rho'' K\sqrt{n} \Big) \\
&\qquad\qquad\qquad \le O\left(\frac1{\rho''}\right)^{2(|J|-|I|)} O_\kappa\left(\frac{\rho'' K\log^{1/2}(\frac{\sqrt{n}}{\rho_0})}{\sig_0^2\rho_0 \theta_0^{1/2}}\right)^{\frac14\sig_0^2n}\\
&\qquad\qquad\qquad= O\left(\frac1{\rho''}\right)^{2(|J|-|I|)} O_{\kappa,a_0}\left(\rho''K^2 \log(K\sqrt{n}) \right)^{\frac14\sig_0^2n}
\end{align*}
where in the second line we substituted the expressions for $\rho_0,\theta_0$ from Lemma \ref{lem:high}. 
Denoting $\eps = \rho'' K^2$, the above bound rearranges to 
\begin{equation}	\label{comp:rearrange}
O_{\kappa,a_0}(\log n)^n n^{O(\gamma)} n^{O(\gamma-1/2)(|J|-|I|)} \eps^{\frac14\sig_0^2 n - 2(|J| - |I|)}.
\end{equation}
We can bound
\begin{align*}
|J|-|I| 
&= \theta m + \beta \sig_0^2m - (1-\beta)^2 \theta m + O(1) \le \beta \sig_0^2m +2\beta \theta m + O(1)\\
& = O(\beta \sig_0^2 m) + O(1)
\end{align*}
where we used our assumption that $\theta\le c_0\sig_0^2$. 
In particular, $|J|-|I| \le \frac18 \sig_0^2 n+O(1)$ if the constant $c_1$ in \eqref{set:beta} is sufficiently small, and \eqref{comp:rearrange} is bounded by
\begin{equation}
O_{\kappa,a_0}(\log n)^n n^{O(\gamma)}n^{O(\gamma-1/2) \beta a_0^2 m} \eps^{\frac18\sig_0^2n - O(1) }.
\end{equation}
Applying the union bound over the choices of $I,J$ in \eqref{unionIJ:gen}, which incurs a harmless factor of $2^{m+n} = O(1)^n$, and substituting the expression \eqref{set:beta} for $\beta$ we have
\begin{equation}
\pro{ \event(\theta_0,\rho_0)^c\wedge\event(\theta,\rho)^c \wedge \event(\theta+\beta a_0^2,\rho'\Big)} = O_{\kappa,a_0}(\log n)^n n^{O(\gamma)}\eps^{-O(1)}(n^{O(c_1)} \eps^{1/8})^{a_0^2n}.
\end{equation}
It only remains to check that we can take $\eps$ sufficiently small to obtain \eqref{bound:increment_gen}.
From \eqref{assume:rhodub} we are constrained to take
\[
\eps = \rho'' K^2 \ge \frac{c \sig_0\rho_0K}{\sqrt{n}} = \frac{c' a_0^3 }{\sqrt{n}}
\]
for some constant $c'\in (0,1)$ sufficiently small.
Taking $\eps = a_0^3/\sqrt{n}$ and $c_1$ sufficiently small we have
\begin{equation}
\pro{ \event(\theta_0,\rho_0)^c\wedge\event(\theta,\rho)^c \wedge \event(\theta+\beta a_0^2,\rho'\Big)} \le O_{\kappa,\sig_0}(1)^n n^{O(\gamma)} n^{-.01 \sig_0^2 n}
\end{equation}
which yields \eqref{bound:increment_gen} as desired.
With this choice of $\eps$,
\[
\rho' \gg \rho'' \beta \rho \theta\ge \rho'' \beta \rho \theta_0 \gg_{\kappa,\sig_0,\gamma} \rho n^{-2\gamma + 1/2-o(1)}
\]
as desired (recall that $\theta_0\gg_\kappa \sig_0^2/\log(K/a_0) \gg_{\gamma,\sig_0,\kappa} 1/\log n$). 
\end{proof}

Now we conclude the proof of Proposition \ref{prop:comp}. 
Since the event $\mB(K)$ is monotone under increasing $K$, by perturbing $\gamma$ and assuming $n$ is sufficiently large we may take $K=n^{\gamma-1/2}$ with $\gamma>1/2$.
Let $\rho_0,\theta_0$ be as in Lemma \ref{lem:high}, and for $l\ge 1$ we let $\theta_l = \theta_0+ l\beta a_0^2$ with $\beta=\beta(\gamma)$ as in \eqref{set:beta}. 
By Lemma \ref{lem:increment_gen} we can inductively define a sequence $\rho_l$ such that for each $l\ge 1$ such that $\theta_l \le c_0\sig_0^2$, 
\[
\rho_l \gg_{\gamma,\sig_0,\kappa} n^{-O(\gamma)}\rho_{l-1}
\]
and
\[
\pro{ \event(\theta_0,\rho_0)^c\wedge \event(\theta_{l-1},\rho_{l-1})^c\wedge \event(\theta_l,\rho_l)} = O_{\gamma,\sig_0,\kappa}(e^{-n}).
\]
Applying the union bound, for some $k=O(\gamma)$ we have
\begin{align*}
\pro{ \event(c_0\sig_0^2,\rho) } 
&\le \pro{ \event(\theta_0,\rho_0)} + \sum_{l=1}^{k} 
\pro{ \event(\theta_0,\rho_0)^c\wedge \event(\theta_{l-1},\rho_{l-1})^c\wedge \event(\theta_l,\rho_l)}\\
& \le e^{-c_\kappa\sig_0^2n} + O_{\gamma,\sig_0,\kappa} (e^{-n})\\
& = O_{\gamma,\sig_0,\kappa} (e^{-c_\kappa\sig_0^2n})
\end{align*}
and 
$
\rho \gg_{\gamma,\sig_0,\kappa} n^{-O(\gamma^2)}.
$
This concludes the proof of Proposition \ref{prop:comp}.

\section{Invertibility from connectivity: Incompressible vectors}		\label{sec:incomp}

In this section we conclude the proofs of Theorems \ref{thm:broad} and \ref{thm:super} by bounding the event that $\|Mu\|$ is small for some incompressible vector $u$ (recall the terminology from Section \ref{sec:net}).
We follow the (by now standard) approach of reducing to the event that a fixed row $R_i$ of $M$ lies close to the span of the remaining rows, an idea which goes back to the work of Koml\'os on the singularity probability for Bernoulli matrices \citep{Komlos67,Komlos68,Komlos77}. 
This can in turn be controlled by the event that a random walk $R_i\cdot v$ concentrates near a particular point, where $v$ is a fixed unit vector in the orthocomplement of the remaining rows. 
Independence of the rows allows us to condition on $v$, and our results from the previous section allow us to argue that $v$ is incompressible. 

For the case that the entries of $R_i$ have variances uniformly bounded below, we could then complete the proof by applying the anti-concentration estimate of Lemma \ref{lem:anti_improved}.
In the present setting, however, a proportion $1-\delta$ of the entries of $R_i$ may have zero variance. 
For the case of broadly connected profile we follow the argument of Rudelson and Zeitouni \citep{RuZe} and use Proposition \ref{prop:slight} to show $v$ has essential support of size $(1-\delta/2)n$, and hence has non-trivial overlap with the support of $R_i$.

For the case of a super-regular profile, Proposition \ref{prop:comp} only gives that $v$ has essential support of size $\gg \delta \ha^2$. 
In Lemma \ref{lem:overlap} we make use of a double counting argument to show that if we choose the row $R_i$ at random, on average it will have good overlap with the corresponding normal vector $v^{(i)}$ (which also depends on $i$). 
Here is where we make crucial use of the super-regularity hypothesis on $A$.
Lemma \ref{lem:overlap} is a natural extension of a double counting argument used by Koml\'os in his work on the singularity probability for Bernoulli matrices, and which was applied to bound the smallest singular value of iid matrices by Rudelson and Vershynin in \citep{RuVe:ilo}.
We were also inspired by
a similar refinement of the double counting argument from the recent paper \citep{LLTTY} on the singularity probability for adjacency matrices of random regular digraphs.

\subsection{Proof of Theorem \ref{thm:broad}}	\label{sec:broad}

By Lemma \ref{lem:wlog.kappa} and multiplying $X$ and $B$ by a phase (which does not affect our hypotheses) we may assume that $\xi$ has $O(\kappa_0)$-controlled second moment.
Fix $K\ge 1$, and let $\rho=\rho(\kappa,\ha,\delta,\nu,K)$ be as in Proposition \ref{prop:slight}. 
We may assume $n$ is sufficiently large depending on $\kappa,\ha,\delta,\nu,K$.
For the remainder of the proof we restrict to the event $\mB(K) = \{\|M\|\le K\sqrt{n}\}$.

For $j\in [n]$ let $M^{(i)}$ denote the $n-1\times n$ matrix obtained by removing the $i$th row from $M$.
Define the good event
\begin{equation}
\good = \Big\{ \forall i\in [n], \forall u\in \Comp(1-\delta/2,\rho), \; \|u^*M\|,\|M^{(i)}u\| >\rho K\sqrt{n}\Big\}.
\end{equation}
Applying Proposition \ref{prop:slight} to $M^*$ and $M^{(i)}$ for each $i\in[n]$ (using our restriction to $\mB(K)$) and the union bound we have
\begin{equation}
\pr(\good) = 1- O_{\kappa,\ha,\delta,\nu,K}(ne^{-c_\kappa\delta \ha^2 n}) = 1- O_{\ha,\delta,\nu,K}(e^{-c_\kappa\delta \ha^2 n})
\end{equation}
adjusting $c_\kappa$ slightly. 
Let $t\le 1$, and define the event
\begin{equation}
\event(t) = \good\wedge \big\{ \exists u\in \Incomp(1/10,\rho): \|u^*M\|\le t/\sqrt{n}\big\}.
\end{equation}
For $n$ sufficiently large (larger than $1/\rho K$) it suffices to show 
\begin{equation}	\label{goal:broad1}
\pr(\event(t)) \ll_{\kappa,\ha,\delta,\nu,K} t + n^{-1/2}.
\end{equation}

Recalling that $R_i$ denotes the $i$th row of $M$, we denote
\begin{equation}
R_{-i} = \Span(R_j: j\in [n]\setminus \{i\})
\end{equation}
and let \begin{equation}
\event_i(t) = \good \wedge \{\dist(R_i,R_{-i}) \le t/\rho\}.
\end{equation}
We now use a double counting argument of Rudelson and Vershynin from \citep{RuVe:ilo} to control $\event(t)$ in terms of the events $\event_i(t)$. 
Suppose that $\event(t)$ holds, and let $u\in \Incomp(1/10, \rho)$ such that $\|u^*M\|\le t/\sqrt{n}$. 
Then we must have $|u_i| \ge \rho/\sqrt{n}$ for at least $n/10$ elements $i\in[n]$. 
For each such $i$ we have
\[
\frac{t}{\sqrt{n}} \ge \|u^*M\| = \bigg\| \sum_{j=1}^n \overline{u_j} R_j\bigg\| 
\ge \bigg\| P_{R_{-i}^\perp}  \sum_{j=1}^n \overline{u_j} R_j\bigg\| 
= |u_i| \left\| P_{R_{-i}^\perp} R_i\right\| 
\ge \frac{\rho}{\sqrt{n}} \dist(R_i,R_{-i})
\]
where we denote by $P_W$ the orthogonal projection to a subspace $W$.
Thus, on $\event(t)$ we have that $\event_i(t)$ holds for at least $n/10$ values of $i\in[n]$, so by double counting,
\begin{equation}
\pro{ \event(t)} \le \frac{10}{n} \sum_{i=1}^n \pro{\event_i(t)}.
\end{equation}
Now it suffices to show that for arbitrary fixed $i\in [n]$,
\begin{equation}	\label{goal:broad2}
\pr(\event_i(t)) \ll_{\kappa,\ha,\delta,\nu,K} t + n^{-1/2}.
\end{equation}

Fix $i\in [n]$ and condition on $\{R_j: j\in [n]\setminus \{i\}\}$.
Draw a unit vector $u\in R_{-i}^\perp$ independent of $R_i$, according to Haar measure (say). 
Since $\dist(R_i,R_{-i}) \le |R_i\cdot u|$, it suffices to show
\begin{equation}	\label{goal:broad3}
\pro{|R_i\cdot u|\le t/\rho } \ll_{\kappa,\ha,\delta,\nu,K} t + n^{-1/2}.
\end{equation}
Since $u\in \ker(M^{(i)})$, on $\good$ we have that $u\in \Incomp(1-\frac\delta2,\rho)$.
By Lemma \ref{lem:spread} there exists $L\subset [n]$ of size $|L| \ge (1-\frac34\delta)n$ such that
\[
\frac{\rho}{\sqrt{n}} \le |u_j| \le \frac{10}{\sqrt{\delta n}}
\]
for all $j\in L$.
By assumption we have $|\mN_{A(\ha)}(i)| = |\{ j\in [n]: \sig_{ij} \ge \ha\}| \ge \delta n$, 
so letting $J=\mN_{A(\ha)}(i)\cap L$ we have $|J|\ge \delta n/4$. 
Denoting $v= (u^i)_J = (\sig_{ij} u_j 1_{j\in J})_j$, we have
\[
\|v\|^2 = \sum_{j\in J} \sig_{ij}^2 |u_j|^2 \ge |J|\ha^2 \rho^2/n \ge \delta\ha^2\rho^2/4
\]
and
\[
\|v\|_\infty \le \|u_J\|_{\infty} \le \frac{10}{\sqrt{\delta n}}
\]
(recall that $a_{ij}\le 1$ for all $i,j\in [n]$).
Conditioning on $u$ and $\{\xi_{ij}\}_{j\notin J}$, we apply Lemma \ref{lem:anti_improved} to conclude
\begin{align*}
\pro{ |R_i\cdot u| \le t/\rho} \ll_\kappa \frac{1}{\|v\|}\left( \frac{t}{\rho} + \|v\|_\infty\right) \ll \frac{1}{\rho\ha\delta^{1/2}} \left(\frac{t}\rho + \frac{1}{\sqrt{\delta n}}\right)
\end{align*}
which gives \eqref{goal:broad3} as desired.

\subsection{Proof of Theorem \ref{thm:super}}	\label{sec:super}

By Lemma \ref{lem:wlog.kappa} and multiplying $X$ and $B$ by a phase (which does not affect our hypotheses) we may assume that $\xi$ has $\kappa=O(\kappa_0)$-controlled second moment.
Fix $\gamma\ge1/2$ and let $K=O(n^{\gamma-1/2})$. 
We will show that for all $\tau\ge 0$,
\begin{equation}	\label{bound:super2}
\pro{ s_n(M) \le \frac{\tau}{\sqrt{n}}\,, \; \|M\|\le K\sqrt{n}} \ll_{\gamma,\ha,\delta,\kappa} n^{O(\gamma^2)}\tau + \sqrt{\frac{\log n}{n}}.
\end{equation}
For the remainder of the proof we restrict to the boundedness event
\begin{equation}
\mB(K) = \{\|M\|\le K\sqrt{n}\}.
\end{equation}

By the assumption that $A(\ha)$ is $(\delta,\eps)$-super-regular we have
\[
\sum_{i=1}^n \sig_{ij}^2 \ge \delta \ha^2n
\]
for all $j\in [n]$. 
Let $\sig_0 = \delta^{1/2}\ha$, and let $\rho=\rho(\gamma,\sig_0,\kappa n)$ and $c_0$ be as in Proposition \ref{prop:comp}.
In particular,
\begin{equation}	\label{incomp:rholb}
\rho \gg_{\gamma,\delta,\ha}n^{-O(\gamma^2)}.
\end{equation}
Denoting $\theta=c_0\delta\ha^2$, for $\tau>0$ we define the good event
\begin{equation}
\good(\tau) = \Big\{ \forall u\in \Comp(\theta,\rho), \, \|Mu\|,\|u^*M\| > \tau/\sqrt{n}\Big\}.
\end{equation}
Applying Proposition \ref{prop:comp} to $M$ and $M^*$, along with the union bound, we have
\begin{equation}	\label{LB:goodtau}
\pro{\good(\tau)} = 1-O_{\gamma,\delta,\ha,\kappa}(e^{-c_\kappa\delta\ha^2n})
\end{equation}
as long as $\tau\le \rho K n$. 

Let $0<\tau\le 1$ to be chosen later.
Recalling our notation $M^{(i)}$ from Section \ref{sec:broad}, we define the sets
\begin{equation}
S_i(\tau) = \left\{ u\in S^{n-1}: \|M^{(i)}u\|\le \frac{\tau}{\sqrt{n}}\right\}.
\end{equation}
Informally, for small $\tau$ this is the set of unit almost-normal vectors to the subspace $R_{-i}$ spanned by the rows of $M^{(i)}$. 
In Lemma \ref{lem:overlap} below we reduce our task to bounding the probability that a row $R_i$ is nearly orthogonal to a vector $u^{(i)}\in S_i(\tau)$ that is independent of $R_i$, and also has many large coordinates in the support of $R_i$.
The reduction uses the super-regularity hypothesis together with a careful averaging argument.
It turns out that for this argument to work it is important to consider almost-normal vectors rather than normal vectors (as in the proof of Theorem \ref{thm:broad}).

Writing $\mN(i)= \mN_{A(\ha)}(i)$, we define the \emph{good overlap events}
\begin{equation}
\mO_i(\tau) = \big\{ \exists u\in S_i(\tau): |\mN(i) \cap L^+(u,\rho)|\ge \delta \theta n\big\}
\end{equation}
where 
\begin{equation}	\label{super:Lplus}
L^+(u)= \{j\in [n]: |u_j|\ge \rho/\sqrt{n}\}.
\end{equation}
On $\mO_i(\tau)$ we fix a vector $u^{(i)} = u^{(i)}(M^{(i)},\tau)\in S_i(\tau)$, chosen measurably with respect to $M^{(i)}$, satisfying $|\mN(i)\cap L^+(u,\rho)|\ge \delta \theta n$.

\begin{lemma}[Good overlap on average]	\label{lem:overlap}
Recall the parameter $\eps$ from our super-regularity hypothesis (cf.\ Definition \ref{def:super}), and assume $\eps\le \theta/2$. 
Then
\begin{equation}
\pro{\good(\tau)\wedge\Big\{ s_n(M)\le \frac{\tau}{\sqrt{n}}\Big\}} \le \frac{2}{\theta n} \sum_{i=1}^n \pro{ \mO_i(\tau)\wedge \bigg\{ |R_i\cdot u^{(i)}|\le \frac{2\tau}{\rho}\bigg\}}.
\end{equation}
\end{lemma}

\begin{proof}
Suppose $\good(\tau)\wedge\{s_n(M)\le \tau/\sqrt{n}\}$ holds. 
Then there exist $u,v\in S^{n-1}$ such that $\|Mu\|, \|M^*v\|\le \tau/\sqrt{n}$.
By our restriction to $\good(\tau)$ we must have $u,v\in \Incomp(\theta,\rho)$.
With notation as in \eqref{super:Lplus} we have $|L^+(u)|, |L^+(v)|\ge \theta n$. 
In particular, $|L^+(u)|\ge \eps n$, so
\begin{equation}
|\mN(i) \cap L^+(u)| \ge \delta |L^+(u)| \ge \delta \theta n
\end{equation}
for at least $(1-\eps)n$ elements $i\in [n]$.
Indeed, otherwise we would have
\[
e_{A(\ha)}(I,L^+(u)) = \sum_{i\in I} |\mN(i)\cap L^+(u)| <\delta |I||L^+(u)|
\]
for some $I\subset[n]$ with $|I|>\eps n$, which contradicts our assumption that $A(\ha)$ is $(\delta,\eps)$-super-regular. 
Since 
$
\|M^{(i)}u\|\le \|Mu\|\le \frac{\tau}{\sqrt{n}}
$
for all $i\in [n]$, we have that $u\in S_i(\tau)$ for all $i\in [n]$. 
Thus, 
\begin{equation}	\label{super:goodvi}
\left| \big\{ i\in L^+(v): \mO_i(\tau) \mbox{ holds}\big\}\right| \ge \theta n -\eps n\ge \theta n/2.
\end{equation}

Fix $i\in L^+(v)$ such that $\mO_i(\tau)$ holds. 
We have
\begin{align*}
\frac{\tau}{\sqrt{n}} \ge \|v^*M\|\ge |v^*Mu^{(i)}| \ge |v_i| |R_i\cdot u^{(i)}| - \bigg|\sum_{j\ne i} \overline{v_j} R_j\cdot u^{(i)}\bigg|.
\end{align*}
The first term on the right hand side is bounded below by $\frac{\rho}{\sqrt{n}}|R_i\cdot u^{(i)}|$ since $i\in L^+(v)$.
By Cauchy--Schwarz the second term is bounded above by $\|M^{(i)} u^{(i)}\|\le \tau/\sqrt{n}$, since $u^{(i)}\in S_i(\tau)$.
Rearranging we conclude
$
|R_i\cdot u^{(i)}|\le 2\tau/\rho
$
for all $i\in L^+(v)$ such that $\mO_i(\tau)$ holds.
Letting $\mE_i(t)=\{|R_i\cdot u^{(i)}|\le t\}$,
we have shown that on the event $\good(\tau)\wedge\{s_n(M)\le \tau/\sqrt{n}\}$, the event
$\mO_i(\tau)\wedge \mE_i(2\tau/\rho)$ holds for at least $\theta n/2$ values of $i\in [n]$ (from \eqref{super:goodvi}). It follows that
\[
\sum_{i=1}^n \un(\mO_i(\tau)\wedge \mE_i(2\tau/\rho)) \ge \frac{\theta n}{2} \un(\good(\tau)\wedge\{s_n(M)\le \tau/\sqrt{n}\}).
\]
Taking expectations on each side and rearranging yields the claim.
\end{proof}

Fix $i\in [n]$ arbitrarily, and suppose that $\mO_i(\tau)$ holds.
We condition on the rows $\{R_{j}\}_{j\in [n]\setminus \{i\}}$ to fix $u^{(i)}$. 
We begin by finding a large set on which $u^{(i)}$ is flat, following a similar dyadic pigeonholing argument as in the proof of Lemma \ref{lem:anti_incomp}.
Letting $L_k = \{j\in [n]: 2^{-(k+1)} < |u^{(i)}_j| \le 2^{-k}$, since
\[
\delta \theta n \le |\mN(i) \cap L^+(u^{(i)})| \le  \bigg| \bigcup_{k=0}^{\ell} \mN(i)\cap L_k \bigg|
\]
for some $\ell \ll \log(\sqrt{n}/\rho)$, by the pigeonhole principle there exists $k^*\le \ell$ such that $J:= \mN(i)\cap L_{k^*}$ satisfies 
\begin{equation}	\label{incomp:Jbound}
|J| \ge \delta \theta n/\ell \gg \frac{\delta \theta n}{\log(\sqrt{n}/\rho)}.
\end{equation}
Let us denote $v= (\sig_{ij} u^{(i)}_j 1_{j\in J})_j$. 
Since $\sig_{ij} \ge \ha$ for $j\in \mN(i)$ and $|u_j^{(i)}|\gg \rho/\sqrt{n}$ for $j\in L_{k^*}$, 
\begin{equation}	\label{v:2}
\|v\| \ge \ha \|(u^{(i)})_J\| \gg \ha \rho (|J|/n)^{1/2}
\end{equation}
and since $u^{(i)}$ varies by at most a factor of $2$ on $J$, 
\begin{equation}	\label{v:infty}
\|v\|_\infty \le \|u^{(i)}1_{J}\|_\infty \le 2\|u^{(i)}\|/|J|^{1/2}.
\end{equation}
By further conditioning on the variables $\{\xi_{ij}\}_{j\notin J}$ and applying Lemma \ref{lem:anti_improved} along with the estimates \eqref{v:2}, \eqref{v:infty} we have
\begin{align*}
\pro{ |R_i\cdot u^{(i)}|\le 2\tau/\rho} 
&\ll_\kappa \frac{\tau/\rho + \|v\|_{\infty}}{\|v\|}\\
&\ll \frac{1}{\ha}\left(\frac{\tau/\rho}{\rho(|J|/n)^{1/2}} + \frac{1}{|J|^{1/2}}\right)\\
&= \frac{1}{\ha} \left(\frac{n}{|J|}\right)^{1/2} \left(\frac{\tau}{\rho^2} + \frac1{\sqrt{n}}\right). 
\end{align*}
Inserting the bound \eqref{incomp:Jbound} and undoing all of the conditioning, we have shown
\[
\pro{ \mO_i(\tau)\wedge \bigg\{ |R_i\cdot u^{(i)}|\le \frac{2\tau}{\rho}\bigg\}}
\ll_\kappa \frac{1}{\ha\sqrt{\delta \theta}} \left(\frac{\tau}{\rho^2} + \frac1{\sqrt{n}}\right) \log^{1/2}(\sqrt{n}/\rho).
\]
Since the right hand side is uniform in $i$, applying Lemma \ref{lem:overlap} (taking $c_1=c_0/2$) and substituting the expression for $\theta$ we have
\begin{equation}
\pro{\good(\tau)\wedge\Big\{ s_n(M)\le \frac{\tau}{\sqrt{n}}\Big\}} 
\ll_\kappa \frac{1}{\ha^4\delta^2} \left(\frac{\tau}{\rho^2} + \frac1{\sqrt{n}}\right) \log^{1/2}(\sqrt{n}/\rho)
\end{equation}
for all $\tau\ge0$ (note that this bound is only nontrivial when $\tau\le \rho^2$, in which case our constraint $\tau\le \rho Kn$ from \eqref{LB:goodtau} holds).
The bound \eqref{bound:super2} now follows by substituting the lower bound \eqref{incomp:rholb} on $\rho$ and the bound \eqref{LB:goodtau} on $\good(\tau)^c$ (which is dominated by the $O(n^{-1/2}\log^{1/2}n)$ term).
This concludes the proof of Theorem \ref{thm:super}.

\section{Invertibility under diagonal perturbation: Proof of main theorem}	\label{sec:diag}

In this final section we prove Theorem \ref{thm:main}. See Section \ref{sec:ideas} for a high level discussion of the main ideas. 
In Sections \ref{sec:tools} and \ref{sec:op} we collect the main tools of the proof: the regularity lemma, the Schur complement bound, and bounds on the operator norm of random matrices.
In Section \ref{sec:decomp} we apply the regularity lemma to decompose the standard deviation profile $A$ into a bounded number of submatrices enjoying various properties. 
In Section \ref{sec:highlevel} we apply the decomposition to prove Theorem \ref{thm:main}, on two technical lemmas, and in the final sections we prove these lemmas.

\subsection{Preliminary Tools}		\label{sec:tools}

We begin by stating a version of the regularity lemma suitable for our purposes.
Recall that in Theorem \ref{thm:broad} we associated the standard deviation profile $A$ with a bipartite graph. 
Here it will be more convenient to associate $A$ with a directed graph.
That is, to a non-negative square matrix $A=(\sig_{ij})_{1\le i,j\le n}$ we associate a directed graph $\Gamma_A$ on vertex set $[n]$ having an edge $i\rightarrow j$ when $\sig_{ij}>0$ (note that we allow $\Gamma_A$ to have self-loops, though the diagonal of $A$ will have a negligible effect on our arguments).
The notation \eqref{def:nbhd}--\eqref{def:edges} extends to this setting.
Additionally, we denote the \emph{density} of the pair $(I,J)$
\[
\rho_A(I,J):= \frac{e_A(I,J)}{|I||J|}.
\]

\begin{definition}[Regular pair]	\label{def:regular.pair}
Let $A$ be an $n\times n$ matrix with non-negative entries. 
For $\eps>0$, we say that a pair of vertex subsets $I,J\subset [n]$ is \emph{$\eps$-regular for $A$} if for every $I'\subset I, J'\subset J$ satisfying 
\[
|I'|> \eps |I|, \quad |J'|> \eps|J|
\]
we have
\[
|\rho_A(I',J')-\rho_A(I,J)|<\eps.
\]
\end{definition}

The following is a version of the regularity lemma for directed graphs which follows quickly from a stronger result of Alon and Shapira \cite[Lemma 3.1]{AlSh:testing}.
Note that \cite[Lemma 3.1]{AlSh:testing} is stated for directed graphs without loops, which in the present setting means that it only applies to matrices $A$ with diagonal entries equal to zero. 
However, Lemma \ref{lem:regularity} follows from applying \cite[Lemma 3.1]{AlSh:testing} to the matrix $A'$ formed be setting the diagonal entries of $A$ to zero, and noting that the diagonal has a negligible impact on the edge densities $\rho_A(I,J)$ when $|I|,|J|\gg n$. 

\begin{lemma}[Regularity Lemma] 	\label{lem:regularity}
Let $\eps>0$.
There exists $m_0\in \N$ with $\eps^{-1}\le m_0\ll_{\eps}1$ such that for all $n$ sufficiently large depending on $\eps$, for every $n\times n$ non-negative matrix $A$ there is a partition of $[n]$ into $m_0+1$ sets $I_0,I_1,\dots, I_{m_0}$ with the following properties:
\begin{enumerate}[(1)]
\item $|I_0|<\eps n$;
\item $|I_1|=|I_2|=\cdots =|I_{m_0}|$;
\item all but at most $\eps m_0^2$ of the pairs $(I_k,I_l)$ are $\eps$-regular for $A$. 
\end{enumerate}
\end{lemma}

\begin{remark}
The dependence on $\eps$ of the bound $m_0\le O_{\eps}(1)$ is very bad: a tower of exponentials of height $O(\eps^{-C})$. 
Indeed, as in Szemer\'edi's proof for the setting of bipartite graphs \citep{Szemeredi:lemma}, the proof in \citep{AlSh:testing} gives such a bound with $C=5$.
It was shown by Gowers that for undirected graphs one cannot do better than $C=1/16$ in general \citep{Gowers:towers}. 
As remarked in \citep{AlSh:testing}, his argument carries over to give a similar result for directed graphs. 
\end{remark}

We will apply this in Section \ref{sec:decomp} to partition the standard deviation profile into a bounded number of manageable submatrices. 
The following elementary fact from linear algebra will be used to lift the invertibility properties obtained for these submatrices back to the whole matrix.

\begin{lemma}[Schur complement bound]		\label{lem:schur}
Let $M\in \mM_{N+n}(\C)$, which we write in block form as
$$M=\begin{pmatrix} A & B\\ C&D\end{pmatrix}$$
for $A\in \mM_{N}(\C), B\in \mM_{N,n}(\C), C\in \mM_{n,N}(\C), D\in \mM_n(\C)$.
Assume that $D$ is invertible.
Then
\begin{equation}	\label{bound:schur}
s_{N+n}(M) \ge \bigg(1+ \frac{\|B\|}{s_n(D)}\bigg)^{-1} \bigg(1+ \frac{\|C\|}{s_n(D)}\bigg)^{-1}  \min\Big( s_n(D), \;s_N(A-BD^{-1}C)\Big) .
\end{equation}
\end{lemma}

\begin{proof}
From the identity
\[
\begin{pmatrix} A & B\\ C&D\end{pmatrix}= 
\begin{pmatrix} I_N & BD^{-1} \\ 0&I_n\end{pmatrix}
\begin{pmatrix} A-BD^{-1}C & 0\\ 0&D\end{pmatrix}
\begin{pmatrix} I_N & 0\\ D^{-1}C&I_n\end{pmatrix}
\]
we have
\begin{align*}
\begin{pmatrix} A & B\\ C&D\end{pmatrix}^{-1} &= 
\begin{pmatrix} I_N & 0\\ -D^{-1}C&I_n\end{pmatrix}
\begin{pmatrix} (A-BD^{-1}C)^{-1} & 0\\ 0&D^{-1}\end{pmatrix}
\begin{pmatrix} I_N & -BD^{-1} \\ 0&I_n\end{pmatrix}.
\end{align*}
We can use the triangle inequality to bound the operator norm of the first and third matrices on the right hand side by $1+\|BD^{-1}\|$ and $1+\|CD^{-1}\|$, respectively. 
Now by sub-multiplicativity of the operator norm,
\begin{align*}
\|M^{-1}\| &\le (1+\|BD^{-1}\|)(1+\|D^{-1}C\|) \max(\|(A-BD^{-1}C)^{-1}\|,\|D^{-1}\|) \\
&\le \bigg( 1+ \frac{\|B\|}{s_n(D)}\bigg)\bigg( 1+ \frac{\|C\|}{s_n(D)}\bigg) \max(\|(A-BD^{-1}C)^{-1}\|,\|D^{-1}\|).
\end{align*}
The bound \eqref{bound:schur} follows after taking reciprocals. 
\end{proof}

\subsection{Control on the operator norm}		\label{sec:op}

The following lemma summarizes the control we will need on the operator norm of submatrices and products of submatrices of $M$.

\begin{lemma}[Control on the operator norm]	\label{lem:opcontrol}
Let $\xi\in \C$ be a centered random variable with $\e|\xi|^{4+\eta}\le 1$ for some $\eta\in (0,1)$. 
Let $\asp \in (0,1)$. 
Then the following hold for all $n\ge 1$:
\begin{enumerate}[(a)]
\item (Control for sparse matrices) 
If $A\in \mM_n([0,1])$ is a fixed matrix and $X=(\xi_{ij})$ is an $n\times n$ matrix of iid copies of $\xi$, then 
\begin{equation}
\|A \circ X\| \ll \tau \sqrt{n}
\end{equation}
except with probability $O_\tau(n^{-\eta/8})$, where $\tau=\tau(A)\in [0,1]$ is any number such that
\begin{equation}
\sum_{k=1}^n a_{ik}^2, \, \sum_{k=1}^n a_{kj}^2 \le \tau^2 n
\end{equation}
for all $i,j\in [n]$, and 
\begin{equation}	\label{4th.mom.bd}
\sum_{i,j=1}^n a_{ij}^4 \le \tau^4 n^2.
\end{equation}

\item (Control for matrix products) 
Let $m\in [\asp n,n]$.
If $A\in \mM_{n,m}([0,1])$ and $D\in \mM_{m, n}(\C)$ are fixed matrices with $\|D\|\le 1$, and $X=(\xi_{ij})$ is an $n\times m$ matrix of iid copies of $\xi$, then 
\begin{equation}	\label{mbp:2}
\|D(A\circ X)\| \ll_\eta 
\sqrt{m}
\end{equation}
except with probability $O_{\asp }(n^{-\eta/8})$.
\end{enumerate}

\end{lemma}

\begin{remark}
The probability bounds in the above lemma can be improved under higher moment assumptions on $\xi$, and improve to exponential bounds under the assumption that $\xi$ is sub-Gaussian (see \eqref{subgaussian}).
\end{remark}

We will use standard truncation arguments to deduce Lemma \ref{lem:opcontrol} from the following bounds on the expected operator norm of random matrices due to Lata\l a and Vershynin.

\begin{theorem}[Lata\l a \citep{Latala}]	\label{thm:latala}
Let $n,m$ be sufficiently large and let $Y$ be an $n\times m$ random matrix with independent, centered entries $Y_{ij}\in \R$ having finite fourth moment. Then
\begin{equation}
\e \|Y\| \ll \max_{i\in [n]} \Bigg( \sum_{j=1}^m \e Y_{ij}^2 \Bigg)^{1/2} + \max_{j\in [m]} \Bigg(\sum_{i=1}^n \e Y_{ij}^2 \Bigg)^{1/2} + \Bigg( \sum_{i=1}^n\sum_{j=1}^m \e Y_{ij}^4\Bigg)^{1/4}.
\end{equation}

\end{theorem}

\begin{theorem}[Vershynin \citep{Vershynin:product}]	\label{thm:vershynin}
Let $\eta\in (0,1)$ and $n,m,N$ sufficiently large natural numbers.
Let $D\in \mM_{m,N}(\R)$ be a deterministic matrix satisfying $\|D\|\le 1$ and $Y\in \mM_{N,n}(\R)$ be a random matrix with independent centered entries $Y_{ij}$ satisfying $\e |Y_{ij}|^{4+\eta}\le 1$.
Then
\begin{equation}
\e \|DY\| \ll_\eta \sqrt{n} + \sqrt{m}.
\end{equation}
\end{theorem}

\begin{proof}[Proof of Lemma \ref{lem:opcontrol}]
We begin with (a). 
By splitting $X$ into real and imaginary parts and applying the triangle inequality we may assume $\xi$ is a real-valued random variable.
Set $\eta_0= \min(1/4,\eta/32)$ and define the product event
\begin{equation}	\label{op1:eij}
\event = \bigwedge_{i,j=1}^n \event_{ij}; \quad\quad \event_{ij} = \big\{ |\xi_{ij}| \le n^{1/2-\eta_0}\big\}.
\end{equation}
By Markov's inequality,
\begin{equation}	\label{op1:pij}
\pr( \event_{ij}^c) \le n^{-(4+\eta)(1/2-\eta_0)} \le n^{-1}
\end{equation}
for all $i,j\in [n]$.
By the union bound,
\begin{equation}	\label{op1:event}
\pr(\event^c) \le n^2 n^{-(4+\eta)(1/2-\eta_0)} \le n^{-\eta/8}.
\end{equation}
We denote
\[
X' = (\xi_{ij}')  = (\xi_{ij} - \e\xi_{ij}\un_{\event_{ij}}) = X - \e (X\un_{\event}).
\]

First we show 
\begin{equation}	\label{op1:AeX}
\|A\circ \e(X\un_{\event})\| \ll \tau \sqrt{n}.
\end{equation}
Since the variables $\xi_{ij}$ are centered, 
$
|\e(\xi_{ij}\un_{\event_{ij}})|  =  | \e (\xi_{ij} \un_{\event_{ij}^c})|.
$
By two applications of H\"older's inequality and \eqref{op1:pij}, 
\[
|\e (\xi_{ij} \un_{\event_{ij}^c})| \le (\e |\xi_{ij}|^4)^{1/4} \pr(\event_{ij}^c)^{3/4} \le n^{-3/4}.
\]
Thus,
\begin{equation}
\|A\circ \e(X\un_{\event})\| \le \|A\circ \e(X\un_{\event})\|_{\HS} \le n^{-3/4} \|A\|_{\HS} \le \tau n^{1/4}
\end{equation}
which yields \eqref{op1:AeX} with room to spare.

Now from \eqref{op1:event}, \eqref{op1:AeX} and the triangle inequality it is enough to show
\begin{equation}	\label{op1:goal1}
\pro{\event \wedge\big\{ \|A\circ X'\| \ge C\tau \sqrt{n} \big\}} = O_\tau(n^{-\eta/8})
\end{equation}
for a sufficiently large constant $C>0$ (we will actually show an exponential bound).
First note that the variables $\xi_{ij}'\un_{\event_{ij}}$ are centered and satisfy $\e |\xi_{ij}'\un_{\event_{ij}}|^4 = O(1)$. 
It follows from Theorem \ref{thm:latala} that
\begin{align*}
\e\un_\event  \|A\circ X'\|
&\ll 
\max_{i\in [n]} \Bigg( \sum_{j=1}^n  a_{ij}^2 \Bigg)^{1/2} + \max_{j\in [n]} \Bigg(\sum_{i=1}^n  a_{ij}^2 \Bigg)^{1/2} + \Bigg( \sum_{i,j=1}^n  a_{ij}^4\Bigg)^{1/4}\\
&\ll \tau \sqrt{n}.
\end{align*}
Thus, \eqref{op1:goal1} will follow if we can show
\begin{equation}	\label{op1:goal2}
\pro{ \|A\circ X'\|\un_\event  - \e \|A\circ X'\| \un_\event\ge \tau \sqrt{n} } = O_\tau(n^{-\eta/8}).
\end{equation}
This in turn follows in a routine manner from Talagrand's inequality \cite[Theorem 6.6]{Talagrand:newlook} (see also \cite[Corollary 4.4.11]{AGZ:book}): 
Observe that $X\mapsto \|A\circ X\|$ is a convex and 1-Lipschitz function on the space $\mM_n(\R)$ equipped with the (Euclidean) Hilbert--Schmidt metric. 
Since the matrix $X'\un_\event$ has centered entries that are bounded by $O(n^{1/2-\eta_0})$, Talagrand's inequality gives that the left hand side of \eqref{op1:goal2} is bounded by 
\begin{equation}
O\big(\exp(-c\tau^2 n/(n^{1/2-\eta_0})^2)\big) = O\big( \exp(-c \tau^2 n^{2\eta_0})\big) 
\end{equation}
which gives \eqref{op1:goal2} with plenty of room.

Now we turn to part (b). 
The proof follows a very similar truncation argument to the one in part (a), so we only indicate the necessary modifications.
As before, by splitting $D$ and $X$ into real and imaginary parts and applying the triangle inequality we may assume $D$ and $X$ are real matrices. 
We define $\event$ as in \eqref{op1:eij}, with 
$
\event_{ij}= \big\{ |\xi_{ij}|\le (n\sqrt{m})^{1/3-\eta_1}\big\}
$
and 
\begin{equation}
\eta_1= \frac{1}{4}\frac{\eta}{4+\eta}.
\end{equation}
With this choice of $\eta_1$, Markov's inequality and the union bound give $\pr(\event^c) = O_{\asp} (n^{-\eta/8})$.
Taking $X'= X-\e(X\un_\event)$ as before, we can bound
$
\|D(A\circ \e(X\un_\event))\| \le \|A\circ \e(X\un_\event)\|
$
by submultiplicativity of the operator norm, and the same argument as before gives
\begin{equation}	\label{op2:det}
\|A\circ \e(X\un_\event)\| \le  nm(n\sqrt{m})^{-\frac34(4+\eta)(1/3-\eta_1)} = m^{1/2-\eta/32} = o(\sqrt{m}).
\end{equation}
Since $X'\un_{\event}$ has centered entries with finite moments of order $4+\eta$, by Theorem \ref{thm:vershynin} we have
\begin{equation}	\label{op2:e}
\e \|D(A\circ X'\un_\event)\| \ll_\eta \sqrt{m}.
\end{equation}
The mapping $X\mapsto \|D(A\circ X)\|$ is convex and 1-Lipschitz with respect to the Hilbert--Schmidt metric on $\mM_{n}(\R)$ (since $\|D\|\le 1$) so using Talagrand's inequality as in part (a) we find that 
\begin{align*}
\pro{ \|D(A\circ X'\un_\event)\| - \e \|D(A\circ X'\un_\event)\| \ge \sqrt{m}} 
&\ll \expo{ -c  m / (n\sqrt{m})^{2/3-2\eta_1}}\\
&\le \expo{-c'(\asp)  n^{c\eta } }
\end{align*}
for some constant $c>0$ and $c'(\asp)>0$ sufficiently small depending on $\asp$. 
As the last line is bounded by $O_{\asp}(n^{-\eta/8})$, the result follows from the above, \eqref{op2:det}, \eqref{op2:e} and the triangle inequality by the same argument as for part (a). 
\end{proof}

\subsection{Decomposition of the standard deviation profile}	\label{sec:decomp}

We now begin the proof of Theorem \ref{thm:main}, which occupies the remainder of the paper. 
In the present subsection we prove Lemma \ref{lem:decomp} below, which shows that the standard deviation profile $A$ can be partitioned into a bounded collection of submatrices with certain nice properties.
For the motivation behind this lemma (and the notation $J_{\free},J_{\cyc}$) see Section \ref{sec:ideas}.

\begin{lemma}	\label{lem:decomp}
Let $A$ be an $n\times n$ matrix with entries $a_{ij}\in [0,1]$.
Let $\eps,\delta ,\ha\in (0,1)$, and assume $\eps$ is sufficiently small depending on $\delta $.
There exists $0\le \me \ll_\eps1$, a partition 
\begin{align}
[n]&=J_{\badd}\cup J_{\free}\cup J_{\cyc}	\notag\\
&= J_{\badd}\cup J_{\free}\cup J_1\cup\cdots \cup J_{\me }	\label{decomp:bf1}
\end{align}
and a set $F\subset[n]^2$ satisfying the following properties:
\begin{enumerate}[(1)]
\item\label{decomp:1} $\eps n\ll |J_{\badd}|\ll \delta^{1/2} n$.
\item\label{decomp:F} $|F| \ll \delta n^2$, and for all $i\in J_{\free}$,
\begin{equation}	\label{decomp:Frows}
|\{j\in J_{\free}: (i,j) \in F\}|, \; |\{j\in J_{\free}: (j,i) \in F\}| \le \delta^{1/2} n.
\end{equation}
\item\label{decomp:2} If $J_{\free}\ne\varnothing$ then there is a permutation $\tau:J_{\free}\to J_{\free}$ such that for all $(i,j)\in J_{\free}\times J_{\free} \setminus F$ with $\tau(i)\ge \tau(j)$, $\sig_{ij} <\ha$.
\item\label{decomp:3} If $\me \ge 1$ then
\begin{equation}
|J_1|=\cdots = |J_{\me }| \gg_\eps n
\end{equation}
and there is a permutation $\pi:[\me ]\to [\me ]$ such that for all $1\le k\le \me $, $A(\ha)_{J_k,J_{\pi(k)}}$ is $(2\delta,2\eps)$-super-regular 
(see Definition \ref{def:super}).
\end{enumerate}
\end{lemma}

\begin{proof}
We begin by applying Lemma \ref{lem:regularity} to $A(\ha)$ to obtain $m_0\in \N$ with $\eps^{-1}\le m_0=O_\eps(1)$ and a partition $[n]=I_0\cup\cdots \cup I_{m_0}$ satisfying the properties in that lemma.

The partition $I_0,\dots,I_{m_0}$ is almost what we need.
In the remainder of the proof we perform a ``cleaning" procedure (as it is commonly referred to in the extremal combinatorics literature) to obtain a partition $J_0,\dots, J_{m_0}$ with improved properties, where $J_k\subset I_k$ for each $1\le k\le m_0$, and $J_0\supset I_0$ collects the leftover elements. 

We start by forming a \emph{reduced digraph} $\mR=([m_0],E)$ on the vertex set $[m_0]$ with directed edge set
\begin{equation}
E:= \Big\{ (k,l)\in [m_0]^2: (I_k,I_l) \mbox{ is $\eps$-regular and } \rho_{A(\ha)}(I_k,I_l) >5\delta\Big\}.
\end{equation}
Next we find a (possibly empty) set $T\subset [m_0]$ such that the induced subgraph $\mR(T)$ is covered by vertex-disjoint directed cycles, and the induced subgraph $\mR([m_0]\setminus T)$ is cycle-free.
Such a set can be obtained by greedily removing cycles and the associated vertices from $\mR$ until the remaining graph has no more directed cycles.
By relabeling $I_1,\dots, I_{m_0}$ we may take $T=[\me]$, where $\me\in [0, m_0]$. 

Assuming $\me \ne 0$, the fact that $\mR([\me ])$ is 
covered by vertex-disjoint cycles
is equivalent to the existence of a permutation $\pi:[\me ]\to [\me ]$ such that $(k,\pi(k))\in E$ for all $1\le k\le \me $.
Now we will obtain the sets $J_1,\dots, J_{\me }$ obeying the properties in part (\ref{decomp:3}) of the lemma.
Let $1\le k\le \me $. 
We have that $(I_k,I_{\pi(k)})$ is $\eps$-regular with density $\rho_k:= \rho_{A(\ha)} (I_k,I_{\pi(k)})>5\delta$, so if we assume $\eps\le \delta$ then for every $I\subset I_k, J\subset I_{\pi(k)}$ with $|I|,|J|\ge \eps |I_k|$, 
\begin{equation}	\label{decomp:edge}
e_{A(\ha)}(I,J) \ge (\rho_k-\eps)|I||J| \ge 4\delta |I||J|.
\end{equation}
It remains to ensure that conditions (1) and (2) from Definition \ref{def:super} also hold, which we will do by removing a small number of rows and columns.
Letting
\[
I_k' = \big\{ i\in I_k: |\mN_{A(\ha)}(i)\cap I_{\pi(k)}|<4\delta|I_k|\big\}
\]
we have $e_{A(\ha)}(I_k',I_{\pi(k)})<4\delta|I_k'||I_{\pi(k)}|$, and it follows that $|I_k'|\le \eps |I_k|$.
Similarly, letting 
\[
I_k'' = \big\{ i\in I_k: |\mN_{A(\ha)^\tran}(i)\cap I_{\pi^{-1}(k)}|<4\delta|I_k|\big\}
\]
we have $|I_k''|\le \eps |I_k|$.
Letting $I_k^*\subset I_k$ be a set of size $\lf 2\eps|I_k|\rf$ containing $I_k'\cup I_k''$, we take
\begin{equation}	\label{decomp:badcyc}
J_k= I_k\setminus I_k^*.
\end{equation}
With this definition we have $|J_1|=\cdots |J_{\me }|$, and for each $1\le k\le \me , i\in J_k$,
\begin{equation}
|\mN_{A(\ha)}(i)\cup J_{\pi(k)}|, \, |\mN_{A(\ha)^\tran}(i)\cap J_{\pi^{-1}(k)}|\ge (4\delta-2\eps)|I_k|\ge 2\delta|J_k|.
\end{equation}
Furthermore, for each $1\le k\le \me $ and $I\subset J_k, J\subset J_{\pi(k)}$ with $|I|,|J|\ge 2\eps |J_k|$,
if we assume $\eps\le 1/4$ then $|I|,|J|\ge \eps |I_k|$, so by \eqref{decomp:edge}
\begin{equation}
e_{A(\ha)}(I,J) \ge 4\delta|I||J|.
\end{equation}
It follows that for every $1\le k\le \me $ the submatrix $A(\ha)_{J_k,J_{\pi(k)}}$ is $(2\delta,2\eps)$-super-regular, which concludes the proof of part (\ref{decomp:3}) of the lemma.

Now we prove parts (\ref{decomp:F}) and (\ref{decomp:2}). 
We will obtain $J_{\free}$ by removing a small number of bad elements from $I_{\free}:=\bigcup_{k=\me +1}^{m_0}I_k$. 
Since the induced subgraph $\mR([\me +1,m_0])$ is cycle-free we may relabel $I_{\me +1},\dots, I_{m_0}$ so that 
\begin{equation}	\label{decomp:relabel}
(k,l) \notin E\; \mbox{ for all $\me < l\le k\le m_0$}.
\end{equation}
We take
\begin{equation}
F= \big\{ (i,j)\in [n]^2: (i,j)\in I_k\times I_l  \mbox{ for some $(k,l)\notin E$}\big\}.
\end{equation}
The contribution to $F$ from irregular pairs $(I_k,I_l)$ is at most $\eps n^2$ by the regularity of the partition $I_0,\dots, I_{m_0}$, and the contribution from pairs $(I_k,I_l)$ with density less than $5\delta$ is at most $5\delta n^2$.
Hence,
\begin{equation}	\label{decomp:Fbound}
|F| \le \eps n^2+5\delta n^2 \le 6\delta n^2
\end{equation}
giving the first estimate in (\ref{decomp:F}) (recall that we assumed $\eps\le \delta$).
Setting
\begin{equation}	\label{decomp:Frows1}
I_{\free}'= \left\{ i\in I_{\free}: \max\big( |\{j\in [n]: (i,j)\in F\}|, |\{j\in [n]: (j,i)\in F\}|\big) \ge \delta^{1/2}n\right\}
\end{equation}
it follows from \eqref{decomp:Fbound} that
\begin{equation}
|I_{\free}'| \le 12\delta^{1/2}n.
\end{equation}
Let $I_{\free}^*\subset I_{\free}$ be any set containing $I_{\free}'$ of size $\min(|I_{\free}|, \lf 12 \delta^{1/2}n\rf)$ and take $J_{\free}= I_{\free}\setminus I_{\free}^*$.
The bounds \eqref{decomp:Frows} now follow immediately from \eqref{decomp:Frows1}. 
For part (\ref{decomp:2}), from \eqref{decomp:relabel} we may take for $\tau$ any ordering of the elements of $J_{\free}$ that respects the order of the sets $J_k:=I_k\setminus I_{\free}^*$, i.e. so that $\tau(j)\ge \tau(i)$ for all $i\in J_k, j\in J_l$ and all $\me <l\le k\le m_0$.

Finally, taking
\begin{equation}
J_{\badd} = I_0 \cup I_{\free}^* \cup \bigcup_{k=1}^{\me } I_k^*.
\end{equation}
we have 
\[
|J_{\badd}| \le \eps n + 12\delta^{1/2}n + 2\eps n \le 15 \delta^{1/2}n
\]
giving the upper bound in part (\ref{decomp:1}).
Now recalling that we took 
$|I_{\free}^*|= \min(|I_{\free}|, \lf 12 \delta^{1/2}n\rf)$
and $|I_k^*| = \lf 2\eps |I_k|\rf$ for all $1\le k\le \me $, we also have the lower bound
\begin{align*}
|J_{\badd}|
&\ge \min\bigg(|I_{\free}^*|, \; \bigg|\bigcup_{k=1}^{\me } I_k^*\bigg| \bigg)  \\
&\ge \min\bigg(  \lf 12\delta^{1/2}n\rf ,\, |I_{\free}|,\, 2\eps \bigg| \bigcup_{k=1}^{\me } I_k \bigg| - \me \bigg)\\
&= \min \bigg( \lf 12 \delta^{1/2}n\rf, \, \bigg| \bigcup_{k=\me +1}^{m_0} I_k \bigg| , \,
2\eps \bigg| \bigcup_{k=1}^{\me } I_k \bigg| - \me \bigg) \\
&\gg \eps n
\end{align*}
where we used that at least one of the sets $I_{\free} = \bigcup_{k=\me +1}^{m_0} I_k$, $I_{\cyc} = \bigcup_{k=1}^{\me } I_k$ must be of size at least $n/4$, say.
This gives the lower bound in part (\ref{decomp:1}) and completes the proof.
\end{proof}

\subsection{High level proof of Theorem \ref{thm:main}}	\label{sec:highlevel}

In this subsection we prove Theorem \ref{thm:main} on two lemmas (Lemmas \ref{lem:nil} and \ref{lem:cyc}) which give control on the smallest singular values of the submatrices $M_{J_{\free}}$ and (perturbations of) $M_{J_{\cyc}}$, with $J_{\free}, J_{\cyc}$ as in Lemma \ref{lem:decomp}.
The proofs of these lemmas are deferred to the remaining subsections.

By our moment assumptions on $\xi$ it follows that $\xi$ is $\kappa_0$-spread for some $\kappa_0=O(\mu_{4+\eta}^2)$ (see Remark \ref{rmk:kappap}).
By Lemma \ref{lem:wlog.kappa} and multiplying $X$ and $B$ by a phase we may assume $\xi$ has $O(\mu_{4+\eta}^2)$-controlled second moment.
Without loss of generality we may assume $\eta<1$.
We introduce parameters $\ha,\delta,\eps\in (0,1)$ to be chosen sufficiently small depending on $r_0,\eta$, and $\mu_{4+\eta}$; specifically we will have the following dependencies:
\begin{equation}	\label{depends}
\ha=\ha(r_0,\mu_{4+\eta}), \quad \delta=\delta(r_0,\eta,\mu_{4+\eta}), \quad \eps=\eps(\ha,\delta).
\end{equation}
For the remainder of the proof we assume that $n$ is sufficiently large depending on all parameters (which will only depend on $r_0,K_0, \eta$ and $\mu_{4+\eta}$). 

We begin by summarizing the control we have on the operator norm of submatrices of $A\circ X$.
%
From Lemma \ref{lem:opcontrol}(a) we have that for any fixed $B=(b_{ij})\in \mM_n([0,1])$ and any $I,J\subset[n]$ with $|I|\le |J|$,
\begin{equation}	\label{opIJ:improved}
\pro{ \|(B\circ X)_{I,J}\| \le \tau K \sqrt{|J|}} = 1-O_\tau(|J|^{-\eta/8})
\end{equation}
for some $K=O(\mu_{4+\eta})$, and any $\tau\le 1$ satisfying 
\begin{equation}
\tau \ge \frac1{|J|^{1/2}}\max\left( \max_{i \in I}\left( \sum_{j\in J} b_{ij}^2\right)^{1/2}, \;\max_{j\in J}\left( \sum_{i\in I} b_{ij}^2 \right)^{1/2}, \left( \sum_{i,j=1}^n b_{ij}^4\right)^{1/4} \right),
\end{equation}
and similarly with $|J|$ replaced by $|I|$ if $|J|\le |I|$.
In particular, taking $\tau=1$ and $B=A$ we have
\begin{align}
\|(A\circ X)_{I,J}\| &\ll_{\mu_{4+\eta}} \sqrt{\max(|I|,|J|)}	\notag \\
&\qquad \text{with probability } 1-O(\max(|I|,|J|)^{-\eta/8}).	\label{opIJ}
\end{align}
(We state \eqref{opIJ:improved} for general $B\in \mM_n([0,1])$ as at one point we will apply this to a residual matrix obtained by subtracting off a collection of ``bad" entries from $A$.)


We now apply Lemma \ref{lem:decomp} (assuming $\eps$ is sufficiently small depending on $\delta$) to obtain a partition $[n] = J_{\badd}\cup J_{\free} \cup J_{\cyc}$ and a set $F\subset[n]^2$ satisfying the properties (1)--(4) in the lemma. 
In the following we abbreviate $M_{\free}:= M_{J_{\free}}$ and $M_{\cyc}:= M_{J_{\cyc}}$.

\begin{lemma}		\label{lem:nil}
Assume $n_1:=|J_{\free}|\ge \delta^{1/2}n$.
If $\ha,\delta$ are sufficiently small depending on $r_0$ and 
$\mu_{4+\eta}$,
then
\begin{equation}
s_{n_1}(M_{\free}) \gg_{\mu_{4+\eta},r_0} \sqrt{n}
\end{equation}
except with probability $O_{\mu_{4+\eta},r_0,\delta}(n^{-\eta/9})$.
\end{lemma}

(Note that while the definition of $M_{\free}$ depends on $\eps$, the bounds in the above lemma are independent of $\eps$.)

\begin{lemma}		\label{lem:cyc}
Assume $n_2:=|J_{\cyc}| \ge \delta^{1/2}n$. 
Fix $\gamma\ge1$ and let $W\in \mM_{n_2}(\C)$ be a deterministic matrix with $\|W\|\le n^{\gamma}$. 
There exists $\beta= \beta(\gamma,\ha,\delta)$ such that if $\eps=\eps(\ha,\delta)$ is sufficiently small,
\begin{equation}
\pro{ s_{n_2}(M_{\cyc} + W) \le n^{-\beta}}  \ll_{K_0,\gamma,\delta,\ha,\mu_{4+\eta}}  \sqrt{\frac{\log n}{n}}.
\end{equation}
\end{lemma}

\begin{remark}	\label{rmk:relaxmom}
We note that in the proof of Lemma \ref{lem:cyc} we do not make use of the fact that the atom variable $\xi$ has more than two finite moments (the dependence on $\mu_{4+\eta}$ is only through the parameter $\kappa_0=O(\mu_{4+\eta}^2)$).
In particular, we can remove the extra moment hypotheses in Theorem \ref{thm:main} under the additional assumption that the standard deviation profile $A$ contains a generalized diagonal of block submatrices which are super-regular and of dimension linear in $n$ (that is, if we can take $J_{\badd}=J_{\free}=\varnothing$ in \eqref{decomp:bf1}). 
\end{remark}

We defer the proofs of Lemmas \ref{lem:nil} and \ref{lem:cyc} to subsequent sections, and conclude the proof of Theorem \ref{thm:main}. 
Note that at this stage (before we have applied Lemma \ref{lem:nil} or \ref{lem:cyc}) the only constraint we have put on the parameters in \eqref{depends} is to assume $\eps$ is sufficiently small depending on $\delta$ for the application of Lemma \ref{lem:decomp}.
We proceed in the following steps:\\

\begin{itemize}
\item []
\begin{itemize}
\item [ {\bf Step 1:}] Bound the smallest singular value of $M_{\free}$ using Lemma \ref{lem:nil}. In this step we fix $\sigma(r_0,\mu_{4+\eta})$, while $\delta$ is assumed to be sufficiently small depending on $r_0,\mu_{4+\eta}$ but is otherwise left free.

\item [ {\bf Step 2:}] Bound the smallest singular value of 
\begin{equation}	\label{decompM1}
M_1 := M_{J_{\free}\cup J_{\badd},\,J_{\free}\cup J_{\badd}} = \begin{pmatrix} M_{\free} & B_1 \\ C_1 & M_0 \end{pmatrix}.
\end{equation}
using the result of Step 1, the Schur complement bound of Lemma \ref{lem:schur}, 
\eqref{opIJ:improved} and Lemma \ref{lem:opcontrol}(b).
%
In this step we fix $\delta(r_0,\eta,\mu_{4+\eta})$.

\item [ {\bf Step 3:}] Bound the smallest singular value of 
\begin{equation}	\label{decompM}
M= \begin{pmatrix} M_{\cyc} & B_2\\ C_2 & M_1\end{pmatrix}.
\end{equation}
using the result of Step 2, the Schur complement bound of Lemma \ref{lem:schur}, and Lemma \ref{lem:cyc}.
In this step we fix $\eps(\ha,\delta)$.
\end{itemize}
\end{itemize}

The case that one of $J_{\free}$ or $J_{\cyc}$ is small (or empty) can be handled essentially by skipping either Step 1 or Step 3. 
We will begin by assuming
\begin{equation}	\label{LB:nilcyc}
|J_{\free}|,\, |J_{\cyc}| \ge \delta^{1/2}n
\end{equation}
and address the case that this does not hold at the end.

\subsubsection*{Step 1}

By Lemma \ref{lem:nil} and the assumption \eqref{LB:nilcyc}, we can take $\ha$ and $\delta$ sufficiently small depending on $r_0$ and 
$\mu_{4+\eta}$
such that 
\begin{equation}	\label{LB:Mnil}
s_{\min}(M_{\free}) \gg_{\mu_{4+\eta},r_0} \sqrt{n}
\end{equation}
except with probability $O_{\mu_{4+\eta},r_0,\delta}(n^{-\eta/9})$.
We now fix $\ha=\ha(r_0,\mu_{4+\eta})$ once and for all, but leave $\delta$ free to be taken smaller if necessary.
By independence of the entries of $M$ we may now condition on a realization of $M_{\free}$ such that \eqref{LB:Mnil} holds.

\subsubsection*{Step 2}

By \eqref{opIJ} and \eqref{LB:nilcyc} we have $\|C_1\|=O_{\mu_{4+\eta}}(\sqrt{n})$ except with probability $O_\delta(n^{-\eta/8})$. We henceforth condition on a realization of $C_1$ satisfying this bound.
Together with \eqref{LB:Mnil} this gives
\begin{equation}
\|C_1M_{\free}^{-1}\| \le \frac{\|C_1\|}{s_{\min}(M_{\free})} \ll_{\mu_{4+\eta},r_0}1.
\end{equation}
Since $B_1$ is independent of $C_1$ and $M_{\free}$ we can apply 
Lemma \ref{lem:opcontrol}(b)
to conclude
\begin{equation}	\label{C1MB1}
\|C_1M_{\free}^{-1}B_1\| \ll_{\eta,\mu_{4+\eta}}\|C_1M_{\free}^{-1}\| |J_{\badd}|^{1/2} \ll_{\eta,\mu_{4+\eta},r_0} |J_{\badd}|^{1/2}
\end{equation}
except with probability 
$O_{\eps}(n_1^{-\eta/8}) = O_{\delta,\eps}(n^{-\eta/9})$, 
where we have used the lower bound $|J_{\badd}|\gg \eps n$ from Lemma \ref{lem:decomp}(1).
On the other hand, by the triangle inequality and \eqref{opIJ},
\begin{equation}	\label{step2:M0}
s_{\min}(M_0) = s_{\min}(Z_{J_{\badd}} \sqrt{n} + (A\circ X)_{J_{\badd}}) \ge r_0\sqrt{n} - O_{\mu_{4+\eta}}(|J_{\badd}|^{1/2})
\end{equation}
except with probability $O(|J_{\badd}|^{-\eta/8}) = O_\eps(n^{-\eta/9})$. 
Again by the triangle inequality and the previous two displays, 
\begin{equation}
s_{\min}(M_0-C_1M_{\free}^{-1} B_1) \ge r_0\sqrt{n} - O_{\eta,\mu_{4+\eta},r_0}(|J_{\badd}|^{1/2})
\end{equation}
except with probability $O_{\delta,\eps}(n^{-\eta/9})$. 	
Since $|J_{\badd}|\ll \delta^{1/2}n$ we can take $\delta$ smaller, if necessary, depending on $r_0,\eta,\mu_{4+\eta}$ to conclude that
\begin{equation}	\label{step2:condition}
s_{\min}(M_0-C_1M_{\free}^{-1} B_1)\ge (r_0/2) \sqrt{n}
\end{equation}
except with probability $O_{\delta,\eps}(n^{-\eta/9})$. 
We may henceforth condition on the event that \eqref{step2:condition} holds. 
Of an event with probability $O_\delta(n^{-\eta/8})$ we may also assume $\|B_1\|=O_{\mu_{4+\eta}}(\sqrt{n})$.
From Lemma \ref{lem:schur} and the preceding estimates we have
\begin{align}
s_{\min}(M_1) 
&\gg \left( 1+ \frac{O_{\mu_{4+\eta}}(\sqrt{n})}{s_{\min}(M_{\free})}\right)^{-2} 
\min\big[ s_{\min}(M_{\free}), s_{\min}(M_0- C_1 M_{\free}^{-1} B_1)\big]	\notag\\
&\gg_{\mu_{4+\eta},r_0}\min\big[ \sqrt{n}, s_{\min}(M_0-C_1M_{\free}^{-1}B_1)\big]	\notag\\
&\gg_{\mu_{4+\eta},r_0} \sqrt{n}.	\label{step2:final}
\end{align}
At this point we fix $\delta=\delta(r_0, \eta,\mu_{4+\eta})$.

\subsubsection*{Step 3} 
Condition on a realization of $M_1$ such that \eqref{step2:final} holds.
By \eqref{opIJ} we may also condition on realizations of the matrices $B_2,C_2$ in \eqref{decompM} such that $\|B_2\|,\|C_2\| \ll_{\mu_{4+\eta}}\sqrt{n}$.
Applying Lemma \ref{lem:schur}, 
\begin{align}
s_n(M) 
&\gg \left( 1+ \frac{O_{\mu_{4+\eta}}(\sqrt{n})}{s_{\min}(M_1)}\right)^{-2} \min \big[ s_{\min}(M_1), \, 
s_{\min}(M_{\cyc} - B_2 M_1^{-1} C_2) \big] 	\notag\\
& \gg_{\mu_{4+\eta},r_0} \min \big[ \sqrt{n}, s_{\min}(M_{\cyc} - B_2 M_1^{-1} C_2) \big].	\label{step3:start}
\end{align}
By our estimates on $\|B_2\|,\|C_2\|$ and $s_{\min}(M_1)$ we have
\begin{equation}
\|B_2M_1^{-1}C_2\| \ll_{\mu_{4+\eta}} \frac{ n}{s_{\min}(M_1)} \ll_{\mu_{4+\eta},r_0} \sqrt{n}
\end{equation}
(unlike in Step 2, here we did not need the stronger control on matrix products provided by \eqref{mbp:2}). 
Now since $M_2$ is independent of $M_1,B_2,C_2$, we can apply Lemma \ref{lem:cyc} with $\gamma=0.51$ (say), fixing $\eps$ sufficiently small depending on $\ha(r_0,\mu_{4+\eta})$ and $\delta(r_0,\eta,\mu_{4+\eta})$, to obtain
\begin{equation}
\pro{ s_{\min}(M_{\cyc} - B_2 M_1^{-1} C_2) \le n^{-\beta}} \ll_{K_0, r_0,\eta,\mu_{4+\eta}} \sqrt{\frac{\log n}{n}}
\end{equation}
for some $\beta = \beta(r_0,\eta,\mu_{4+\eta})>0$.
The result now follows from the above and \eqref{step3:start}, taking $\alpha=\min(\eta/9,1/4)$, say.

It only remains to address the case that the assumption \eqref{LB:nilcyc} fails. 
We may assume that $\delta$ is small enough that only one of these bounds fails. 
In this case we simply redefine $J_{\badd}$ to include the smaller of $J_{\cyc}, J_{\free}$. 
Note that we still have $|J_{\badd}|= O(\delta^{1/2}n)$.
If $|J_{\cyc}|<\delta^{1/2}n$, then with this new definition of $J_{\badd}$ we have $M=M_1$, and the desired bound on $s_n(M)$ follows from \eqref{step2:final} (with plenty of room). 
If $|J_{\free}|< \delta^{1/2}n$ then we skip Step 2, proceeding with Step 3 using $M_0$ in place of $M_1$. The bound \eqref{step2:final} in this case follows from \eqref{step2:M0} and the bound $|J_{\badd}|\ll \delta^{1/2}n$, taking $\delta$ sufficiently small depending on $\mu_{4+\eta},r_0$.
This concludes the proof of Theorem \ref{thm:main}.

\subsection{Proof of Lemma \ref{lem:nil}}		\label{sec:nil}

We denote
\begin{equation}
A_{F}= (\sig_{ij} 1_{(i,j)\in F}).
\end{equation}
By the estimates on $F$ in Lemma \ref{lem:decomp} we can apply \eqref{opIJ:improved} with $\tau=O(\delta^{1/4})$ to obtain
\begin{equation}
\|(A_{F}(\ha)\circ X)_{J_{\free}}\| \ll_{\mu_{4+\eta}} \delta^{1/4}\sqrt{n}
\end{equation}
except with probability at most $O_{\delta}(n_1^{-\eta/8}) = O_{\delta}(n^{-\eta/9})$.
By another application of \eqref{opIJ:improved} with $\tau=1$, 
\begin{equation}
\big\|\big((A-A(\ha))\circ X\big)_{J_{\free}}\big\| \ll_{\mu_{4+\eta}} \ha\sqrt{n}
\end{equation}
except with probability at most $O_{\delta}(n^{-\eta/9})$. 
Let
\begin{equation}
\tM_{\free} := (\tA\circ X)_{J_{\free}} + Z_{J_{\free}}\sqrt{n}, \quad\quad \tA := A(\ha)-A_{F}(\ha).
\end{equation}
By the above estimates and the triangle inequality,
\begin{align}
s_{\min}(M_{\free}) &\ge s_{\min}(\tM_{\free}) - \|((A-\tA)\circ X)_{J_{\free}}\| \notag\\
&\ge s_{\min}(\tM_{\free})- O_{\mu_{4+\eta}}(\delta^{1/4}+ \ha)\sqrt{n}	\label{nil:redux}
\end{align}
except with probability $O_{\delta}(n^{-\eta/9})$. 
Thus, it suffices to show 
\begin{equation}	\label{nil:goal1}
s_{\min}(\tM_{\free})\gg_{\mu_{4+\eta},r_0}\sqrt{n}.
\end{equation}
except with probability $O_{\mu_{4+\eta},r_0,\delta}(n^{-\eta/9})$ -- the result will then follow from \eqref{nil:goal1} and \eqref{nil:redux} by taking $\delta,\ha$ sufficiently small depending on $\mu_{4+\eta},r_0$.
Furthermore, by Lemma \ref{lem:decomp}(3) and conjugating $M_{\free}$ by a permutation matrix we may assume that $\tA$ is (strictly) upper triangular. 
Now it suffices to prove the following:

\begin{lemma}
Let $M= A\circ X + B$ be an $n\times n$ matrix as in Definition \ref{def:profile}, and further assume that for some $r_0>0, K\ge 1, \alpha>0$,
\begin{itemize}
\item $A$ is upper triangular;
\item $B=Z\sqrt{n} = \diag(z_i\sqrt{n})_{i=1}^n$ with $|z_i|\ge r_0$ for all $1\le i\le n$;
\item $\xi$ is such that for all $n'\ge 1$ and any fixed $A'\in \mM_{n'}([0,1])$, $\|A'\circ X'\|\le K\sqrt{n'}$  except with probability $O((n')^{-\alpha})$.
\end{itemize}
Then $s_n(M) \gg_{K,r_0}\sqrt{n}$ except with probability $O_{K,r_0}(1)^\alpha n^{-\alpha}$.
\end{lemma}

\begin{remark}
The proof gives an implied constant of order $\exp(-O(K/r_0)^{O(1)})$ in the lower bound on $s_n(M)$. 
\end{remark}

To deduce Lemma \ref{lem:nil} we apply the above lemma with $M=\tM_{\free}$, $\alpha=\eta/8$, $K=O(\mu_{4+\eta})$ (by \eqref{opIJ}) and $n_1\gg_{\delta}n$ in place of $n$, which gives that \eqref{nil:goal1} holds with probability 
\begin{equation}
1-O_{\mu_{4+\eta},r_0}(n_1^{-\eta/8}) = 1-O_{\mu_{4+\eta},r_0,\delta}(n^{-\eta/9})
\end{equation}
where in the first bound we applied our assumption that $\eta<1$. 

\begin{proof}
First we note that we may take $n$ to be a dyadic integer, i.e. $n=2^q$ for some $q\in \N$.
Indeed, if this is not the case, then letting $2^q$ be the smallest dyadic integer larger than $n$ we can increase the dimension of $M$ to $2^q$ by padding $A$ out with rows and columns of zeros, adding additional rows and columns of iid copies of $\xi$ to $X$, and extending the diagonal of $Z$ with entries $z_i\equiv r_0$ for $n<i\le 2^q$.
The hypotheses on $A$ and $Z$ in the lemma are still satisfied, and the smallest singular value of the new matrix is a lower bound for that of the original matrix (since the original matrix is a submatrix of the new matrix).

Now fix an arbitrary dyadic filtration $\mF= \bigcup_{p\ge 0}\{J_s: s\in \{0,1\}^p\}$ of $[n]$, where we view $\{0,1\}^0$ as labeling the trivial partition of $[n]$, consisting only of the empty string $\varnothing$, so that $J_\varnothing = [n]$. 
Thus, for every $0\le p< q$ and every binary string $s\in \{0,1\}^{p}$, $J_s$ has cardinality $n2^{-p}$ and is evenly partitioned by $J_{s0},J_{s1}$. 
For a binary string $s$ we abbreviate $M_s:= M_{J_s}$ and similarly define $A_s,X_s,Z_s$. 
We also write $B_s=M_{J_{s0},J_{s1}}$, so that we have the block decomposition
\begin{equation}	\label{Ms:block}
M_s= \begin{pmatrix} M_{s0} & B_s\\ 0 & M_{s1} \end{pmatrix}.
\end{equation}

For $p\ge1$ define the boundedness event 
\begin{equation}
\mB^*(p) = \big\{ \|A\circ X\| \le K\sqrt{n}\} \wedge \big\{ \forall s\in \{0,1\}^{p}, \; \|A_s\circ X_s\| \le K\sqrt{n 2^{-p}}\big\}.
\end{equation}
By our assumption on $\xi$ we have 
\begin{equation}	\label{mBstar:lb}
\pr(\mB^*(p))\ge 1- O(n^{-\alpha}) - 2^{p} O((n2^{-p})^{-\alpha}) = 1- O(2^{(1+\alpha)p}n^{-\alpha}).
\end{equation}
For arbitrary $s\in \{0,1\}^{p}$, by the triangle inequality we have that on $\mB^*(p)$,
\begin{align*}
s_{\min}(M_s) &\ge s_{\min}(Z_s) - \|A_s\circ X_s\| \ge (r_0 - K2^{-p/2})\sqrt{n}.
\end{align*}
Setting $p_0= \lf 2\log (2K/r_0)\rf +1$ we have that on $\mB^*(p_0)$,
\begin{equation}	\label{lambdap:00}
s_{\min}(M_s) \ge (r_0/2)\sqrt{n}
\end{equation}
for all $s\in \{0,1\}^{p_0}$.
For the remainder of the proof we restrict the sample space to the event $\mB^*(p_0)$ and will use the Schur complement bound (Lemma \ref{lem:schur}) to show that the desired lower bound on $s_{\min}(M)$ holds deterministically (note that by \eqref{mBstar:lb} and our choice of $p_0$, $\mB^*(p_0)$ holds with probability 
$1-O_{K,r_0}(n^{-\alpha})$).

For $0\le p\le p_0$ let
\begin{equation}
\lambda_p= \min_{s\in \{0,1\}^p} \frac{1}{\sqrt{n}} s_{\min}(M_s).
\end{equation}
From \eqref{lambdap:00} we have
\begin{equation}	 \label{lambdap:0}
\lambda_{p_0}\ge r_0/2
\end{equation}
Now let $1\le p\le p_0$ and $s\in \{0,1\}^{p-1}$. By the block decomposition \eqref{Ms:block} and Lemma \ref{lem:schur},
\begin{align*}
s_{\min}(M_s) 
&\gg \left(1+ \frac{\|B_s\|}{s_{\min}(M_{s0})}\right)^{-1} \min\big( s_{\min}(M_{s0}), s_{\min}(M_{s1})\big)\\
&\ge (1+K/\lambda_p)^{-1}\lambda_p\sqrt{n}
\end{align*}
so $\lambda_{p-1} \gg (1+K/\lambda_p)^{-1}\lambda_p\sqrt{n}$ for all $0\le p\le p_0$. 
Applying this iteratively along with \eqref{lambdap:0} we conclude $\lambda_0\gg_{K,r_0} 1$, i.e. 
\begin{equation}
s_{\min}(M) \gg_{K,r_0} \sqrt{n}
\end{equation}
as desired.
\end{proof}

\subsection{Proof of Lemma \ref{lem:cyc}}		\label{sec:cyc}
We may assume throughout that $n$ is sufficiently large depending on the parameters $K_0,\gamma,\delta,\ha$, and $\mu_{4+\eta}$. Note we may also assume $\gamma> 2$ without loss of generality.
We will apply only the following crude control on the operator norm of submatrices: 
\begin{equation}	\label{crude.op}
\pr(\|(A\circ X)_{I,J}\|\ge n^2 ) \le n^{-2}	\quad \forall I,J\subset[n].
\end{equation}
Indeed, for any $I,J\subset[n]$,
\begin{align*}
\pr(\|(A\circ X)_{I,J}\|\ge n^2 ) \le  \pr(\|A\circ X\|_{\HS}\ge n^2) .
\end{align*}
Furthermore, $\e\|A\circ X\|_{\HS}^2 \le \e \|X\|_{\HS}^2 = n^2$, 
and \eqref{crude.op} follows from the above display and Markov's inequality.

By multiplying $M_{\cyc}$ by a permutation matrix we may assume that $A_k:= A_{J_k}$ is $(2\delta,2\eps)$-super-regular for $1\le k\le \me $ (unlike in the proof of Lemma \ref{lem:nil} the diagonal matrix $Z\sqrt{n}$ plays no special role here).
We denote $J_{\le k}= J_1\cup\cdots\cup J_k$, and for any matrix $W$ of dimension at least $|J_{\le k}|$ we abbreviate
\begin{equation}
W_k=W_{J_k},\quad W_{\le k} = W_{J_{\le k}},\quad W_{\le k-1, k} = W_{J_{\le k-1}, J_{k}}, \quad W_{k,\le k-1}= W_{J_{k}, J_{\le k-1}}
\end{equation}
so that for $2\le k\le \me $ we have the block decomposition
\begin{equation}
W_{\le k} = \begin{pmatrix} W_{\le k-1} & W_{\le k-1,k}\\ W_{k,\le k-1} & W_{k}\end{pmatrix}.
\end{equation}
Let us denote 
\begin{equation}	\label{nprime.eps}
n'= |J_1|=\cdots=|J_{\me }| \gg_\eps n.
\end{equation}
For $1\le k\le \me -1$, $\beta>0$ and a fixed $kn'\times kn'$ matrix $W$, we denote the event 
\begin{equation}
\event_k(\beta,W) := \big\{ s_{kn'}(M_{\le k} + W) > n^{-\beta}\big\}.
\end{equation}

Let $\gamma>2$ and fix an arbitrary matrix $W\in \mM_{n',n'}(\C)$ with $\|W\|\le n^\gamma$.
By \eqref{crude.op} we have
\begin{equation}
\|M_1+W\| \le K_0\sqrt{n} + n^2 + n^\gamma \le 2n^\gamma
\end{equation}
with probability $1-O(n^{-2})$ if $n$ is sufficiently large depending on $K_0$ and $\gamma$. 
By Theorem \ref{thm:super} there exists $\beta_1(\gamma) = O(\gamma^2)$ such that if $\eps$ is sufficiently small depending on $\ha,\delta$, then 
\begin{align}
&\pr\big(\event_1(\beta_1,W)^c\big) \notag\\
&\le \pro{ \|M_1+W\|>2n^\gamma} + \pro{ \event_1(\beta_1,W)^c\wedge \{ \|M_1+W\|\le 2n^\gamma\}}\notag\\
&\ll_{\gamma,\delta,\ha,\eps,\mu_{4+\eta}} \sqrt{\frac{\log n}{n}},		\label{cyc:event1}
\end{align}
where we have used \eqref{nprime.eps} to write $n$ in $n^{-\beta_1}$ rather than $n'$, and the fact that the atom variable is $O(\mu_{4+\eta}^2)$-spread. 

Now let $2\le k\le \me $, and suppose we have found a function $\beta_{k-1}(\gamma)$ such that for any $\gamma>2$ and any fixed $(k-1)n'\times (k-1)n'$ matrix $W$ with $\|W\|\le n^\gamma$, 
\begin{equation}
\pro{ \event_{k-1}(\beta_{k-1}(\gamma),W)^c} \ll_{\gamma,\delta,\ha,\eps,\mu_{4+\eta}} \sqrt{\frac{\log n}{n}}.
\end{equation}
Fix a $kn'\times kn'$ matrix $W$ with $\|W\|\le n^\gamma$.
By Lemma \ref{lem:schur} we have
\begin{align}
s_{kn'}( M_{\le k} + W) 
&\gg \left( 1+ \frac{\|(M+W)_{\le k-1, k}\|}{s_{(k-1)n'}(M_{\le k-1}+W_{\le k-1})}\right)^{-1}
\left( 1+ \frac{\|(M+W)_{k,\le k-1}\|}{s_{(k-1)n'}(M_{\le k-1}+W_{\le k-1})}\right)^{-1}	\notag\\
&\quad\quad\quad\quad \times \min\Big[ s_{(k-1)n'}(M_{\le k-1}+W_{\le k-1}) , s_{n'}\big(M_k + B_k\big)\Big]	\label{Mk:schur}
\end{align}
where we have abbreviated
\begin{equation}
B_k:=W_k -  (M+W)_{k,\le k-1}( M_{\le k-1} + W_{\le k-1})^{-1} (M+W)_{\le k-1,k} .
\end{equation}

Suppose that the event $\event_{k-1}(\beta_{k-1}(\gamma),W_{\le k-1})$ holds.
We condition on a realization of the submatrix $M_{\le k-1}$ satisfying
\begin{equation}
s_{(k-1)n'}(M_{\le k-1}+W_{\le k-1}) \ge n^{-\beta_{k-1}(\gamma)}.
\end{equation}
Moreover, from \eqref{crude.op} we have
\begin{equation}
\|(M+W)_{\le k-1, k}\|, \|(M+W)_{k,\le k-1}\| \le K_0\sqrt{n}+n^2+n^\gamma\le 2n^\gamma
\end{equation}
with probability $1-O(n^{-2})$. Conditioning on the event that the above holds, from the previous two displays we have
$
\|B_k\| \le n^\gamma+ 4n^{\gamma + \beta_{k-1}(\gamma)}.
$
Again by \eqref{crude.op},
\begin{equation}
\|M_k+ B_k\|\le K_0\sqrt{n}+n^2 + 4n^{\gamma+\beta_{k-1}(\gamma)} \le 5n^{\gamma+\beta_{k-1}(\gamma)}
\end{equation}
with probability $1-O(n^{-2})$ in the randomness of $M_k$. 
By Theorem \ref{thm:super} and independence of $M_k$ from $M_{\le k-1},M_{k,\le k-1},M_{k,\le k-1}$, there exists $\beta_k' = O(\gamma^2 + \beta_{k-1}(\gamma)^2)$ such that 
\begin{equation}	\label{MkBk:bound}
\pro{s_{n'}(M_k+B_k) \le n^{-\beta_k'}} \ll_{\gamma,\delta,\ha,\eps,\mu_{4+\eta}} \sqrt{\frac{\log n}{n}}.
\end{equation}
Restricting further to the event that $s_{n'}(M_k+B_k) > n^{-\beta_k'}$ and substituting the above estimates into \eqref{Mk:schur}, we have
\begin{equation}
s_{kn'}(M_{\le k}+W) \gg n^{-2\gamma-2\beta_{k-1}(\gamma)}\min(n^{-\beta_{k-1}(\gamma)}, n^{-\beta_k'}) \ge n^{-\beta_k(\gamma)}
\end{equation}
for some $\beta_k(\gamma) = O(\gamma^2 + \beta_{k-1}(\gamma)^2)$. 
With this choice of $\beta_k(\gamma)$ we have shown
\begin{equation}
\pro{ \event_k(\beta_{k}(\gamma), W_{\le k})^c \wedge \event_{k-1}(\beta_{k-1}(\gamma),W_{\le k-1})}  \ll_{\gamma,\delta,\ha,\eps,\mu_{4+\eta}} \sqrt{\frac{\log n}{n}}.
\end{equation}
Applying this bound for all $2\le k'\le k$ together with \eqref{cyc:event1} and Bayes' rule we conclude that for any fixed $k$ and any square matrix $W$ of dimension at least $kn'$ and operator norm at most $n^\gamma$,
\begin{equation}
\pro{ \event_k(\beta_{k}(\gamma), W_{\le k})^c }  \ll_{\gamma,\delta,\ha,\eps,\mu_{4+\eta}} k \sqrt{\frac{\log n}{n}}.
\end{equation}
The result now follows by taking $k=\me $ and recalling that $\me =O_\eps(1)$.

\appendix

\section{Invertibility for perturbed non-Hermitian band matrices}	\label{app:band}

In this appendix we prove Corollary \ref{cor:band}.

By conditioning on the entries $\xi_{ij}$ with $\min(|i-j|,n-|i-j|)> \eps n$ and absorbing the corresponding entries of $A\circ X$ into $B$ we may assume the entries of $A(\ha)$ are zero outside the band.
By Theorem \ref{thm:broad} it suffices to show that $A(\ha)$ is $(\delta,\nu)$-broadly connected for $\delta,\nu\in (0,1)$ sufficiently small depending on $\eps$. 
Throughout the proof we may assume that $n$ is sufficiently large depending on $\eps$, i.e.\ $n\ge n_0$ for any $n_0(\eps)\in \N$.

Let $\delta,\nu\in (0,1)$ to be chosen sufficiently small depending on $\eps$. 
For all $i\in [n]$ we have $|\mN_{A(\ha)}(i)|,|\mN_{A^\tran(\ha)}(i)|\ge 2\eps n$, so taking $\delta<2\eps$, it only remains to verify the third condition in Definition \ref{def:broad}.
Note that if $|J|>(1-\eps)n$ we trivially have $|J(i)| \ge |\mN_{A(\ha)}(i)| - \eps n \ge \eps n$ for every $i\in [n]$, and the condition holds in this case.

Fix a set $J\subset[n]$ with $1\le |J|\le (1-\eps)n$. For the remainder of the proof we abbreviate $J(i):= J\cap \mN_{A(\ha)}(i)$ and 
\[
I_\delta:= \mN_{A^\tran(\ha)}^{(\delta)}(J) = \{i: |J(i)| \ge \delta |J|\}.
\]

It will be convenient to view $i\mapsto |J(i)|$ as a function on the torus $\Z/n\Z$ (which we identify with $[n]$ in the natural way). From double counting we have
\begin{equation}	\label{ex:doublect}
\sum_{i\in \Z/n\Z} |J(i)| = (1+\lf 2\eps n\rf)|J| \ge 2\eps |J|.
\end{equation}
On the other hand, we have the discrete derivative bound
\begin{equation}	\label{discderiv}
||J(i)|-|J(i-1)||\le 1	\quad \forall i\in \Z/n\Z.
\end{equation}

Suppose towards a contradiction that
\begin{equation}	\label{Ji:suppose}
|I_\delta| <(1+\nu) |J|.
\end{equation}
Since we took $\delta<2\eps$, from \eqref{ex:doublect} and the pigeonhole principle it follows that $|I_\delta|\ge 1$.
We decompose $I_\delta=\cup_{l\in L} I_l$ as a disjoint union of interval subsets $I_l=[a_l,b_l]\subset\Z/n\Z$ that are pairwise separated by a distance at least 2. 
We further split $L=L_>\cup L_\le$, where $L_>= \{l\in L: |I_l|\ge 4\eps n\}$ and $L_\le = L\setminus L_>$.
Note that for each $l\in L$ we have
\begin{equation}	\label{endpoints}
|J(a_l)|=|J(b_l)| = \lf \delta |J|\rf + 1.
\end{equation}
From the bound \eqref{discderiv} and the endpoint conditions \eqref{endpoints} we see that within $I_l$, 
\begin{equation}	\label{Ji.pointwise}
|J(i)|\le \min\big[ \lf \delta |J|\rf + 1+ \min(i-a_l,b_l-i) , \, 2\eps n+1\big],
\end{equation}
where the second argument in the outer minimum comes from the bound $|J(i)| \le \mN_{A(\ha)}(i)\le 2\eps n+1$. 
For $l\in L_\le$ we ignore the second argument in the outer minimum (which only increases the bound), and sum to obtain
\[
\sum_{i\in I_l} |J(i)| \le (\delta|J| +1) |I_l| + \frac14 |I_l|^2 \le (1+\delta |J| + \eps n)|I_l|,\quad l\in L_\le.
\]
For $l\in L_>$ we have
\begin{align*}
\sum_{i\in I_l} |J(i)|  
&=\sum_{i\in I_l: \min(i-a_l,b_l-i)\le 2\eps n} \lf\delta |J|\rf + 1+ \min(i-a_l,b_l-i)\\
&\qquad+ (2\eps n+1) |\{i\in I_l: i-a_l, b_l-i\ge 2\eps n+1\}|\\
&\le 4\eps n (\lf \delta |J|\rf + 1) + 4\eps^2 n^2 + (2\eps n+1) (|I_l| -4\eps n)\\
&\le (2\eps n+1)|I_l| + 4\eps n\delta |J| - 4\eps^2n^2.
\end{align*}
From the previous two displays we obtain
\begin{align*}
\sum_{i\in \Z/n\Z} |J(i)| 
&\le \delta |J| n + \sum_{i\in I_\delta} |J(i)|\\
&\le \delta |J| n + \sum_{l\in L_\le} (1+\delta |J| + \eps n)|I_l| \\
&\qquad \qquad+ \sum_{l\in L_>} \Big[(2\eps n+1)|I_l| + 4\eps n\delta |J| - 4\eps^2 n^2\Big]\\
&= \delta|J|n + 4\eps n(\delta |J|-\eps n)|L_>|\\
&\qquad\qquad+(1+\delta|J|+ \eps n) \sum_{l\in L_\le }|I_l| + (2\eps n+1)\sum_{l\in L_>} |I_l|.
\end{align*}
If $|L_>|=0$ then
\begin{align*}
\sum_{i\in \Z/n\Z} |J(i)| \le \delta |J| n + (1+\delta |J|+\eps n) |I_\delta|.
\end{align*}
Combining with \eqref{ex:doublect} and rearranging we obtain 
\[
|I_\delta| \ge \frac{(2\eps - \delta)|J| n}{1+ \eps n + \delta |J|} \ge \frac{2\eps-\delta}{\eps + \delta}|J|,
\]
and we contradict \eqref{Ji:suppose} taking $\nu<1/2$, say, and $\delta<c\eps$ for a sufficiently small constant $c>0$. 
If $|L_>|\ge 1$, from our assumption $\delta<2\eps$ we have
\begin{align*}
\sum_{i\in \Z/n\Z} |J(i)| 
&\le \delta |J| n -2\eps n|L_>| + (2\eps n+1) \sum_{i\in L} |I_l|\\
&\le \delta |J| n - 2\eps^2 n^2+ (2\eps n+1) |I_\delta|.
\end{align*}
Together with \eqref{ex:doublect} this gives
\[
|I_\delta| \ge \frac{2\eps n}{2\eps n+1} |J| + \frac{2\eps^2n^2-\delta n|J|}{2\eps n+1} \ge \frac{2\eps n}{2\eps n+1} |J| + \frac14\eps n
\]
where in the last bound we took $\delta<\eps^2$ and assumed $n\ge 1/\eps$. 
Taking $\nu<\eps/8$, say, we contradict \eqref{Ji:suppose} if $n$ is sufficiently large.
The claim follows.

\section{Proofs of anti-concentration lemmas}
\label{app:anti}

In this appendix we prove Lemmas \ref{lem:wlog.kappa}, \ref{lem:anti_improved} and \ref{lem:tensorize}.
All three are established by modification of existing arguments from the literature.

\subsection{Proof of Lemma \ref{lem:wlog.kappa}}	\label{app:kappa}




\eqref{cond.kappa1} is immediate by our assumptions. 
It remains to show 
\begin{equation}	\label{kappa.goal}
\e | \re(z\xi - w)|^2\un(|\xi|\le {\kappa_0}) \gg \frac1{\kappa_0} |\re(z)|^2
\end{equation}
for all $z,w\in \C$ after rotating $\xi$ by a phase if necessary.
We may assume $\kappa_0$ is larger than any fixed constant.
Let $\event$ denote the event $\{|\xi|\le \kappa_0\}$. 
By Chebyshev's inequality,
\begin{equation}	\label{kappa.cheb}
\pr(\event) \ge 1-\frac1{\kappa_0^2}.
\end{equation}

Fix $z,w\in \C$.
Write $\widetilde{\e}:= \e(\cdot|\event)$.
By \eqref{kappa.cheb} and assuming $\kappa_0$ is sufficiently large we have that the left hand side of \eqref{kappa.goal} is $\gg \widetilde{\e}| \re(z\xi - w)|^2$, so it suffices to show
\begin{equation}
\widetilde{\e}| \re(z\xi - w)|^2 \gg \frac{1}{\kappa_0}|\re(z)|^2
\end{equation}
after rotating $\xi$ by a phase.
Denoting $\eta:= \xi -\widetilde{\e}\xi$, we have
\[
\widetilde{\e}| \re(z\xi - w)|^2 = \widetilde{\e}| \re(z\eta + (\widetilde{\e} \xi - w))|^2 
= \widetilde{\e}| \re(z\eta)|^2 + |\widetilde{\e} \xi - w|^2 
\]
so it suffices to show that after rotating $\xi$ by a phase,
\begin{equation}
\widetilde{\e}| \re(z\eta)|^2 \gg \frac{1}{\kappa_0} |\re(z)|^2.
\end{equation}

We first estimate the conditional variance of $\eta$. 
We have
\begin{align*}
\widetilde{\e}|\eta|^2 
& = \widetilde{\e}|\xi|^2 - |\widetilde{\e}\xi|^2 \\
&= \frac1{\pr(\event)} \e |\xi|^2 \un_{\event} - \frac1{\pr(\event)^2} |\e \xi \un_{\event}|^2\\
&= \frac1{\pr(\event)^2} \var(\xi\un_{\event}) + \frac1{\pr(\event)}\left(1-\frac1{\pr(\event)}\right) \e |\xi|^2\un_\event\\
&= \frac{1}{\pr(\event)^2} \left( \var(\xi\un_\event) - \pr(\event^c)\e |\xi|^2\un_{\event}\right)\\
&\gg \var(\xi\un_\event) - O(1/\kappa_0^2)
\end{align*}
where in the final line we applied \eqref{kappa.cheb}, the assumption $\e|\xi|^2=1$, and assumed $\kappa_0$ is sufficiently large. 
Now by our assumption that $\xi$ is $\kappa_0$-spread we have $\var(\xi\un_\event) \gg 1/\kappa_0$, so
\begin{equation}
\widetilde{\e}|\eta|^2  \gg 1/\kappa_0
\end{equation}
taking $\kappa_0$ larger if necessary.

Now consider the covariance matrix
\begin{equation}
\Sigma_{\kappa_0} := \begin{pmatrix} \widetilde{\e} |\re(\eta)|^2 & \widetilde{\e} (\re(\eta)\im(\eta))\\
\widetilde{\e} (\re(\eta)\im(\eta)) & \widetilde{\e} |\im(\eta)|^2 
\end{pmatrix}.
\end{equation}
Writing $z=a-ib$ and letting $x=(a\quad b)^\tran$ be the associated column vector, we have
\begin{equation}
\widetilde{\e}| \re(z\eta)|^2 = \widetilde{\e} |a\re(\eta) + b \im(\eta)|^2 = x^\tran \Sigma_{\kappa_0} x.
\end{equation}
Since $\Sigma_{\kappa_0}$ has two non-negative eigenvalues $\sigma^2_1\ge \sigma^2_2\ge 0$ summing to $\widetilde{\e}|\eta|^2  \gg 1/{\kappa_0}$, it follows that $\sigma_1^2 \gg 1/{\kappa_0}$.
We may rotate $\xi$ by an appropriate phase to assume the corresponding eigenspace is spanned by $(1\quad 0)^\tran$. 
This gives 
\[\widetilde{\e}| \re(z\eta)|^2 \gg \sigma_1^2 |\re(z)|^2 \gg \frac1{{\kappa_0}} |\re(z)|^2\] as desired.

\subsection{Proof of Lemma \ref{lem:anti_improved}}	\label{app:improved}



We first need to recall a couple of lemmas from \citep{TaVu:smooth, TaVu:circ}.

\begin{lemma}[Fourier-analytic bound, cf.\ {\cite[Lemma 6.1]{TaVu:smooth}}]	\label{lem:fourier}
Let $\xi$ be a complex-valued random variable.
For all $r> 0$ and any $v\in S^{n-1}$ we have
\begin{equation}
p_{\xi,v}(r) \ll r^2\int_{w\in \C: |w|\le 1/r} \exp\bigg( -c\sum_{j=1}^n \|wv_j\|_\xi^2\bigg) \dd w
\end{equation}
where 
\begin{equation}
\|z\|_\xi^2:= \e \|\re(z(\xi-\xi'))\|_{\R/\Z}^2,
\end{equation}
$\xi'$ is an independent copy of $\xi$, and $\|x\|_{\R/\Z}$ denotes the distance from $x$ to the nearest integer. 
\end{lemma}

The next lemma gives an important property enjoyed by the ``norm" $\|\cdot\|_\xi$ from Lemma \ref{lem:fourier} under the assumption that $\xi$ has $\kappa$-controlled second moment.

\begin{lemma}[cf.\ {\cite[Lemma 5.3]{TaVu:circ}}]	\label{lem:modbound}
For any $\kappa>0$ there are constants $c_1,c_2>0$ such that if $\xi$ is $\kappa$-controlled, then $\|z\|_\xi \ge c_1|\re(z)|$ whenever $|z|\le c_2$.
\end{lemma}

\begin{proof}[Proof of Lemma \ref{lem:anti_improved}]
Let $r\ge0$.
We may assume $r\ge C_0\|v\|_\infty$ for any fixed constant $C_0>0$ depending only on $\kappa$.
From Lemma \ref{lem:fourier}, 
\[
p_{\xi,v}(r) \ll r^2\int_{|w|\le 1/r} \exp\bigg(-c\sum_{j=1}^n \|wv_j\|_\xi^2\bigg)\dd w.
\]
If $C_0$ is sufficiently large depending on $\kappa$, it follows from Lemma \ref{lem:modbound} that whenever $|w|\le 1/r$, $\|wv_j\|_\xi \ge c_1|\re(wv_j)|$, giving
\[
p_{\xi,v}(r) \ll r^2\int_{|w|\le 1/r} \exp\bigg(-c'\sum_{j=1}^n (\re(wv_j))^2\bigg)\dd w
\]
where $c'$ depends only on $\kappa$.
By change of variable,
\begin{equation}	\label{be:cov}
p_{\xi,v}(r) \ll \int_{|w|\le 1} \exp\bigg(-\frac{c'}{r^2}\sum_{j=1}^n (\re(wv_j))^2\bigg)\dd w.
\end{equation}
Write $v_j=r_je^{i\theta_j}$ for each $j\in [n]$.
Since $v\in S^{n-1}$ we have $\sum_{j=1}^n r_j^2=1$. 
By Jensen's inequality,
\begin{align*}
p_{\xi,v}(r) &\ll \int_{|w|\le 1} \exp\bigg( -\frac{c'}{r^2} \sum_{j=1}^n r_j^2\big( \re(we^{i\theta_j})\big)^2\bigg)\dd w \\
&\le \int_{|w|\le 1}\sum_{j=1}^n  r_j^2 \exp\bigg( -\frac{c'}{r^2} \big(\re(we^{i\theta_j})\big)^2\bigg)\dd w.
\end{align*}
By rotational invariance the last expression is equal to
\[
\sum_{j=1}^n  r_j^2\int_{|w|\le 1} \exp\bigg( -\frac{c'}{r^2} (\re(w))^2\bigg)\dd w = \int_{|w|\le 1} \exp\bigg( -\frac{c'}{r^2} (\re(w))^2\bigg)\dd w
\]
which by direct computation is seen to be of size $O(r)$ (with implied constant depending on $\kappa$). 
Together with our assumption that $r\ge C_0\|v\|_\infty$ this gives \eqref{be:1d}.
\end{proof}

\subsection{Proof of Lemma \ref{lem:tensorize}}

We only prove part (a) as part (b) is given in \cite[Lemma 2.2]{RuVe:ilo}.

Let $c_1>0$ to be taken sufficiently small depending on $p_0$, and let $\alpha>0$ a sufficiently small constant to be chosen later.
We have
\begin{align}
\pro{ \sum_{j=1}^n |\zeta_j|^2 \le c_1\eps_0^2 n} 
&= \pro{ n- \frac1{c_1\eps_0^2} \sum_{j=1}^n |\zeta_j|^2 \ge 0}	\notag\\
&\le \e \expo{ c_1\alpha n - \frac{\alpha}{\eps_0^2} \sum_{j=1}^n |\zeta_j|^2 } \notag\\
&= e^{c_1\alpha n} \prod_{j=1}^n \e \expo{ -\alpha |\zeta_j|^2/\eps_0^2}.	\label{tensor:above}
\end{align}
For arbitrary $j\in [n]$ we have
\begin{align*}
\e \expo{ -\alpha |\zeta_j|^2/\eps_0^2}
&= \int_0^1 \pro{ \expo{-\alpha|\zeta_j|^2/\eps_0^2 } \ge u} \dd u\\
&= \int_0^\infty \pro{ |\zeta_j| \le s\eps_0/\sqrt{\alpha}} \dd (e^{-s^2})\\
&\le p_0\int_0^{\sqrt{\alpha}} \dd(e^{-s^2}) + \int_{\sqrt{\alpha}}^\infty \dd(e^{-s^2})\\
&= p_0(1-e^{-\alpha}) + e^{-\alpha}\\
&= 1- (1-p_0)(1-e^{-\alpha}).
\end{align*}
Inserting this in \eqref{tensor:above}, we obtain
\begin{align*}
\pro{ \sum_{j=1}^n |\zeta_j|^2 \le c_1\eps_0^2 n} 
&\le e^{c_1\alpha n} \big[ 1- (1-p_0)(1-e^{-\alpha}) \big]^n\\
&\le \expo{ n\big(c_1\alpha - (1-p_0)(1-e^{-\alpha})\big)}.
\end{align*}
The claim now follows by setting $c_1=(1-p_0)/2$ (for instance) and taking $\alpha$ a sufficiently small constant.

\bibliographystyle{imsart-number}
\bibliography{nonid-ssv}

\begin{thebibliography}{44}

\bibitem{ARS:block}
\begin{barticle}[author]
\bauthor{\bsnm{Aljadeff},~\bfnm{Johnatan}\binits{J.}},
  \bauthor{\bsnm{Renfrew},~\bfnm{David}\binits{D.}} \AND
  \bauthor{\bsnm{Stern},~\bfnm{Merav}\binits{M.}}
(\byear{2015}).
\btitle{Eigenvalues of block structured asymmetric random matrices}.
\bjournal{J. Math. Phys.}
\bvolume{56}
\bpages{103502, 14}.
\bdoi{10.1063/1.4931476}
\bmrnumber{3403052}
\end{barticle}
\endbibitem

\bibitem{AlSh:testing}
\begin{barticle}[author]
\bauthor{\bsnm{Alon},~\bfnm{Noga}\binits{N.}} \AND
  \bauthor{\bsnm{Shapira},~\bfnm{Asaf}\binits{A.}}
(\byear{2004}).
\btitle{Testing subgraphs in directed graphs}.
\bjournal{J. Comput. System Sci.}
\bvolume{69}
\bpages{353--382}.
\bdoi{10.1016/j.jcss.2004.04.008}
\bmrnumber{2087940 (2005e:68083)}
\end{barticle}
\endbibitem

\bibitem{AGZ:book}
\begin{bbook}[author]
\bauthor{\bsnm{Anderson},~\bfnm{Greg~W.}\binits{G.~W.}},
  \bauthor{\bsnm{Guionnet},~\bfnm{Alice}\binits{A.}} \AND
  \bauthor{\bsnm{Zeitouni},~\bfnm{Ofer}\binits{O.}}
(\byear{2010}).
\btitle{An introduction to random matrices}.
\bseries{Cambridge Studies in Advanced Mathematics}
\bvolume{118}.
\bpublisher{Cambridge University Press, Cambridge}.
\bmrnumber{2760897 (2011m:60016)}
\end{bbook}
\endbibitem

\bibitem{BSY:largestsv}
\begin{barticle}[author]
\bauthor{\bsnm{Bai},~\bfnm{Z.~D.}\binits{Z.~D.}},
  \bauthor{\bsnm{Silverstein},~\bfnm{Jack~W.}\binits{J.~W.}} \AND
  \bauthor{\bsnm{Yin},~\bfnm{Y.~Q.}\binits{Y.~Q.}}
(\byear{1988}).
\btitle{A note on the largest eigenvalue of a large-dimensional sample
  covariance matrix}.
\bjournal{J. Multivariate Anal.}
\bvolume{26}
\bpages{166--168}.
\bdoi{10.1016/0047-259X(88)90078-4}
\bmrnumber{963829}
\end{barticle}
\endbibitem

\bibitem{BaHa}
\begin{barticle}[author]
\bauthor{\bsnm{Bandeira},~\bfnm{Afonso~S.}\binits{A.~S.}} \AND
  \bauthor{\bparticle{van} \bsnm{Handel},~\bfnm{Ramon}\binits{R.}}
(\byear{2016}).
\btitle{Sharp nonasymptotic bounds on the norm of random matrices with
  independent entries}.
\bjournal{Ann. Probab.}
\bvolume{44}
\bpages{2479--2506}.
\bdoi{10.1214/15-AOP1025}
\bmrnumber{3531673}
\end{barticle}
\endbibitem

\bibitem{BoCh:survey}
\begin{barticle}[author]
\bauthor{\bsnm{Bordenave},~\bfnm{Charles}\binits{C.}} \AND
  \bauthor{\bsnm{Chafa{\"{\i}}},~\bfnm{Djalil}\binits{D.}}
(\byear{2012}).
\btitle{Around the circular law}.
\bjournal{Probab. Surv.}
\bvolume{9}
\bpages{1--89}.
\bdoi{10.1214/11-PS183}
\bmrnumber{2908617}
\end{barticle}
\endbibitem

\bibitem{BEYY:band}
\begin{bunpublished}[author]
\bauthor{\bsnm{Bourgade},~\bfnm{Paul}\binits{P.}},
  \bauthor{\bsnm{Erdos},~\bfnm{Laszlo}\binits{L.}},
  \bauthor{\bsnm{Yau},~\bfnm{Horng-Tzer}\binits{H.-T.}} \AND
  \bauthor{\bsnm{Yin},~\bfnm{Jun}\binits{J.}}
\btitle{Universality for a class of random band matrices}.
\bnote{Preprint available at arXiv:1602.02312}.
\end{bunpublished}
\endbibitem

\bibitem{BoTz:rit}
\begin{barticle}[author]
\bauthor{\bsnm{Bourgain},~\bfnm{J.}\binits{J.}} \AND
  \bauthor{\bsnm{Tzafriri},~\bfnm{L.}\binits{L.}}
(\byear{1987}).
\btitle{Invertibility of ``large'' submatrices with applications to the
  geometry of {B}anach spaces and harmonic analysis}.
\bjournal{Israel J. Math.}
\bvolume{57}
\bpages{137--224}.
\bdoi{10.1007/BF02772174}
\bmrnumber{890420 (89a:46035)}
\end{barticle}
\endbibitem

\bibitem{Cook:thesis}
\begin{bphdthesis}[author]
\bauthor{\bsnm{Cook},~\bfnm{Nicholas~A.}\binits{N.~A.}}
(\byear{2016}).
\btitle{Spectral properties of non-Hermitian random matrices}
\btype{PhD thesis},
\bpublisher{University of California, Los Angeles}.
\end{bphdthesis}
\endbibitem

\bibitem{CHNR}
\begin{bunpublished}[author]
\bauthor{\bsnm{Cook},~\bfnm{Nicholas~A.}\binits{N.~A.}},
  \bauthor{\bsnm{Hachem},~\bfnm{Walid}\binits{W.}},
  \bauthor{\bsnm{Najim},~\bfnm{Jamal}\binits{J.}} \AND
  \bauthor{\bsnm{Renfrew},~\bfnm{David}\binits{D.}}
\btitle{Limiting spectral distribution for non-Hermitian random matrices with a
  variance profile}.
\bnote{In preparation}.
\end{bunpublished}
\endbibitem

\bibitem{Edelman:condition}
\begin{barticle}[author]
\bauthor{\bsnm{Edelman},~\bfnm{Alan}\binits{A.}}
(\byear{1988}).
\btitle{Eigenvalues and condition numbers of random matrices}.
\bjournal{SIAM J. Matrix Anal. Appl.}
\bvolume{9}
\bpages{543--560}.
\bdoi{10.1137/0609045}
\bmrnumber{964668}
\end{barticle}
\endbibitem

\bibitem{Gowers:towers}
\begin{barticle}[author]
\bauthor{\bsnm{Gowers},~\bfnm{W.~T.}\binits{W.~T.}}
(\byear{1997}).
\btitle{Lower bounds of tower type for {S}zemer{\'e}di's uniformity lemma}.
\bjournal{Geom. Funct. Anal.}
\bvolume{7}
\bpages{322--337}.
\bdoi{10.1007/PL00001621}
\bmrnumber{1445389 (98a:11015)}
\end{barticle}
\endbibitem

\bibitem{HLN:detequiv}
\begin{barticle}[author]
\bauthor{\bsnm{Hachem},~\bfnm{Walid}\binits{W.}},
  \bauthor{\bsnm{Loubaton},~\bfnm{Philippe}\binits{P.}} \AND
  \bauthor{\bsnm{Najim},~\bfnm{Jamal}\binits{J.}}
(\byear{2007}).
\btitle{Deterministic equivalents for certain functionals of large random
  matrices}.
\bjournal{Ann. Appl. Probab.}
\bvolume{17}
\bpages{875--930}.
\bdoi{10.1214/105051606000000925}
\bmrnumber{2326235}
\end{barticle}
\endbibitem

\bibitem{Komlos77}
\begin{bunpublished}[author]
\bauthor{\bsnm{Koml{{\'o}}s},~\bfnm{J{{\'a}}nos}\binits{J.}}
\btitle{Circulated manuscript, 1977}.
\bnote{Edited version available online at:
  http://www.math.rutgers.edu/$\sim$komlos/01short.pdf}.
\end{bunpublished}
\endbibitem

\bibitem{Komlos67}
\begin{barticle}[author]
\bauthor{\bsnm{Koml{{\'o}}s},~\bfnm{J.}\binits{J.}}
(\byear{1967}).
\btitle{On the determinant of {$(0,\,1)$} matrices}.
\bjournal{Studia Sci. Math. Hungar}
\bvolume{2}
\bpages{7--21}.
\bmrnumber{0221962 (36 \#\#5014)}
\end{barticle}
\endbibitem

\bibitem{Komlos68}
\begin{barticle}[author]
\bauthor{\bsnm{Koml{{\'o}}s},~\bfnm{J.}\binits{J.}}
(\byear{1968}).
\btitle{On the determinant of random matrices}.
\bjournal{Studia Sci. Math. Hungar.}
\bvolume{3}
\bpages{387--399}.
\bmrnumber{0238371 (38 \#\#6647)}
\end{barticle}
\endbibitem

\bibitem{KoSi:survey}
\begin{bincollection}[author]
\bauthor{\bsnm{Koml{{\'o}}s},~\bfnm{J.}\binits{J.}} \AND
  \bauthor{\bsnm{Simonovits},~\bfnm{M.}\binits{M.}}
(\byear{1996}).
\btitle{Szemer{\'e}di's regularity lemma and its applications in graph theory}.
In \bbooktitle{Combinatorics, {P}aul {E}rd{\H o}s is eighty, {V}ol.\ 2
  ({K}eszthely, 1993)}.
\bseries{Bolyai Soc. Math. Stud.}
\bvolume{2}
\bpages{295--352}.
\bpublisher{J{\'a}nos Bolyai Math. Soc., Budapest}.
\bmrnumber{1395865 (97d:05172)}
\end{bincollection}
\endbibitem

\bibitem{Latala}
\begin{barticle}[author]
\bauthor{\bsnm{Lata{\l}a},~\bfnm{Rafa{\l}}\binits{R.}}
(\byear{2005}).
\btitle{Some estimates of norms of random matrices}.
\bjournal{Proc. Amer. Math. Soc.}
\bvolume{133}
\bpages{1273--1282 (electronic)}.
\bdoi{10.1090/S0002-9939-04-07800-1}
\bmrnumber{2111932 (2005i:15041)}
\end{barticle}
\endbibitem

\bibitem{LLTTY}
\begin{barticle}[author]
\bauthor{\bsnm{Litvak},~\bfnm{Alexander~E.}\binits{A.~E.}},
  \bauthor{\bsnm{Lytova},~\bfnm{Anna}\binits{A.}},
  \bauthor{\bsnm{Tikhomirov},~\bfnm{Konstantin}\binits{K.}},
  \bauthor{\bsnm{Tomczak-Jaegermann},~\bfnm{Nicole}\binits{N.}} \AND
  \bauthor{\bsnm{Youssef},~\bfnm{Pierre}\binits{P.}}
(\byear{2017}).
\btitle{Adjacency matrices of random digraphs: singularity and
  anti-concentration}.
\bjournal{J. Math. Anal. Appl.}
\bvolume{445}
\bpages{1447--1491}.
\bdoi{10.1016/j.jmaa.2016.08.020}
\bmrnumber{3545253}
\end{barticle}
\endbibitem

\bibitem{LPRT}
\begin{barticle}[author]
\bauthor{\bsnm{Litvak},~\bfnm{A.~E.}\binits{A.~E.}},
  \bauthor{\bsnm{Pajor},~\bfnm{A.}\binits{A.}},
  \bauthor{\bsnm{Rudelson},~\bfnm{M.}\binits{M.}} \AND
  \bauthor{\bsnm{Tomczak-Jaegermann},~\bfnm{N.}\binits{N.}}
(\byear{2005}).
\btitle{Smallest singular value of random matrices and geometry of random
  polytopes}.
\bjournal{Adv. Math.}
\bvolume{195}
\bpages{491--523}.
\bdoi{10.1016/j.aim.2004.08.004}
\bmrnumber{2146352 (2006g:52009)}
\end{barticle}
\endbibitem

\bibitem{LPRTV2}
\begin{barticle}[author]
\bauthor{\bsnm{Litvak},~\bfnm{A.~E.}\binits{A.~E.}},
  \bauthor{\bsnm{Pajor},~\bfnm{A.}\binits{A.}},
  \bauthor{\bsnm{Rudelson},~\bfnm{M.}\binits{M.}},
  \bauthor{\bsnm{Tomczak-Jaegermann},~\bfnm{N.}\binits{N.}} \AND
  \bauthor{\bsnm{Vershynin},~\bfnm{R.}\binits{R.}}
(\byear{2005}).
\btitle{Euclidean embeddings in spaces of finite volume ratio via random
  matrices}.
\bjournal{J. Reine Angew. Math.}
\bvolume{589}
\bpages{1--19}.
\bdoi{10.1515/crll.2005.2005.589.1}
\bmrnumber{2194676}
\end{barticle}
\endbibitem

\bibitem{LiRi}
\begin{barticle}[author]
\bauthor{\bsnm{Litvak},~\bfnm{Alexander~E.}\binits{A.~E.}} \AND
  \bauthor{\bsnm{Rivasplata},~\bfnm{Omar}\binits{O.}}
(\byear{2012}).
\btitle{Smallest singular value of sparse random matrices}.
\bjournal{Studia Math.}
\bvolume{212}
\bpages{195--218}.
\bdoi{10.4064/sm212-3-1}
\bmrnumber{3009072}
\end{barticle}
\endbibitem

\bibitem{MSS:icm}
\begin{binproceedings}[author]
\bauthor{\bsnm{Marcus},~\bfnm{A.~W.}\binits{A.~W.}},
  \bauthor{\bsnm{Spielman},~\bfnm{D.~A.}\binits{D.~A.}} \AND
  \bauthor{\bsnm{Srivastava},~\bfnm{N.}\binits{N.}}
(\byear{2014}).
\btitle{Ramanujan graphs and the solution of the Kadison--Singer problem}.
In \bbooktitle{Proc. ICM, Vol III}
\bpages{375--386}.
\end{binproceedings}
\endbibitem

\bibitem{NgVu:normal}
\begin{bunpublished}[author]
\bauthor{\bsnm{Nguyen},~\bfnm{Hoi~H.}\binits{H.~H.}} \AND
  \bauthor{\bsnm{Vu},~\bfnm{Van~H.}\binits{V.~H.}}
(\byear{2016}).
\btitle{Normal vector of a random hyperplane}.
\bnote{Preprint available at arXiv:1604.04897}.
\end{bunpublished}
\endbibitem

\bibitem{RaAb:neural}
\begin{barticle}[author]
\bauthor{\bsnm{Rajan},~\bfnm{Kanaka}\binits{K.}} \AND
  \bauthor{\bsnm{Abbott},~\bfnm{LF}\binits{L.}}
(\byear{2006}).
\btitle{Eigenvalue spectra of random matrices for neural networks}.
\bjournal{Physical review letters}
\bvolume{97}
\bpages{188104}.
\end{barticle}
\endbibitem

\bibitem{ReTi}
\begin{bunpublished}[author]
\bauthor{\bsnm{Rebrova},~\bfnm{Elizaveta}\binits{E.}} \AND
  \bauthor{\bsnm{Tikhomirov},~\bfnm{Konstantin}\binits{K.}}
\btitle{Covering of random ellipsoids, and invertibility of matrices with
  i.i.d.\ heavy-tailed entries}.
\bnote{Preprint available at arXiv:1508.06690}.
\end{bunpublished}
\endbibitem

\bibitem{Rudelson:inv}
\begin{barticle}[author]
\bauthor{\bsnm{Rudelson},~\bfnm{Mark}\binits{M.}}
(\byear{2008}).
\btitle{Invertibility of random matrices: norm of the inverse}.
\bjournal{Ann. of Math. (2)}
\bvolume{168}
\bpages{575--600}.
\bdoi{10.4007/annals.2008.168.575}
\bmrnumber{2434885}
\end{barticle}
\endbibitem

\bibitem{RuVe:ilo}
\begin{barticle}[author]
\bauthor{\bsnm{Rudelson},~\bfnm{Mark}\binits{M.}} \AND
  \bauthor{\bsnm{Vershynin},~\bfnm{Roman}\binits{R.}}
(\byear{2008}).
\btitle{The {L}ittlewood-{O}fford problem and invertibility of random
  matrices}.
\bjournal{Adv. Math.}
\bvolume{218}
\bpages{600--633}.
\bdoi{10.1016/j.aim.2008.01.010}
\bmrnumber{2407948 (2010g:60048)}
\end{barticle}
\endbibitem

\bibitem{RuVe:uppertail}
\begin{barticle}[author]
\bauthor{\bsnm{Rudelson},~\bfnm{Mark}\binits{M.}} \AND
  \bauthor{\bsnm{Vershynin},~\bfnm{Roman}\binits{R.}}
(\byear{2008}).
\btitle{The least singular value of a random square matrix is {$O(n^{-1/2})$}}.
\bjournal{C. R. Math. Acad. Sci. Paris}
\bvolume{346}
\bpages{893--896}.
\bdoi{10.1016/j.crma.2008.07.009}
\bmrnumber{2441928}
\end{barticle}
\endbibitem

\bibitem{RuZe}
\begin{barticle}[author]
\bauthor{\bsnm{Rudelson},~\bfnm{Mark}\binits{M.}} \AND
  \bauthor{\bsnm{Zeitouni},~\bfnm{Ofer}\binits{O.}}
(\byear{2016}).
\btitle{Singular values of {G}aussian matrices and permanent estimators}.
\bjournal{Random Structures Algorithms}
\bvolume{48}
\bpages{183--212}.
\bdoi{10.1002/rsa.20564}
\bmrnumber{3432577}
\end{barticle}
\endbibitem

\bibitem{SST:smoothed}
\begin{barticle}[author]
\bauthor{\bsnm{Sankar},~\bfnm{Arvind}\binits{A.}},
  \bauthor{\bsnm{Spielman},~\bfnm{Daniel~A.}\binits{D.~A.}} \AND
  \bauthor{\bsnm{Teng},~\bfnm{Shang-Hua}\binits{S.-H.}}
(\byear{2006}).
\btitle{Smoothed analysis of the condition numbers and growth factors of
  matrices}.
\bjournal{SIAM J. Matrix Anal. Appl.}
\bvolume{28}
\bpages{446--476 (electronic)}.
\bdoi{10.1137/S0895479803436202}
\bmrnumber{2255338 (2008b:65060)}
\end{barticle}
\endbibitem

\bibitem{SpSr:rit}
\begin{barticle}[author]
\bauthor{\bsnm{Spielman},~\bfnm{Daniel~A.}\binits{D.~A.}} \AND
  \bauthor{\bsnm{Srivastava},~\bfnm{Nikhil}\binits{N.}}
(\byear{2012}).
\btitle{An elementary proof of the restricted invertibility theorem}.
\bjournal{Israel J. Math.}
\bvolume{190}
\bpages{83--91}.
\bdoi{10.1007/s11856-011-0194-2}
\bmrnumber{2956233}
\end{barticle}
\endbibitem

\bibitem{Szemeredi:lemma}
\begin{bincollection}[author]
\bauthor{\bsnm{Szemer{{\'e}}di},~\bfnm{Endre}\binits{E.}}
(\byear{1978}).
\btitle{Regular partitions of graphs}.
In \bbooktitle{Probl{\`e}mes combinatoires et th{\'e}orie des graphes
  ({C}olloq. {I}nternat. {CNRS}, {U}niv. {O}rsay, {O}rsay, 1976)}.
\bseries{Colloq. Internat. CNRS}
\bvolume{260}
\bpages{399--401}.
\bpublisher{CNRS, Paris}.
\bmrnumber{540024 (81i:05095)}
\end{bincollection}
\endbibitem

\bibitem{Talagrand:newlook}
\begin{barticle}[author]
\bauthor{\bsnm{Talagrand},~\bfnm{Michel}\binits{M.}}
(\byear{1996}).
\btitle{A new look at independence}.
\bjournal{Ann. Probab.}
\bvolume{24}
\bpages{1--34}.
\bdoi{10.1214/aop/1042644705}
\bmrnumber{1387624 (97d:60028)}
\end{barticle}
\endbibitem

\bibitem{TaVu:ssv}
\begin{barticle}[author]
\bauthor{\bsnm{Tao},~\bfnm{Terence}\binits{T.}} \AND
  \bauthor{\bsnm{Vu},~\bfnm{Van}\binits{V.}}
(\byear{2010}).
\btitle{Random matrices: the distribution of the smallest singular values}.
\bjournal{Geom. Funct. Anal.}
\bvolume{20}
\bpages{260--297}.
\bdoi{10.1007/s00039-010-0057-8}
\bmrnumber{2647142}
\end{barticle}
\endbibitem

\bibitem{TaVu:circ}
\begin{barticle}[author]
\bauthor{\bsnm{Tao},~\bfnm{Terence}\binits{T.}} \AND
  \bauthor{\bsnm{Vu},~\bfnm{Van~H.}\binits{V.~H.}}
(\byear{2008}).
\btitle{Random matrices: the circular law}.
\bjournal{Commun. Contemp. Math.}
\bvolume{10}
\bpages{261--307}.
\bdoi{10.1142/S0219199708002788}
\bmrnumber{2409368 (2009d:60091)}
\end{barticle}
\endbibitem

\bibitem{TaVu:cond}
\begin{barticle}[author]
\bauthor{\bsnm{Tao},~\bfnm{Terence}\binits{T.}} \AND
  \bauthor{\bsnm{Vu},~\bfnm{Van~H.}\binits{V.~H.}}
(\byear{2009}).
\btitle{Inverse {L}ittlewood-{O}fford theorems and the condition number of
  random discrete matrices}.
\bjournal{Ann. of Math. (2)}
\bvolume{169}
\bpages{595--632}.
\bdoi{10.4007/annals.2009.169.595}
\bmrnumber{2480613 (2010j:60110)}
\end{barticle}
\endbibitem

\bibitem{TaVu:esd}
\begin{barticle}[author]
\bauthor{\bsnm{Tao},~\bfnm{Terence}\binits{T.}} \AND
  \bauthor{\bsnm{Vu},~\bfnm{Van~H.}\binits{V.~H.}}
(\byear{2010}).
\btitle{Random matrices: universality of {ESD}s and the circular law}.
\bjournal{Ann. Probab.}
\bvolume{38}
\bpages{2023--2065}.
\bnote{With an appendix by Manjunath Krishnapur}.
\bdoi{10.1214/10-AOP534}
\bmrnumber{2722794 (2011e:60017)}
\end{barticle}
\endbibitem

\bibitem{TaVu:smooth}
\begin{barticle}[author]
\bauthor{\bsnm{Tao},~\bfnm{Terence}\binits{T.}} \AND
  \bauthor{\bsnm{Vu},~\bfnm{Van~H.}\binits{V.~H.}}
(\byear{2010}).
\btitle{Smooth analysis of the condition number and the least singular value}.
\bjournal{Math. Comp.}
\bvolume{79}
\bpages{2333--2352}.
\bdoi{10.1090/S0025-5718-2010-02396-8}
\bmrnumber{2684367 (2011k:65065)}
\end{barticle}
\endbibitem

\bibitem{TuVe:rmt_wireless}
\begin{bbook}[author]
\bauthor{\bsnm{Tulino},~\bfnm{Antonia~M}\binits{A.~M.}} \AND
  \bauthor{\bsnm{Verd{\'u}},~\bfnm{Sergio}\binits{S.}}
(\byear{2004}).
\btitle{Random matrix theory and wireless communications}
\bvolume{1}.
\bpublisher{Now Publishers Inc}.
\end{bbook}
\endbibitem

\bibitem{vanHandel:norm}
\begin{bunpublished}[author]
\bauthor{\bparticle{van} \bsnm{Handel},~\bfnm{Ramon}\binits{R.}}
\btitle{On the spectral norm of Gaussian random matrices}.
\bnote{Preprint available at arXiv:1502.05003}.
\end{bunpublished}
\endbibitem

\bibitem{Vershynin:product}
\begin{barticle}[author]
\bauthor{\bsnm{Vershynin},~\bfnm{Roman}\binits{R.}}
(\byear{2011}).
\btitle{Spectral norm of products of random and deterministic matrices}.
\bjournal{Probab. Theory Related Fields}
\bvolume{150}
\bpages{471--509}.
\bdoi{10.1007/s00440-010-0281-z}
\bmrnumber{2824864}
\end{barticle}
\endbibitem

\bibitem{vNG:numerical}
\begin{barticle}[author]
\bauthor{\bparticle{von} \bsnm{Neumann},~\bfnm{John}\binits{J.}} \AND
  \bauthor{\bsnm{Goldstine},~\bfnm{H.~H.}\binits{H.~H.}}
(\byear{1947}).
\btitle{Numerical inverting of matrices of high order}.
\bjournal{Bull. Amer. Math. Soc.}
\bvolume{53}
\bpages{1021--1099}.
\bmrnumber{0024235}
\end{barticle}
\endbibitem

\bibitem{BKY:largestsv}
\begin{barticle}[author]
\bauthor{\bsnm{Yin},~\bfnm{Y.~Q.}\binits{Y.~Q.}},
  \bauthor{\bsnm{Bai},~\bfnm{Z.~D.}\binits{Z.~D.}} \AND
  \bauthor{\bsnm{Krishnaiah},~\bfnm{P.~R.}\binits{P.~R.}}
(\byear{1988}).
\btitle{On the limit of the largest eigenvalue of the large-dimensional sample
  covariance matrix}.
\bjournal{Probab. Theory Related Fields}
\bvolume{78}
\bpages{509--521}.
\bdoi{10.1007/BF00353874}
\bmrnumber{950344}
\end{barticle}
\endbibitem

\end{thebibliography}

\end{document}